\documentclass[11pt,oneside,reqno]{amsart}
\usepackage{amsmath, amssymb, amsthm}
\usepackage{color}
\usepackage{graphicx}\usepackage{subfigure}
\usepackage{amsfonts}
\usepackage[title]{appendix}
\usepackage{latexsym}
\usepackage{psfrag}
\usepackage{multirow}

\usepackage{epstopdf}

\allowdisplaybreaks[4]
\usepackage{paralist}

\usepackage{setspace,kantlipsum}

\def\dfr#1#2{\displaystyle{\frac{#1}{#2}}}
\renewcommand{\vec}[1]{\mbox{\boldmath \small $#1$}}
\newcommand{\smallvec}[1]{\mbox{\boldmath \scriptsize $#1$}}

\newtheorem{theorem}{Theorem}[section]
\newtheorem{lemma}{Lemma}[section]

\newtheorem{remark}{Remark}[section]
\newtheorem{corollary}{Corollary}[section]

\newtheorem{example}{Example}[section]
\numberwithin{equation}{section}
\numberwithin{figure}{section}
\numberwithin{table}{section}

\begin{document}

\begin{spacing}{1.27}

\title[Admissible State and PCP Schemes for RMHD]
{Admissible state and physical constraints preserving schemes for relativistic magnetohydrodynamic equations}%Magnetohydrodynamics
\author{Kailiang Wu}
\address{School of Mathematical Sciences\\
Peking University\\
Beijing 100871, P.R. China}\email{wukl@pku.edu.cn}
\author{Huazhong Tang}
\address{HEDPS, CAPT \& LMAM, School of Mathematical Sciences\\
Peking University\\
Beijing 100871, P.R. China}\email{hztang@math.pku.edu.cn}
\begin{abstract}{\small
This paper first studies  the admissible state set $\mathcal G$ of relativistic magnetohydrodynamics (RMHD). It paves a way for developing physical-constraints-preserving (PCP) schemes for the RMHD equations with the solutions in $\mathcal G$.
To overcome the  difficulties arising from
the extremely strong nonlinearities {and} no explicit formulas of the primitive variables and the flux vectors with respect to the conservative vector,
two equivalent forms of $\mathcal G$ with explicit constraints on the conservative vector
are skillfully discovered.
The first is derived by analyzing  roots of several polynomials and transferring successively them,
and further used
to prove the convexity of $\mathcal G$ with the aid of semi-positive definiteness of the second fundamental form of a hypersurface. While the second is derived based on the convexity, and then used to show
 the orthogonal invariance of  $\mathcal G$.
The Lax-Friedrichs (LxF) splitting property does not hold
generally {for the nonzero magnetic field}, but by a constructive inequality and pivotal techniques,
we discover  the generalized LxF splitting properties, combining
the convex combination of some LxF splitting terms with a discrete divergence-free condition of the magnetic
field.

Based on the above analyses,
several one- and two-dimensional PCP schemes are then studied.
In the 1D case, a first-order accurate LxF type scheme is
first proved to be PCP under the Courant-Friedrichs-Lewy (CFL) condition, and then
the high-order accurate PCP schemes are proposed via a PCP limiter.
In the 2D case,
 the discrete divergence-free condition and  PCP property are analyzed for a
 first-order accurate LxF type scheme, and two sufficient conditions are derived for high-order accurate PCP schemes.
  Our analysis reveals in theory  for the first time that the discrete divergence-free condition is closely connected with the PCP property.
  Several numerical examples demonstrate the theoretical findings and the performance of  numerical schemes.
}
\end{abstract}

%\thanks{arXiv preprint: {\tt 1603.06660}. Submitted for the {\tt Eighteenth IMA Leslie Fox Prize} competition.}

%A challenging issue overcome here is that finite EnKF in huge dimensions introduces unavoidable bias and model errors which need to be controlled and estimated.

\maketitle

\section{Introduction}

The paper is concerned with establishing mathematical properties on the admissible
state set and developing physical-constraints-preserving (PCP) numerical schemes {(which preserve the positivity of the density and pressure, and the bound of the fluid velocity)} for special relativistic magnetohydrodynamics (RMHD).
The $d$-dimensional governing equations of the special RMHDs is a first-order quasilinear hyperbolic system, see e.g. \cite{Friedrichs1974}, and
in the laboratory frame, it can be written in the divergence form
\begin{equation}\label{eq:RMHD1D}
\frac{{\partial \vec U}}{{\partial t}} + \sum\limits_{i = 1}^d {\frac{{\partial {\vec F_i}(\vec U)}}{{\partial x_i}}}  = \vec 0,
\end{equation}
together with the divergence-free condition on the magnetic field
$\vec B=(B_1,B_2,B_3)$, i.e.
\begin{equation}\label{eq:2D:BxBy0}
\sum\limits_{i = 1}^d \frac{\partial B_i } {\partial x_i} =0,
\end{equation}
where $d=1$, or 2 or 3,
and $\vec U$ and $\vec F_i(\vec U)$ denote the conservative vector and the flux in the $x_i$-direction,  respectively, defined by
\begin{align*}
&
\vec U = \big( D,\vec m,\vec B,E \big)^{\top}, \\
&
\vec F_i(\vec U) = \bigg( D v_i,  v_i \vec m  -  B_i \Big( W^{-2} \vec B + (\vec v \cdot \vec B) \vec v \Big)  + p_{tot}  \vec e_i, v_i \vec B - B_i \vec v  ,m_i \bigg)^{\top},\quad i=1,\cdots,d,
\end{align*}
with the mass density $D = \rho W$, the momentum density (row) vector $\vec m = \rho h{W^2}\vec v + |\vec B{|^2} \vec v - (\vec v \cdot \vec B)\vec B$, the energy density $E=\rho h W^2 - p_{tot} +|\vec B|^2$, and
the row vector $\vec e_i$ denoting the $i$-th row of the unit matrix of size 3.
Here $\rho$ is the rest-mass density, the row vector $\vec v=(v_1,v_2,v_3)$ denotes
the fluid velocity,  $p_{tot}$ is the total pressure containing the gas pressure $p$ and  magnetic pressure $p_m:=\frac12 \left(W^{-2} |\vec B|^2 +(\vec v \cdot \vec B)^2 \right)$,
$W=1/\sqrt{1- v^2}$ is the Lorentz factor with $v:=\left(v_1^2+v_2^2+v_3^2\right)^{1/2}$,
 $h$ is the specific enthalpy defined by
$$h = 1 + e + \frac{p}{\rho},$$
with units in which the speed of light $c$ is equal to one, and $e$ is the specific internal energy. It can be seen that the conservative variables
$\vec m$ and $E$ depend on the magnetic field $\vec B$  nonlinearly.

The system \eqref{eq:RMHD1D} takes into account the relativistic description for the dynamics of electrically-conducting
fluid (plasma) at nearly speed of light in vacuum in the presence of magnetic fields. The relativistic magneto-fluid flow appears in investigating numerous astrophysical phenomena from stellar to galactic scales, e.g. core collapse super-novae, coalescing neutron stars, X-ray binaries, active galactic nuclei, formation of black holes, super-luminal jets and gamma-ray bursts, etc.
However, due to relativistic effect, especially the appearance of the Lorentz factor, the system \eqref{eq:RMHD1D}   involves strong
nonlinearity, making its analytic treatment extremely difficult. A primary and powerful approach to improve our understanding of the physical
mechanisms in the RMHDs is through numerical simulations. In comparison with the non-relativistic MHD case, the numerical difficulties are coming from strongly nonlinear coupling between the RMHD equations \eqref{eq:RMHD1D}, which leads to no explicit expression of the primitive {variable vector} $\vec V=(\rho,\vec v,\vec B,p)^{\top} $ and the flux $\vec F_i$ in terms of $\vec U$, and some physical constraints such as $\rho>0$, $p >0$, and $v < c=1$, etc.

Since nearly the 2000s, the numerical study of the RMHDs has attracted
considerable attention, and various modern shock-capturing methods have been developed for the RMHD equations, e.g. the Godunov-type scheme based on  the linear Riemann solver \cite{GodunovRMHD}, the total variation diminishing scheme \cite{Balsara2001},
the third-order accurate central-type scheme based on two-speed approximate Riemann solver \cite{Zanna:2003}, the high-order kinetic flux-splitting method \cite{Qamar2005}, the HLLC (Harten-Lax-van Leer-Contact) type schemes \cite{Honkkila:2007,Kim2014,MignoneHLLCRMHD},
the adaptive methods with mesh refinement \cite{Anderson:2006,Host:2008},
the adaptive moving mesh method \cite{HeTang2012RMHD},
the locally divergence-free Runge-Kutta discontinuous Galerkin (RKDG) method and
exactly divergence-free central RKDG method with the weighted essentially non-oscillatory limiters \cite{ZhaoTang2016},
the ADER (Arbitrary high order schemes using DERivatives) DG method \cite{Zanotti2015}, and the ADER-WENO type schemes with subluminal reconstruction \cite{BalsaraKim2016}, etc. The readers are also referred to the
early review articles \cite{font2008,Marti2015}.

To our best knowledge, up to now, no work  shows in theory that
those existing numerical methods for RMHDs can preserve the positivity of the rest-mass density and the pressure and the bounds of the fluid velocity
at the same time,
although those schemes have been used to simulate some RMHD flows successfully.
There exists a large and long-standing risk of failure when a numerical scheme is applied to the RMHD problems with large Lorentz factor, low density or pressure, or strong discontinuity. This is
because as soon as the negative density or pressure, or the superluminal fluid velocity may be obtained, the eigenvalues of the Jacobian matrix become
imaginary, such that the discrete problem is ill-posed.
It is of great significance to develop
high-order accurate  {physical-constraints}-preserving (PCP) numerical schemes, in the
sense of that the solution of numerical scheme always
belongs to the (physical) {\em admissible state set}
\begin{equation}\label{eq:RMHD:definitionG}
{\mathcal G} := \left\{ \left. \vec U=(D,\vec m,\vec B,E)^{\top}\in \mathbb R^{8}~ \right| ~\rho(\vec U)>0,\ p(\vec U)>0,\ v(\vec U) <c = 1  \right\}.
\end{equation}
Because the functions $\rho(\vec U)$, $p(\vec U)$, and $v(\vec U)$ in \eqref{eq:RMHD:definitionG}, and $\vec F_i(\vec U)$ are highly nonlinear and cannot be expressed explicitly in $\vec U$,
 it is extremely hard to check  whether a given state $\vec U$ is admissible,
 or whether the numerical scheme is PCP.
For this reason, developing the
 PCP schemes for the RMHDs is highly challenging. It is still an unsolved problem.
In fact, it is also a blank in developing the positivity-preserving scheme with strictly and completely theoretical proof for the (non-relativistic) ideal compressible MHD\footnote{
Although several efforts were made to enforce positivity of
the reconstructed or DG polynomial solutions based on the assumption that
the cell average values calculated by the numerical schemes are admissible,
% and archived good performance in proposed numerical tests,
 see e.g. \cite{Balsara2012,cheng,BalsaraKim2016}. However,
 there is no any rigorous proof,  especially in the multi-dimensional case, to genuinely show
 that those schemes can  preserve the positivity.
 In fact, the two-dimensional first-order accurate scheme % (without reconstruction)
 is still possibly not PCP, see  Example 3.1 in this paper.
}.
Studying the intrinsic mathematical properties of the admissible state set $\mathcal G$ may open a window for
such unsolved problem, see e.g.
the recent works \cite{WuTang2015,WuTang2016} on the PCP schemes for the
RHDs.

Besides three physical constraints \eqref{eq:RMHD:definitionG}
on the admissible state $\vec U$, another difficulty for the RMHD system \eqref{eq:RMHD1D} comes from
the divergence-free condition \eqref{eq:2D:BxBy0}.
%If the state  $\vec U$ is given by the numerical scheme, then
%it does nonlinearly depend
%on several initial states, which should be constrained
%by  the discrete version of  \eqref{eq:2D:BxBy0}.
Numerically preserving such condition %continuously or discretely
is very non-trivial (for $d\ge 2$) but important for the robustness of numerical scheme,
and has to be respected.
In physics, numerically incorrect magnetic field topologies may lead to nonphysical plasma transport orthogonal to the magnetic field, see e.g. \cite{Brackbill1980}.
The condition \eqref{eq:2D:BxBy0} is also very crucial for the stability of induction equation
\cite{Yang2016}. The existing numerical experiments in the non-relativistic MHD case have also indicated
that violating the divergence-free condition of magnetic field
 may lead to numerical instability and nonphysical or inadmissible solutions \cite{Brackbill1980,Balsara2004,Rossmanith2006,Balsara2012}.
 Up to now, several numerical treatments have been proposed to reduce such risk, see e.g. \cite{Evans1988,Balsara2004,Li2005,Balsara2009,Li2011} and  references therein.
However, it is still unknown in theory why violating the divergence-free condition of magnetic field
does more easily cause inadmissible solution.

The aim of the paper is to do the first attempt in studying {the properties of} $\mathcal G$ and the
PCP numerical schemes for the special RMHD equations \eqref{eq:RMHD1D}.
The main contributions are outlined as follows:
\begin{enumerate}[\hspace{0em}$\bullet$]
\item[(1)] {\bf Deriving the first equivalent form of $\mathcal G$} by analyzing of polynomial roots and transferring successively. The  constraints in this equivalent form of $\mathcal G$ depend explicitly on the value of $\vec U$
so that the judgment of the admissible state becomes direct and it is useful to
    develop the PCP limiter for the high-order accurate PCP schemes for the RMHDs.
\item[(2)] {\bf Proving the convexity of $\mathcal G$}. The convexity of $\mathcal G$ seems natural from the physical point of view and is critical for studying the PCP schemes. However, its proof is non-trivial and suffers from the difficulty arising from the strongly nonlinear constraints. The key point is to utilize the semi-positive definiteness of the second fundamental form of a hypersurface, which is discovered to have a proper parametric equation.
\item[(3)] {\bf Discovery of the second equivalent form of $\mathcal G$} based on its convexity. This equivalent form of $\mathcal G$ is simple and beautiful with
the  constraints depending  linearly  on  $\vec U$ and plays a pivotal role in verifying  the PCP property of the numerical scheme for the RMHDs. Moreover, it also implies the orthogonal invariance of $\mathcal G$.

\item[(4)] {\bf Establishment of the generalized LxF splitting properties of $\mathcal G$}. An analytic counterexample shows that $\mathcal G$ does not have the LxF splitting property in general.
Luckily, we discover an alternative, the so-called generalized LxF splitting property, which is coupling
the  convex combination of some LxF splitting terms with a ``discrete divergence-free'' condition for the magnetic
field. Since the generalized LxF splitting properties  involve lots of states with strongly coupling condition,
their discovery and  proofs are extremely technical and become  the most highlighted point of this paper.
\item[(5)] {\bf Close connection between the discrete divergence-free condition and PCP property is revealed in theory for the first time}.
Analytic example indicates that first-order accurate LxF type scheme  violating the  divergence-free condition may produce inadmissible solution.
Our theoretical analysis clearly shows the importance of  discrete divergence-free condition in proving the PCP properties of numerical schemes.
\item[(6)] {\bf Theoretical analysis on several 1D and 2D PCP schemes}.
The 1D first-order accurate LxF type scheme is
 proved to be PCP under the CFL condition and the PCP limiter is developed to
propose the 1D high-order accurate PCP schemes.
 The  discrete divergence-free condition and  PCP property are analyzed for the 2D first-order accurate LxF type scheme, and
  two sufficient conditions are derived for the 2D high-order accurate PCP schemes.
 Several numerical examples are given to demonstrate the theoretical analyses and the performance of
  numerical schemes.
\end{enumerate}

The rest of the paper is organized as follows. Section \ref{sec:eqDef} {derives the two} equivalent definitions of $\mathcal G$,
 proves its convexity, and establishes the generalized LxF splitting properties under the ``discrete divergence-free'' condition. They play pivotal roles in analyzing the PCP property of the numerical methods based on the LxF type flux for the RMHD equations \eqref{eq:RMHD1D}, see  Section \ref{sec:app},
where the PCP properties of the 1D and 2D first-order accurate LxF schemes are proved,  the  PCP limiting procedure and
 the high-order accurate PCP schemes for the 1D RMHD equations \eqref{eq:RMHD1D} are presented,
and  sufficient conditions for the 2D  high-order accurate PCP schemes
are also proposed.  Section \ref{sec:examples} conducts several numerical experiments to demonstrate
the theoretical analyses and the performance of the proposed schemes.
Section \ref{sec:con} concludes the paper with several remarks.

\section{Properties of the admissible state set}\label{sec:eqDef}

Throughout the paper, the equation of state (EOS) will be restricted to the $\Gamma$-law
\begin{equation}\label{eq:EOS}
p = (\Gamma-1) \rho e,
\end{equation}
where the adiabatic index $\Gamma \in (1,2]$. The restriction of $\Gamma\le 2$ is required for the compressibility assumptions \cite{Cissoko1992} and the causality in the theory of relativity (the sound speed does not exceed the speed of light {$c=1$}).
All results in this paper can be extended to the
general EOS case by the similar discussion presented in \cite{WuTang2016}.

\begin{lemma}\label{theo:RMHD:condition}
	The admissible state $\vec U=(D,\vec m,\vec B,E)^{\top} \in {\mathcal G}$ must satisfy
	\begin{equation}\label{eq:GconToG2}
	D>0,\quad  q(\vec U):= E-\sqrt{D^2+|\vec m|^2}>0.
	\end{equation}
\end{lemma}

\begin{proof}
	Under three conditions of ${\mathcal G}$ in \eqref{eq:RMHD:definitionG},
	i.e., $\rho(\vec U)>0$, $p(\vec U)>0$, and $v(\vec U) <1$,
	one has
\begin{align*}
	&
	D = \rho W > 0, \\
	&
	E = \rho h{W^2} - p + \frac{{1 + {v^2}}}{2}|\vec B{|^2} - \frac{1}{2}{(\vec v \cdot \vec B)^2}  \ge \rho h{W^2} - p > \rho h -p = \rho+\frac{p}{\Gamma-1}>0,
\end{align*}	
	and
\begin{equation*}
	\begin{split}
	&
	{E^2} - \left( {{D^2} + {{\left| {\vec{m}} \right|}^2}} \right) \\
	&=  {\left[ {\rho h{W ^2} - p + \frac{{1 + {v^2}}}{2}{{\left| {\vec{B}} \right|}^2} - \frac{1}{2}{{({\vec{v}} \cdot {\vec{B}})}^2}} \right]^2} - {(\rho W )^2} - {\left| {(\rho h{W ^2} + {{\left| {\vec{B}} \right|}^2}){\vec{v}} - ({\vec{v}} \cdot {\vec{B}}){\vec{B}}} \right|^2} \\
	&= \left[
	{\left( {\rho h{W ^2} - p} \right)^2} - {(\rho W )^2} - {(\rho h{W ^2} + {\left| {\vec{B}} \right|^2})^2}{v^2}
	\right]
	+ {\left[ {\frac{{1 + {v^2}}}{2}{{\left| {\vec{B}} \right|}^2} - \frac{1}{2}{{({\vec{v}} \cdot {\vec{B}})}^2}} \right]^2}
	\\
	&\quad
	+ (\rho h{W ^2} - p)\left[ {(1 + {v^2}){{\left| {\vec{B}} \right|}^2} - {{({\vec{v}} \cdot {\vec{B}})}^2}} \right]
	- {({\vec{v}} \cdot {\vec{B}})^2}{\left| {\vec{B}} \right|^2} + 2(\rho h{W ^2} + {\left| {\vec{B}} \right|^2}){({\vec{v}} \cdot {\vec{B}})^2}.
	\end{split}
\end{equation*}
	The first term at the right hand side of the above identity
	should be larger than
	$$
	- (2\rho h{W ^2}{\left| {\vec{B}} \right|^2} + {\left| {\vec{B}} \right|^4}){v^2},
	$$
	because of  the inequality
	\begin{equation*}
	\big(\rho h{W^2} - p\big)^2 > |\rho h{W^2}\vec v|^2 + (\rho W)^2,
	\end{equation*}
	which has been proved  in \cite{WuTang2015}.
	Thus, one has
\begin{equation*}
	\begin{split}
	{E^2} - \left( {{D^2} + {{\left| {\vec{m}} \right|}^2}} \right)
	&> 2\rho h{W ^2}\left[ {{{({\vec{v}} \cdot {\vec{B}})}^2} - {{\left| {\vec{B}} \right|}^2}{v^2}} \right] + (\rho h{W ^2} - p)\left[ {(1 + {v^2}){{\left| {\vec{B}} \right|}^2} - {{({\vec{v}} \cdot {\vec{B}})}^2}} \right]
	\\
	&\quad   + {\left| {\vec{B}} \right|^2}{({\vec{v}} \cdot {\vec{B}})^2} + {\left[ {\frac{{1 + {v^2}}}{2}{{\left| {\vec{B}} \right|}^2} - \frac{1}{2}{{({\vec{v}} \cdot {\vec{B}})}^2}} \right]^2} - {\left| {\vec{B}} \right|^4}{v^2}
	\\
	&=  \rho h{W ^2}\left[ {{{\left| {\vec{B}} \right|}^2} - {v^2}{{\left| {\vec{B}} \right|}^2} + {{({\vec{v}} \cdot {\vec{B}})}^2}} \right] - p\left[ {(1 + {v^2}){{\left| {\vec{B}} \right|}^2} - {{({\vec{v}} \cdot {\vec{B}})}^2}} \right]
	\\
	&\quad
	+ {\left| {\vec{B}} \right|^2}\left[ {{{({\vec{v}} \cdot {\vec{B}})}^2} - {{\left| {\vec{B}} \right|}^2}{v^2}} \right] + {\left[ {\frac{{{{\left| {\vec{B}} \right|}^2}}}{2} + \frac{{{v^2}{{\left| {\vec{B}} \right|}^2} - {{({\vec{v}} \cdot {\vec{B}})}^2}}}{2}} \right]^2}
	\\
	&\ge
	(\rho h - p(1 + {v^2})){\left| {\vec{B}} \right|^2} + {\left[ {\frac{{{{\left| {\vec{B}} \right|}^2}}}{2} - \frac{{{v^2}{{\left| {\vec{B}} \right|}^2} - {{({\vec{v}} \cdot {\vec{B}})}^2}}}{2}} \right]^2}
	\\
	& \ge (\rho h - 2p){\left| {\vec{B}} \right|^2} = \left( \rho + \frac{2-\Gamma}{\Gamma-1} p \right){\left| {\vec{B}} \right|^2} \ge 0,
	\end{split}
\end{equation*}
	which along with $E>0$ yield
	$q(\vec U)=E-\sqrt{D^2 + |\vec m|^2}>0$.
	The proof is completed. \end{proof}

If the magnetic field $\vec B$ is zero,  then
 \eqref{eq:GconToG2} is also sufficient for $\vec U\in\mathcal G$, see \cite{WuTang2015},
and  $q(\vec U)$ is   a  concave function in terms of  $\vec U$.
Those results have played  pivotal roles in the analysis and constructions of the PCP schemes for the RHDs \cite{WuTang2015}. Unfortunately,
 \eqref{eq:GconToG2}  is only necessary (not sufficient) for $\vec U\in \mathcal G$ if $\vec B\neq \vec 0$.
In spite of this,   \eqref{eq:GconToG2}  is still important and essential in the coming analysis.

Since there is no explicit expression of {$\rho(\vec U), p(\vec U)$, and $\vec v(\vec U)$ for the RMHDs, the value of $\vec V$ should be derived from given $\vec U$} by solving
some  nonlinear algebraic equation, see e.g. \cite{Balsara2001,Zanna:2003,GodunovRMHD,MignoneHLLCRMHD,Newman,Noble}.
The present paper considers the nonlinear algebraic equation  used in \cite{MignoneHLLCRMHD}
\begin{equation}\label{eq:RMHD:fU(xi)}
f_{ \smallvec U}(\xi ) := \xi  - \frac{{\Gamma  - 1}}{\Gamma }\left( {\frac{\xi }{{{{W}^2}}} - \frac{D}{ W}} \right) + {\left| \vec B \right|^2} - \frac{1}{2}\left[ {\frac{{{{\left| \vec B \right|}^2}}}{{{{W}^2}}} + \frac{{{{(\vec m \cdot \vec B)}^2}}}{{{\xi ^2}}}} \right] - E = 0,
\end{equation}
for the unknown $\xi\in \mathbb R^+$, where
 the Lorentz factor $W$   has been expressed as a function
 of $\xi$  by
\begin{equation}\label{eq:Wxi}
%W(\xi) = {\left( {1 - \frac{{{\xi ^2}|\vec m{|^2} + 2\xi {{(\vec m \cdot %\vec B)}^2} + {{(\vec m \cdot \vec B)}^2}|\vec B{|^2}}}{{{\xi ^2}{{(\xi  + %|\vec B{|^2})}^2}}}} \right)^{ - \frac{1}{2}}},
W(\xi) = \left(  {\xi^{-2}} {{(\xi + {|\vec B|^2})}^{-2}}
f_{\Omega}(\xi)\right)^{ - {1}/{2} },
\end{equation}
with
\begin{equation}\label{eq:Wxi-zzzzz}
f_{\Omega}(\xi):={\xi^2}{{(\xi + {|\vec B|^2})}^2} - \left[ {\xi^2}{|\vec m|^2} + (2\xi+{|\vec B|^2})  {{(\vec m \cdot \vec B)}^2} \right].
\end{equation}
It is reasonable to find the solution of  \eqref{eq:RMHD:fU(xi)} within the
interval
\begin{equation}\label{eq:Wxi-zzzzz000000000000}
\Omega_f:=\mathbb{R}^+ \cap
\left\{ {\left. \xi \right| f_{\Omega}(\xi) >0} \right\},
\end{equation}
 otherwise,
$f_{\Omega}(\xi)\leq 0$ such that  $W(\xi)$ takes the value of  0 or  the imaginary number.
If denote  the solution of the equation \eqref{eq:RMHD:fU(xi)}  by $\xi_*=\xi_*(\vec U)$,
then $\xi_*=\rho(\vec U) h(\vec U)  W^2(\xi_*)=\rho(\vec U) h(\vec U) /\left(1-v^2(\vec U)\right)$,
 and the values of the primitive variables $\rho(\vec U)$, $p(\vec U)$, and $v(\vec U)$ in \eqref{eq:RMHD:definitionG} can be calculated by
 \begin{align}\label{eq:RMHD:getv}
 &
 \vec v(\vec U) = \left( {\vec m + \xi_* ^{ - 1}(\vec m \cdot \vec B) \vec B} \right)/(\xi_*+ {|\vec B|^2}),\\
 &
 \label{eq:RMHD:getrho}
 \rho (\vec U) = \frac{D}{{W(\xi_*)}},\\
 \label{eq:RMHD:getp}
 &
 p(\vec U) = \frac{{\Gamma  - 1}}{{\Gamma W^2(\xi_*)}}\big( {{\xi_*} - D W (\xi_*)} \big).
 \end{align}
The above procedure  clearly shows the strong nonlinearity of the functions
 {$\vec v(\vec U)$, $\rho(\vec U)$}, and $p(\vec U)$,
 and   the  difficulty in verifying whether  $\vec U$ is in the admissible state set
$\mathcal G$.
To overcome such  difficulty, % resulting from {such strong nonlinearities},
  two equivalent definitions of the admissible state set $\mathcal G$
  will be given in the following.
The  first  is very suitable to check whether a given state $\vec U$  is admissible and construct
the PCP limiter {for the development of  high-order accurate PCP schemes for 1D} RMHD equations,
while the second is very effective in verifying the PCP property of a numerical scheme.
Moreover, the convexity  of $\mathcal G$ will also be analyzed.

\subsection{First equivalent definition}\label{sec:first}

This subsection introduces  the first equivalent definition of the  admissible state set $\mathcal G$.

\begin{theorem}[First equivalent definition]\label{theo:RMHD:CYconditionFINAL2}
The admissible state set~${\mathcal G}$ is equivalent to the following set
\begin{align}
{\mathcal G}_0 := \left\{   \vec U=(D,\vec m,\vec B,E)^{\top} \big|  D>0,q(\vec U)>0, {\Psi} (\vec U) > 0 \right\},
\label{eq:RMHD:definitionG2}
\end{align}
where
$$
 {\Psi} (\vec U) := \Big( \Phi(\vec U)-2\big(|\vec B|^2-E\big) \Big) \sqrt{\Phi(\vec U)+|\vec B|^2-E} - \sqrt{ \frac{27}{2} \bigg( D^2|\vec B|^2+(\vec m \cdot \vec B)^2 \bigg)},
$$
with ${\Phi(\vec U):}= \sqrt{({|\vec B|^2} - E)^2 + 3({E^2} - {D^2} - |\vec m|^2)}$.
\end{theorem}

\begin{proof}
The proof of Theorem \ref{theo:RMHD:CYconditionFINAL2} is very technical,
and will be  built on  several lemmas given behind it.

\begin{enumerate}%[\hspace{0em}$\bullet$]
  \item[(1)]  Lemma \ref{theo:RMHD:CYcondition} tells us that three constraints
  $\rho>0$, $p>0$, and $v<1$ in ${\mathcal G}$
 can be equivalently replaced with
  \begin{align}
\label{eq:G:xi(U)}
 \xi_* (\vec U)  >0,\  f_4 ( \xi_* (\vec U)   ) > 0,\ D>0, \ q (\vec U) > 0,
\end{align}
      where the existence and uniqueness of $\xi_* (\vec U)$ have been required, and
      $ f_4(\xi)$ is a
      quartic polynomial
      defined by
\begin{equation}\label{eq:f4:xi}
         f_4(\xi):=  f_\Omega(\xi)  -{D^2}{({\xi} + {|\vec B|^2})^2}.
\end{equation}
For $\xi \in \Omega_f$,  $f_4(\xi)=(W(\xi))^{-2} (\xi^2-D^2(W(\xi))^2)\left(\xi+|\vec B|^2\right)^2$.
The subsequent task is to prove the equivalence between
first two conditions in \eqref{eq:G:xi(U)} and the third one in ${\mathcal G}_0$
under \eqref{eq:GconToG2}.

  \item[(2)]  Lemma \ref{theo:Omegaf} shows that $\Omega_f$ is an open interval and can be equivalently expressed as  $\Omega_f  = \left( \xi_\Omega,+\infty\right)$, where $\xi_\Omega=\xi_\Omega(\vec U)$ denotes the biggest nonnegative root of $f_{\Omega}(\xi)$ in \eqref{eq:Wxi-zzzzz}.
  \item[(3)] Lemma \ref{lemma:RMHD:f4} shows that the polynomial $f_4(\xi)$ %defined in \eqref{eq:f4:xi}
  has {unique} positive root in $\Omega_f$, denoted by $\xi_4$, and
  first two constraints in \eqref{eq:G:xi(U)}
 are equivalently replaced with  $\xi_*> \xi_4$, that is to say,  \eqref{eq:G:xi(U)} is equivalent to
\begin{align}
 \label{eq:RMHD:defGcapG1:2}
\xi_*(\vec U) > \xi_4(\vec U) ,  \ D>0,\ q(\vec U) > 0.
\end{align}
  \item[(4)]
      Lemma \ref{theo:RMHD:fUincrease} states that the function $f_{\smallvec U} (\xi)$ defined in \eqref{eq:RMHD:fU(xi)} is strictly monotone increasing in $\Omega_f$, and $\mathop {\lim }\limits_{\xi \to +\infty } f_{\smallvec U} (\xi) = +\infty$.
Hence the first inequality in
  \eqref{eq:RMHD:defGcapG1:2}
  holds if and only if
\begin{equation*}%\label{eq:RMHD:fU(xi4)neg}
 f_{ \smallvec U}(\xi_4 )=\xi_4  - \frac{D^2|\vec B|^2+(\vec m \cdot \vec B)^2}{2 \xi_4^2} + {\left| \vec B \right|^2} - E<0=f_{\smallvec U} (\xi_*).
\end{equation*}
{Here} we have used that $\xi_4 \in \Omega_f$ and %$\xi_4(\vec U) $ satisfies
$
\xi_4 = D W(\xi_4)$ {for the left equal sign}.

If defining a cubic polynomial ${f_3}(\xi) $  by
\begin{equation}\label{eq:f3:xi}
{f_3}(\xi) := {\xi^3} + \big({|\vec B|^2} - E\big){\xi^2} - \frac{{{|\vec B|^2}{D^2} + {{(\vec m \cdot \vec B)}^2}}}{2},
\end{equation}
then ${f_3}(\xi_4)=\xi_4^2f_{ \smallvec U}(\xi_4 )$ and
\eqref{eq:RMHD:defGcapG1:2} is equivalent to
\begin{align}
 \label{eq:RMHD:defGcapG1:3}
f_3(\xi_4( \vec U) )<0 ,\ D>0,\ q(\vec U) > 0. % \right\}.
\end{align}

\item[(5)]
Let us  reduce the degree of polynomial in the constraints by transferring  successively
the lower order constraint on the root of a high degree polynomial into the higher order constraint on the root of
a low degree polynomial.
Lemma \ref{lemma:RMHD:f3} shows that the polynomial $f_3(\xi)$ has {unique} positive root, denoted by $\xi_3$. The continuity of $f_3(\xi)$
implies that for any $\xi>0$, one has
    \begin{equation}\label{eq:RMHD:f3xi3}
f_3(\xi)<0 ~ \Leftrightarrow ~ \xi <\xi_3,\qquad \mbox{or} \ f_3(\xi)>0 ~ \Leftrightarrow ~\xi >\xi_3.
\end{equation}
Thus the first inequality in \eqref{eq:RMHD:defGcapG1:3} is equivalent to
\begin{equation}\label{eq:xi4<xi3}
\xi_4( \vec U)  < \xi_3( \vec U).
\end{equation}
Lemma \ref{lemma:RMHD:f4}  yields
$$ f_4(\xi)>0 ~ \Leftrightarrow ~ \xi_4 < \xi,$$
for any $\xi>0$.
Therefore, \eqref{eq:xi4<xi3} is equivalent to
\begin{equation}\label{eq:RMHD:f4(xi3)posi}
f_4(\xi_3( \vec U) )>0.
\end{equation}
If defining a quadratic polynomial $f_2(\xi)$ by
\begin{equation}\label{eq:f2:xi}
f_2(\xi) := 3\xi^2+4 \big(|\vec B|^2 -E \big) \xi +|\vec B|^4+D^2+|\vec m|^2 -2 |\vec B|^2 E,
\end{equation}
then one gets
\begin{equation*}
\begin{split}
f_4(\xi_3 ) &= \xi_3^2(\xi_3 + |\vec B|^2)^2 - \left[ D^2 (\xi_3 + |\vec B|^2)^2 + \xi_3^2 |\vec m|^2 + \big(2\xi_3+ |\vec B|^2 \big)(\vec m \cdot \vec B)^2 \right]\\
& = \xi_3^2(\xi_3 + |\vec B|^2)^2 - \xi_3^2 (D^2 + |\vec m|^2) -\left[ D^2 |\vec B|^2 + (\vec m \cdot \vec B)^2 \right] \big(2\xi_3+ |\vec B|^2 \big)
\\
&= \xi_3^2(\xi_3 + |\vec B|^2)^2 - \xi_3^2 (D^2 + |\vec m|^2) - 2\big(\xi_3^3+(|\vec B|^2-E)\xi_3^2 \big) \big(2\xi_3+ |\vec B|^2 \big)
\\
&=-\xi_3^2 f_2(\xi_3).
\end{split}
\end{equation*}
Here the identity $f_3(\xi_3)=0$ has been used in the third equal sign.
Hence, \eqref{eq:RMHD:f4(xi3)posi} becomes
\begin{equation}\label{eq:RMHD:f2(xi3)neg}
f_2(\xi_3( \vec U) )<0.
\end{equation}
\item[(6)] Lemma \ref{lemma:RMHD:f2} tells us that the polynomial $f_2(\xi)$ has two real roots,
denoted by $\xi_{2,L}$ and $\xi_{2,R}$ with $\xi_{2,L}<\xi_{2,R}$.
Because the graph of $f_2(\xi)$ opens upward,
 \eqref{eq:RMHD:f2(xi3)neg} is equivalent to
\begin{equation}\label{eq:RMHD:xi2xi3}
\xi_{2,L}( \vec U) < \xi_3( \vec U)< \xi_{2,R}( \vec U),
\end{equation}
which implies
\begin{equation}\label{eq:RMHD:f3xi2}
  \xi_{2,R}( \vec U) >0, \quad f_3(  \xi_{2,R}( \vec U) )>0,
\end{equation}
because of \eqref{eq:RMHD:f3xi3} and $\xi_3>0$.
Conversely, one can show that  \eqref{eq:RMHD:f3xi2} also implies \eqref{eq:RMHD:xi2xi3}, thus they are equivalent to each other. In fact, if
\eqref{eq:RMHD:f3xi2} holds, one has $\xi_3 < \xi_{2,R}$ by using \eqref{eq:RMHD:f3xi3}. Assume that \eqref{eq:RMHD:f3xi2} holds but
\eqref{eq:RMHD:xi2xi3} do not holds, then $\xi_{2,R}>\xi_{2,L} \ge \xi_3$.
By using {\em Vieta's formula} for the quadratic polynomial
that relate the coefficients of a polynomial to sums and products of its roots, $\xi_{2,M}:=\frac12 (\xi_{2,L}+\xi_{2,R}) = - \frac{2}{3} (|\vec B|^2-E) > \xi_3>0$.
Due to \eqref{eq:RMHD:f3xi3}, one has
\begin{equation}\label{eq:f3xi2M}
f_3(  \xi_{2,M} )>0.
\end{equation}
On the other hand, because
$$
f'_3(\xi) = 3 \xi^2 + 2\big(|\vec B|^2 - E \big) \xi = 3 \xi (\xi - \xi_{2,M}),
$$
the function $f_3(\xi)$ is strictly monotone decreasing in
the interval $(0, \xi_{2,M})$, and thus
$$f_3(  \xi_{2,M}) < f_3(0)
=  - \frac{{{|\vec B|^2}{D^2} + {{(\vec m \cdot \vec B)}^2}}}{2}  \le 0,$$
which leads to a contradiction with \eqref{eq:f3xi2M}.
Therefore  \eqref{eq:RMHD:xi2xi3} is equivalent to
 \eqref{eq:RMHD:f3xi2} under \eqref{eq:GconToG2}.
\end{enumerate}

Because
$$
\xi_{2,R}  =  \frac{ \Phi(\vec U) - 2 \big(|\vec B|^2-E \big) }{3} ,
$$
two inequalities in \eqref{eq:RMHD:f3xi2}
 become
\begin{equation}\label{eq:RMHD:exConstraint}
\begin{aligned}%
&\Phi(\vec U) - 2 \big(|\vec B|^2-E \big) > 0,\\
&\big( \Phi(\vec U) -2(|\vec B|^2-E) \big)^2 \big(\Phi(\vec U)+(|\vec B|^2-E)\big) >  \frac{27}{2} \big( D^2|\vec B|^2+(\vec m \cdot \vec B)^2 \big),
\end{aligned}
\end{equation}
which are equivalent to ${\Psi}(\vec U)>0$ by noting that
$$\Phi(\vec U)+(|\vec B|^2-E)> \big||\vec B|^2-E\big| +(|\vec B|^2-E) \ge 0,$$
under $q(\vec U)>0$.
The proof is completed.
\end{proof}

The rest of this subsection gives all lemmas used in the proof of Theorem
\ref{theo:RMHD:CYconditionFINAL2} and two remarks on Theorem
\ref{theo:RMHD:CYconditionFINAL2} as well as a corollary.

\begin{lemma}\label{theo:RMHD:CYcondition}
$\vec U=(D,\vec m,\vec B,E)^{\top} \in {\mathcal G}$ if and only if
$f_{\smallvec U}(\xi)$ has unique {zero $\xi_*(\vec U)$ in $ \Omega_f$ and satisfies}
\begin{equation}\label{eq:fourC}
D>0,~q(\vec U) > 0,~ \xi_*(\vec U) >0,~  f_4 ( \xi_*(\vec U)  ) > 0,
\end{equation}
{where
 $f_4(\xi)$} is a
quartic
polynomial defined in \eqref{eq:f4:xi}.
\end{lemma}
\begin{proof}
(i). Assume $\vec U \in {\mathcal G}$.
Lemma \ref{theo:RMHD:condition} shows that the first two inequalities in \eqref{eq:fourC} hold.
Because $\rho(\vec U)>0,p(\vec U)>0$, and $v(\vec U)<1$, one has
    $$\xi_*  =\rho hW^2 = \frac{\rho(\vec U) h (\vec U)}{1-v^2(\vec U)} = \frac{\rho(\vec U) + \frac{ \Gamma }{\Gamma -1} p (\vec U)}{1-v^2(\vec U)} >0. $$
On the other hand,  because of \eqref{eq:RMHD:getp},
and the facts that $\Gamma>1$ and  $v<1$, one has
$\xi_* > D W(\xi_*) $,
which implies $f_4 ( \xi_*  ) > 0$.

(ii). Assume that four inequalities in \eqref{eq:fourC} hold.
Because of \eqref{eq:f4:xi} and  $D>0, \xi_*>0 $,
one has
$$ f_\Omega ( \xi_*  )  >f_\Omega(\xi_*)  -{D^2}{({\xi_*} + {|\vec B|^2})^2} = f_4 ( \xi_*  )>0 ,$$
which implies
$$
W^{-2}=1 - v^2(\vec U) = \frac{  f_\Omega ( \xi_* ) } {\xi^2_*(\xi_* + |\vec B|^2)^2} > 0.
$$
Thus $v(\vec U)<1$ and $W(\xi_*)\ge 1$.
Thanks to \eqref{eq:RMHD:getrho} and $D>0$, one has $\rho(\vec U) = D/W(\xi_*) > 0$.
Using \eqref{eq:RMHD:getp} and $\Gamma>1$ gives
\begin{equation*}
\begin{split}
p(\vec U) &= \frac{{\Gamma  - 1}}{{\Gamma W(\xi_*)}}\left( {\frac{\xi_*}{W (\xi_*)} - D } \right)
=  \frac{{\Gamma  - 1}}{{\Gamma W(\xi_*)}}\left( {\frac{\xi_*}{W (\xi_*)} + D } \right)^{-1} \Big( \xi^2_* W^{-2} (\xi_*) - D^2 \Big)\\
& =
\frac{{\Gamma  - 1}}{{\Gamma W(\xi_*)}}\left( {\frac{\xi_*}{W (\xi_*)} + D } \right)^{-1} \frac{ f_4 ( \xi_*  ) } { (\xi_* + |\vec B|^2)^2 } >0.
\end{split}
\end{equation*}
The proof is completed.
\end{proof}

\begin{lemma}\label{theo:Omegaf}
For any $\vec U=(D,\vec m,\vec B,E)^{\top}\in   \mathbb R^{8}$,
the set $\Omega_f$  in \eqref{eq:Wxi-zzzzz000000000000} is an open interval and can be expressed as
\begin{equation}\label{eq:RMHD:Omegafnew}
\Omega_f  = %\left\{ \xi | \xi > \xi_\Omega(\vec U)  \right\},
( \xi_\Omega,+\infty ),
\end{equation}
where $\xi_\Omega=\xi_\Omega(\vec U)$ is the biggest nonnegative root of $f_{\Omega}(\xi)$.
\end{lemma}

\begin{proof}
If $\vec m \cdot \vec B =0$, then \eqref{eq:Wxi-zzzzz} gives
$f_{\Omega}(\xi)=\xi^2(\xi +|\vec B|^2+|\vec m|)(\xi +|\vec B|^2-|\vec m|)$, whose biggest nonnegative root is $\xi_\Omega =\max \{0,|\vec m|-|\vec B|^2\}$. {Thus \eqref{eq:RMHD:Omegafnew} holds.}

Assume that $\vec m \cdot \vec B \neq 0$ and $\xi>0$.
In this case, $|\vec B| \neq 0$ such that
the expression of $f_{\Omega}(\xi)$ in \eqref{eq:Wxi-zzzzz} is
reformulated as follows
\begin{equation}\label{eq:RMHD:fOmega}
\begin{split}
f_{\Omega}(\xi) &=
{\xi^2}{{(\xi + {|\vec B|^2})}^2} - \left[ {\xi^2}{|\vec m|^2} + \frac{(\vec m \cdot \vec B)^2}{|\vec B|^2} \left( 2 \xi |\vec B|^2 + |\vec B|^4\right) \right]
\\
&=
{\xi^2}{{(\xi + {|\vec B|^2})}^2} -   {\xi^2}\left( |\vec m|^2 -\frac{(\vec m \cdot \vec B)^2}{|\vec B|^2} \right)
- \frac{(\vec m \cdot \vec B)^2}{|\vec B|^2} \left(  \xi + |\vec B|^2 \right)^2.
\end{split}
\end{equation}
Define
\begin{equation}\label{eq:gOmega}
g_{\Omega}(\xi) : =
\left(  1 - \frac{(\vec m \cdot \vec B)^2}{\xi^2|\vec B|^2} \right)  (\xi + {|\vec B|^2})^2 - \left( |\vec m|^2 -\frac{(\vec m \cdot \vec B)^2}{|\vec B|^2} \right),
\end{equation}
which implies $f_{\Omega}(\xi)=\xi^2 g_{\Omega}(\xi)$
and
$g_{\Omega}(\xi)\le 0$ for $0 < \xi\le
{\left|\vec m \cdot \vec B \right|}/{\left|\vec B\right|}=:\zeta_0$.
It is also easy to verify that  $g_{\Omega}(\xi)$ satisfies
$$
g_{\Omega}\left( \zeta_0 \right) \le 0,\quad
\mathop {\lim }\limits_{\xi \to +\infty } g_{\Omega}(\xi) = + \infty,
$$
and   is also strictly monotone increasing
in the interval $\left[ \zeta_0, +\infty \right)$,
because the first term at the right hand side of \eqref{eq:gOmega}
is a product of two nonnegative and strictly monotone
increasing functions in $\left[ \zeta_0, +\infty \right)$.
The {\em intermediate value theorem} shows that
$g_{\Omega}(\xi)$ has {unique} positive root $\xi_\Omega(\vec U)$ in  $\left[\zeta_0, +\infty \right)$,
 which is the biggest positive root
of {$f_{\Omega}(\xi)$
because} of the relationship
$f_{\Omega}  (\xi) = \xi^2 g_{\Omega}(\xi)$.
Therefore, the domain $\Omega_f$
can be equivalently replaced with \eqref{eq:RMHD:Omegafnew}.
The proof is completed.
\end{proof}

\begin{lemma}\label{lemma:RMHD:f4}
	If  $\vec U=(D,\vec m,\vec B,E)^{\top}\in \mathbb R^{8}$ with $D>0$,
	then the quartic polynomial $f_4(\xi)$ defined in \eqref{eq:f4:xi} has {unique} positive root $\xi_4$, satisfying $\xi_4 >\xi_\Omega $.
	Moreover, $f_4(\xi)>0$ is equivalent to $\xi_4 < \xi$ for any $\xi \in \mathbb{R}^+$.
\end{lemma}
\begin{proof}
	If $|\vec B| =0$, then  $f_4(\xi)=\xi^2(\xi^2-(D^2+|\vec m|^2))$ has {unique} positive root
	$\xi_4(\vec U)=\sqrt{D^2+|\vec m|^2}$, which satisfies $\xi_4(\vec U)> |\vec m|=\xi_\Omega$.
	If $|\vec B| \neq 0$, then
	$f_4(\xi)$  is rewritten as follows
	\begin{align*}
	f_4(\xi)
	= {\xi^2} g_4 (\xi), \quad \xi>0,
	\end{align*}
	where the rational polynomial
	\begin{equation}\label{eq:g4:xi}
	g_4 (\xi) :=
	\left(  1 - {\xi^{-2}} \xi_0^2
	\right)  (\xi + {|\vec B|^2})^2 - \left( |\vec m|^2 -\frac{(\vec m \cdot \vec B)^2}{|\vec B|^2} \right),
	\end{equation}
	with
	$$
	\xi_0:=\sqrt{ D^2 + {\left(\vec m \cdot \vec B \right)^2}/{\left|\vec B\right|^2} }.
	$$
	%One has $f_4(\xi)$ and $ g_4 (\xi)$ share the same possible, positive zeros.
	Obviously, if $\xi \in \left(0,\xi_0\right)$, then one has
	$$
	g_4 (\xi) < - \left( |\vec m|^2 -\frac{(\vec m \cdot \vec B)^2}{|\vec B|^2} \right) \le 0.
	$$
	Thus,  the positive zero of $g_4 (\xi)$ may  be in the interval $\left[ \xi_0, +\infty \right)$.
	The existence of  the positive zero of $g_4 (\xi)$ is verified as follows.
	It is easy to get that
	$$
	g_4 \left( \xi_0 \right) \le 0,\quad
	\mathop {\lim }\limits_{\xi \to +\infty } g_4(\xi) = + \infty.
	$$
	On the other hand, the function $g_4(\xi)$ is strictly monotone increasing
	in the interval $\left[ \xi_0, +\infty \right)$, because
	the first term at the right hand side of \eqref{eq:g4:xi} is a product of two positive and strictly monotone increasing functions in
	$\left[ \xi_0, +\infty \right)$.
	The {\em intermediate value theorem} shows
	that $g_4(\xi)$ has {unique} positive root in
	$\left[ \xi_0, +\infty \right)$, equivalently,
	$f_4(\xi)$ has {unique} positive root $\xi_4$. It satisfies
	$$
	f_\Omega (\xi_4 ) =  f_4(\xi_4) + D^2 (\xi_4 + |\vec B|^2)^2 = D^2 (\xi_4 + |\vec B|^2)^2 >0,
	$$
	which implies $\xi_4 \in \Omega_f$.
	Using Lemma \ref{theo:Omegaf} completes the proof.
	\end{proof}

\begin{lemma}\label{theo:RMHD:fUincrease}
For any $\vec U=(D,\vec m,\vec B,E)^{\top}\in \mathbb R^{8}$ with $D>0$, the function $f_{ \smallvec U}(\xi ) $ defined in \eqref{eq:RMHD:fU(xi)} is strictly monotone increasing
in the interval $\Omega_f  = \left( \xi_\Omega,+\infty  \right)$, and $\mathop {\lim }\limits_{\xi \to +\infty } f_{\smallvec U} (\xi) = +\infty$.
\end{lemma}

\begin{proof}
From \eqref{eq:RMHD:fU(xi)} and \eqref{eq:Wxi},  the derivatives of $f_{ \smallvec U}(\xi ) $ and $W(\xi)$ with respect to $\xi $ is calculated as follows
\begin{equation}\label{eq:dfUdxi}
%\frac{\Gamma }{{\Gamma  - 1}}
{f'_{\smallvec U}}(\xi ) =      %\frac{\Gamma }{{\Gamma  - 1}}
\varXi_\xi
- \frac{{\Gamma  - 1}}{\Gamma }\left( {\frac{1}{{{W^2}}} - \frac{{2\xi }}{{{W^3}}}W'(\xi ) + \frac{D}{{{W^2}}}W'(\xi )} \right),
\end{equation}
and
\[
W'(\xi ) =  - {W^3}\frac{{{{(\vec m \cdot \vec B)}^2}(3{\xi^2} + 3\xi{|\vec B|^2} + {|\vec B|^4}) + {|\vec m|^2}{\xi^3}}}{{{\xi^3}{{(\xi + {|\vec B|^2})}^3}}},
\]
where
$$
\varXi_\xi :=
{1 + \frac{{{{\left|\vec  B \right|}^2}}}{{{W^3}}}W'(\xi )
+ \frac{{{{(\vec m \cdot \vec B)}^2}}}{{{\xi^3}}}}.
$$

Let us prove that $\varXi_\xi >0$ for any $\vec B \in {\mathbb {R}}^3$ and $\xi \in \Omega_f$.
If $\vec B = \vec 0$, then  $\varXi_\xi =1>0$.
Assume that  $\vec B \neq \vec0$ and thus \eqref{eq:RMHD:fOmega} is available.
Using \eqref{eq:RMHD:fOmega} and $f_\Omega(\xi)>0$ gives
 $|\vec B|^2 \xi^2 - ( \vec m \cdot \vec B)^2 >0 $ and
$$
  (\xi + {|\vec B|^2})^2 > \frac{ {\xi^2}\left( |\vec m|^2 |\vec B|^2 -(\vec m \cdot \vec B)^2 \right) }
  { \left|\vec B \right|^2 \xi^2 - ( \vec m \cdot \vec B)^2 } .
$$
 It follows that
\begin{equation*}
\begin{split}
\varXi_\xi &
= \frac{ (\xi + |\vec B|^2)^3 - \big( |\vec m|^2 |\vec B|^2 - (\vec m \cdot \vec B)^2  \big) } { (\xi + |\vec B|^2)^3 }
\\
&
>  \frac{ (\xi + |\vec B|^2) \frac{ {\xi^2}\left( \left|\vec m \right|^2 \left|\vec B\right|^2 -(\vec m \cdot \vec B)^2 \right) }
  { \left|\vec B \right|^2 \xi^2 - ( \vec m \cdot \vec B)^2 } - \big( |\vec m|^2 |\vec B|^2 - (\vec m \cdot \vec B)^2  \big) } { (\xi + |\vec B|^2)^3 }
\\
&
= \frac{ (\xi^3+(\vec m \cdot \vec B)^2) \big(  \left|\vec m \right|^2 \left|\vec B\right|^2 -(\vec m \cdot \vec B)^2 \big) }
       { (\xi + |\vec B|^2)^3 \big( \left|\vec B \right|^2 \xi^2 - ( \vec m \cdot \vec B)^2 \big)} \ge 0.
\end{split}
\end{equation*}

Because  $\Gamma \in (1,2]$,  $\frac{\Gamma}{\Gamma-1} \ge 2$.
Noting that $W'(\xi )\leq 0$ for $\xi\in \Omega_f$
and  using \eqref{eq:dfUdxi} {give}
\begin{equation*}
\begin{split}
\frac{\Gamma }{{\Gamma  - 1}}{f'_{\smallvec U}}(\xi )
&
 \ge 2 \varXi_\xi - \left( {\frac{1}{{{W^2}}} - \frac{{2\xi }}{{{W^3}}}W'(\xi ) + \frac{D}{{{W^2}}}W'(\xi )} \right)
 \ge 2\varXi_\xi
 - \left( \frac{1}{{{W^2}}} - \frac{{2\xi }}{{{W^3}}}W'(\xi )  \right)
 \\
 &
=2\left[ 1+ \frac{(\vec m \cdot \vec B)^2}{\xi^3} - \frac{{{{(\vec m \cdot \vec B)}^2}(3{\xi^2} + 3\xi{|\vec B|^2} + {|\vec B|^4}) + {|\vec m|^2}{\xi^3}}}{{{\xi^3}{{(\xi + {|\vec B|^2})}^2}}}    \right] -\frac{1}{W^2}
\\
&
= \frac{2}{W^2} -\frac{1}{W^2} = \frac{f_\Omega(\xi)}{\xi^2(\xi+|\vec B|^2)^2} > 0,
\end{split}
\end{equation*}
which implies ${f'_{\smallvec U}}(\xi ) >0$ and ${f_{\smallvec U}}(\xi ) $ is strictly monotone increasing
in the interval $\Omega_f$. Note that
\begin{align*}
{f_{\smallvec U}}(\xi )
> \left(1- \frac{\Gamma-1}{\Gamma W^2} \right) \xi  - \frac{1}{2}\left[ {\frac{{{{\left| \vec B \right|}^2}}}{{{{W}^2}}} + \frac{{{{(\vec m \cdot \vec B)}^2}}}{{{\xi ^2}}}} \right] - E \to + \infty, \quad \mbox{as } \xi \to +\infty,
\end{align*}
where $\mathop {\lim }\limits_{\xi \to +\infty } W(\xi) = 1$ has been used. This implies $\mathop {\lim }\limits_{\xi \to +\infty } f_{\smallvec U} (\xi) = + \infty$ and the proof is  completed.
\end{proof}

\begin{lemma}\label{lemma:RMHD:f3}
If $\vec U=(D,\vec m,\vec B,E)^{\top}\in \mathbb R^{8}$ satisfying
\eqref{eq:GconToG2},
 then
the cubic polynomial $f_3(\xi)$ defined in \eqref{eq:f3:xi} has {unique} positive root $\xi_3 $. % in the positive half axis of $\xi$.
\end{lemma}

\begin{proof}
If~$\vec B = \vec 0$, then  $f_3(\xi)=\xi^2(\xi-E)$ {with
unique positive root} $\xi_3  =E$.
If $\vec B \neq \vec 0$, then
 $f_3(\xi)$ is rewritten as follows
$$
f_3(\xi)=\xi^2 g_3 (\xi ), \quad \xi>0,
$$
with the rational polynomial
$$
g_3 (\xi ) := \xi - \frac{D^2|\vec B|^2 +(\vec m \cdot \vec B)^2}{2 \xi^2} +|\vec B|^2-E,
$$
which is  strictly {monotone increasing in ${\mathbb{R}}^+$} and satisfies
$$
\mathop {\lim }\limits_{\xi \to 0^+ } g_3(\xi) = - \infty,\quad
\mathop {\lim }\limits_{\xi \to +\infty } g_3(\xi) = + \infty.
$$
According to the {\em intermediate value theorem},  $g_3 (\xi )$ has unique positive root, and thus $f_3(\xi)$ has unique positive root in {${\mathbb{R}}^+$}. The proof is  completed.
\end{proof}

\begin{lemma}\label{lemma:RMHD:f2}
If $q(\vec U)>0$, then the quadratic polynomial $f_2(\xi)$ defined in \eqref{eq:f2:xi}  has two real roots. \end{lemma}

\begin{proof}
Because the discriminant of the quadratic polynomial $f_2(\xi)$ is
\begin{equation*}
\begin{split}
\Delta &= 16 \big(|\vec B|^2 -E \big)^2 - 12 \big(|\vec B|^4+D^2+|\vec m|^2 -2 |\vec B|^2 E \big)
\\
&= 4 (|\vec B|^2-E)^2 + 12( E^2-(D^2+|\vec m|^2))   \\
& \ge  12q(\vec U) \left( E+\sqrt{D^2+|\vec m|^2} \right)
\ge 12q^2(\vec U)  > 0 ,
\end{split}
\end{equation*}
the function $f_2(\xi)$  has two real roots.
\end{proof}

\begin{remark}\label{rem:rmhd:equal}
	Using \eqref{eq:RMHD:exConstraint} and some algebraic manipulations, one can verify that the constraint ${\Psi}(\vec U) >0$ is equivalent to two constraints $\hat q(\vec U)>0$ and $\tilde q(\vec U)>0$, where
\begin{align*}
		&\hat q(\vec U):= \Phi(\vec U) -  2 \big(|\vec B|^2-E \big) = \sqrt{ \left( E-|\vec B|^2 \right)^2+3\left( E^2-D^2-|\vec m|^2 \right) } + 2 \left( E-|\vec B|^2 \right), \\
		&\tilde q(\vec U):= \Phi^6(\vec U) - \bigg( \left( E-|\vec B|^2 \right)^3 + \frac{27}{2} \left( {|\vec B|^2D^2 + |\vec m \cdot \vec B|^2} \right) - 9 \left( E^2-D^2-|\vec m|^2 \right) \left( E-|\vec B|^2 \right)  \bigg)^2.
\end{align*}
	Moreover, ${\Psi}(\vec U) =0$ if and only if $\hat q(\vec U)\ge0$ and $\tilde q(\vec U)=0$.
\end{remark}

\begin{remark}
	The first equivalent definition of $\mathcal G$ is very important in following  aspects:
	\begin{itemize}%[\hspace{0em}$\bullet$]
		\item to guide the initial guess in numerically solving the
		nonlinear algebraic equation \eqref{eq:RMHD:fU(xi)},
		because the proof of Theorem \ref{theo:RMHD:CYconditionFINAL2}
		has shown that $\xi_*>\xi_4$ for $\vec U \in {\mathcal G}$, where $\xi_4$ is discussed in
		Lemma \ref{lemma:RMHD:f4},
		and
		\begin{equation*}
        \begin{split}
			\Gamma E - \xi_*(\vec U) & = \Gamma \left(\rho h W^2 - p + \frac{1+v^2}{2} |\vec B|^2 - \frac{1}{2} (\vec v \cdot \vec B)^2 \right) - \rho h W^2 \\
			& \ge  \Gamma \left(\rho h W^2 - p \right) - \rho h W^2 \ge (\Gamma - 1) \rho h - \Gamma p = (\Gamma -1) \rho > 0.
		\end{split}
        \end{equation*}
		\item to develop the {PCP} limiter  and  high-order accurate PCP schemes for  the {1D} RMHD equations  \eqref{eq:RMHD1D}, see Section \ref{sec:High1D}.
		\item to prove the convexity of $\mathcal G$, see Section \ref{sec:convexity},
		and the scaling invariance.
	\end{itemize}
\end{remark}

\begin{corollary}[Scaling invariance] \label{lem:RMHD:scaling}
	If the state $\vec U=(D,\vec m,\vec B,E)^{\top}\in {\mathcal G}_0$, then for any $\lambda \in {\mathbb{R}}^+$,
	the state
	$\vec U_\lambda := ( \lambda D, \lambda \vec m, \sqrt{\lambda}  \vec B, \lambda  E)^{\top} \in {\mathcal G}_0$.
\end{corollary}

\begin{proof}
	It can be verified that
	$q( \vec U_\lambda  ) = \lambda q(\vec U) >0$ and
	${\Psi}( \vec U_\lambda  ) = \lambda^{3/2} {\Psi} (\vec U) >0$.
	The proof is completed.
\end{proof}

\subsection{Convexity}\label{sec:convexity}

This section will prove the convexity of  admissible state set {$\mathcal G_0=\mathcal G$ for} the RMHDs.
It will play a pivotal role in  analyzing the PCP property of numerical schemes.

\begin{theorem}\label{theo:RMHD:convex}
The admissible state set ${\mathcal G}_0$ is a convex set.
\end{theorem}

\begin{proof}
	It is not trivial and cannot be completed by using the convexity definition of the set
	because of the strong nonlinearity of the function ${\Psi} (\vec U)$ used in  \eqref{eq:RMHD:definitionG2}.
	Instead, it will be done
	with the aid of the close connection between the set  convexity in $\mathbb{R}^N$ and the concave-convex character of the region boundary  corresponding to the set, see e.g. \cite{Jonker1975}.
	
 It is easy to show by  the proof of Lemma 2.2 in \cite{WuTang2015} that the set
	$$\mathcal G_2:=\left\{   \vec U=(D,\vec m,\vec B,E)^{\top} \in {\mathbb{R}}^8 \big|  D>0,q(\vec U)>0 \right\},$$
is convex. Therefore, the subsequent task is to prove that
the hypersurface $\mathcal S$ in ${\mathbb{R}}^8$ described by
\begin{equation}\label{eq:RMHD:Constraint3=0}
{\Psi} (\vec U) = 0,
\end{equation}
is convex within the region ${\mathcal G}_2$, and the points in ${\mathcal G}_0 $ are all located in the concave side of the hypersurface $\mathcal S$.
Unfortunately, it is  impractical  to  check
the convexity of the hypersurface $\mathcal S$ by directly using the
highly nonlinear equation \eqref{eq:RMHD:Constraint3=0} via
the theory of geometry.
%thanks to the strong nonlinearity of ${\Psi} (\vec U)$. % is highly nonlinear and complicate.
%It is difficult and will involve very complicated algebraic manipulations
To overcome this difficulty, we try to
give a parameter equation for the hypersurface $\mathcal S$.
% may make the analysis feasible and relatively easy.

An important discovery is that \eqref{eq:RMHD:Constraint3=0} is equivalent to $p(\vec U)=0$
for $\vec U\in \mathcal G_2$. In fact, on the one hand, it can {be seen} from the proof of Theorem \ref{theo:RMHD:CYconditionFINAL2} that  \eqref{eq:RMHD:Constraint3=0} implies $f_3(  \xi_{2,R}(\vec U) )=0$.  It means that $\xi_3(\vec U) = \xi_{2,R}(\vec U)$ and  satisfies $f_2(\xi_3(\vec U))=0$,
which  yields $f_4(\xi_3(\vec U))=0$. It follows that $\xi_4(\vec U)=\xi_3(\vec U)$
and satisfies  $f_3(\xi_4(\vec U))=0$ and $f_{\smallvec U}  (\xi_4)=0$. This further implies $\xi_*(\vec U)=\xi_4(\vec U)$, and thus one has that $\xi_*=D W(\xi_*)$ and
 $p(\vec U)=0$. On the other hand, if $p(\vec U)=0$, then $h=1+e+p/\rho=1$, and
\begin{equation}\label{eq:RMHD:pU=0}
\begin{cases}
 D = \dfr{\rho }{{\sqrt {1 - { v^2}} }} ,\\
 \vec m = \dfr{{\rho \vec v}}{{1 - {  v^2}}} + {|\vec B|^2} \vec v - (\vec v \cdot \vec B) \vec B, \\
 E  =\rho W^2-p_m+|\vec B|^2= \dfr{\rho }{{1 - { v^2}}} + \dfr{{1 + { v^2}}}{2}{|\vec B|^2} - \dfr{{{{(\vec v \cdot \vec B)}^2}}}{2}.
 \end{cases}
\end{equation}
Thus one has
\begin{align*} %\label{eq:RMHD:pU=0,2}
&
\Phi^2(\vec U)=\left( E-|\vec B|^2 \right)^2+3\left( E^2-D^2-|\vec m|^2 \right)=\left(\rho W^2 + 2 p_m \right)^2,\\
&
\left( E-|\vec B|^2 \right)^3 + \frac{27}{2} \left( {|\vec B|^2D^2 + |\vec m \cdot \vec B|^2} \right) - 9 \left( E^2-D^2-|\vec m|^2 \right) \left( E-|\vec B|^2 \right) = \left(\rho W^2 + 2 p_m \right)^3,
\end{align*}
which imply that $\hat q(\vec U)$ and $\tilde q(\vec U)$
 in Remark \ref{rem:rmhd:equal} satisfy
\begin{align*}
\hat q(\vec U) =\sqrt{ \left( E-|\vec B|^2 \right)^2+3\left( E^2-D^2-|\vec m|^2 \right) } + 2 \left( E-|\vec B|^2 \right) %=\sqrt{\left(\rho W^2 + 2 p_m \right)^2}+2(E-|\vec B|^2)
= 3\rho W^2 >0.
\end{align*}
and $\tilde q(\vec U)=0$.
The conclusion in Remark \ref{rem:rmhd:equal} yields \eqref{eq:RMHD:Constraint3=0}.

Based on the above {discovery}, the hypersurface $\mathcal S$ defined in \eqref{eq:RMHD:Constraint3=0} can be represented
by the  parametric equations  \eqref{eq:RMHD:pU=0}
 through seven
parameters $\vec {\mathcal V}:=(\rho,\vec v,\vec B)^{\top} $ with $\rho>0, |\vec v|<1 $ and $\vec B \in {\mathbb {R}}^3$.
Obviously, the hypersurface $\mathcal S$ is {not 6-cylindrical}.
Based on the theorem in \cite{Jonker1975},
one only needs to show that its second fundamental form is positive semi-definite, i.e. prove that the matrix
\begin{align*}
\vec \Pi :={\left[ {\sum\limits_{l = 1}^{8} {\frac{{{\partial ^2}{U_{\langle l \rangle}}}}{{\partial {{\mathcal V}_{\langle i \rangle}}\partial {{\mathcal V}_{ \langle j \rangle}}}}{n_l}} } \right]_{7 \times 7}}
\end{align*}
is positive semi-definite, where $U_{\langle l \rangle}$ and ${\mathcal V}_{\langle i \rangle}$ denote the $l$-th component of the vector $\vec U$ and the $i$-th component of the vector $\vec {\mathcal V}$, respectively, and $\vec n:=(n_1,n_2,\cdots,n_8)^{\top}$ represents the inward-pointing (to the region ${\mathcal G}_0 $) normal vector of the hypersurface $\mathcal S$.
Taking partial derivatives of $\vec U$ with respect to ${\mathcal V}_{\langle i \rangle }$ gives
%\begin{align*}
\begin{align*}
& {\partial _\rho } {\vec U} = \left(W,{W^2}{v_1},{W^2}{v_2},{W^2}{v_3},0,0,0,{W^2}\right)^{\top}, \\
 \begin{split}
& {\partial _{{v_1}}} {\vec U} = \Big(\rho {W^3}{v_1},\rho {W^2}(1 + 2{W^2}v_1^2) + B_2^2 + B_3^2,2\rho {W^4}{v_1}{v_2} - {B_1}{B_2},2\rho {W^4}{v_1}{v_3} - {B_1}{B_3}, \\
&\qquad \qquad   0,0,0,{\left| \vec B \right|^2}{v_1} - {B_1}(\vec v \cdot \vec B) + 2\rho {W^4}{v_1}\Big)^{\top},
 \end{split}
\\
 \begin{split}
& {\partial _{{v_2}}} \vec U = \Big(\rho {W^3}{v_2},2\rho {W^4}{v_1}{v_2} - {B_1}{B_2},\rho {W^2}(1 + 2{W^2}v_2^2) + B_1^2 + B_3^2,2\rho {W^4}{v_2}{v_3} -  {B_2}{B_3}, \\
&\qquad \qquad  0,0,0, {\left| \vec B \right|^2}{v_2} - {B_2}(\vec v \cdot \vec B) + 2\rho {W^4}{v_2}\Big)^{\top},
 \end{split}
\\
 \begin{split}
 &{\partial _{{v_3}}}\vec U = \Big(\rho {W^3}{v_3},2\rho {W^4}{v_1}{v_3} - {B_1}{B_3},2\rho {W^4}{v_2}{v_3} - {B_2}{B_3},\rho {W^2}(1 + 2{W^2}v_3^2) + B_1^2 + B_2^2,\\
&\qquad  \qquad   0,0,0,{\left| \vec B \right|^2}{v_3} - {B_3}(\vec v \cdot \vec B) + 2\rho {W^4}{v_3}\Big)^{\top},
 \end{split}
\\
& {\partial _{{B_1}}}\vec U = \left( {0, - {B_2}{v_2} - {B_3}{v_3},2{B_1}{v_2} - {B_2}{v_1},2{B_1}{v_3} - {B_3}{v_1},1,0,0,{B_1}(1 + v^2) - {v_1}(\vec v \cdot \vec B)} \right)^{\top}, \\
& {\partial _{{B_2}}}\vec U = \left( {0,2{B_2}{v_1} - {B_1}{v_2}, - {B_1}{v_1} - {B_3}{v_3},2{B_2}{v_3} - {B_3}{v_2},0,1,0,{B_2}(1 +  v^2) - {v_2}(\vec v \cdot \vec B)} \right)^{\top}, \\
& {\partial _{{B_3}}}\vec U = \left( {0,2{B_3}{v_1} - {B_1}{v_3},2{B_3}{v_2} - {B_2}{v_3}, - {B_1}{v_1} - {B_2}{v_2},0,0,1,{B_3}(1 +  v^2) - {v_3}(\vec v \cdot \vec B)} \right)^{\top}.
\end{align*}
These are seven  tangent vectors of the hypersurface $\mathcal S$ and  generate the local tangent space. Because they are perpendicular to the  normal vector $\vec n$,
their inner products  with $\vec n$ should be equal to zero, and thus  a linear system of seven algebraic equations for $(n_1,n_2,\cdots,n_8)^{\top}$ is formed.
Solving this linear system gives
\begin{equation}\label{eq:RMHD:vecn}
\vec n = {\left( - \sqrt {1 - v^2} , - \vec v, - (1 - { v^2}) \vec B - (\vec v \cdot \vec B) \vec v,1 \right)^{\top}}.
\end{equation}
First, let us  check the positive semi-definiteness of $\vec \Pi$.
%and will prove that $\vec n$ in \eqref{eq:RMHD:vecn} is the inward-pointing normal vector  to the region ${\mathcal G}_0 $ later.
Taking the second-order   partial derivatives of $\vec U$ with respect to $\vec {\mathcal V}$, and then calculating their inner products  with $\vec n$ give the expression of the matrix $\vec \Pi$ as follows
\[\vec \Pi  = {\rm diag}\{ 0,\vec \Pi_1, \vec \Pi_2 \},\]
where
\begin{equation*}
\begin{split}
&
 {\vec \Pi _1} = \rho W^4 \left[ {(1 - { v^2}){\vec I_3} + \vec v^{\top} {\vec v}} \right] + {|\vec B|^2}{\vec I_3} - \vec B^{\top}{\vec B}=\rho W^4{\vec \Pi _2} + {|\vec B|^2}{\vec I_3} - \vec B^{\top}{\vec B}, \\
&
 {\vec \Pi _2} = (1 - { v^2}){\vec I_3} + \vec v^{\top}{\vec v}.
\end{split}
\end{equation*}
Here $\vec I_3$ denotes a unit matrix of size 3.
The matrix $\vec v ^{\top}{\vec v}$ has rank of 1 and eigenvalues of $\{0,0,|\vec v|^2\}$, so the eigenvalues of ${\vec \Pi _2} $ are $\{1 - {|\vec v|^2},1 - {|\vec v|^2},1\}$, %and all positive,
which imply the positive definiteness of ${\vec \Pi _2} $.
Similarly, one can show that the eigenvalues of ${|\vec B|^2}{\vec I_3} - \vec B^{\top}{\vec B}$ are $\{|\vec B|^2,|\vec B|^2,0\}$. Because ${\vec \Pi _1}$ is the sum of a positive definite matrix and a positive semi-definite matrix, it is  positive semi-definite.
In conclusion, $\vec \Pi$ is a positive semi-definite matrix with positive inertia index of 6
so that  the hypersurface $\mathcal S$ described in  \eqref{eq:RMHD:Constraint3=0} is a convex surface in $\mathcal G_2$.

Next, let us  prove that all the points in ${{\mathcal G}_0 = {\mathcal G} }$ are located at the concave side of the hypersurface $\mathcal S$, that is to say, the normal vector $\vec n$ in \eqref{eq:RMHD:vecn} is the inward-pointing vector  to the region ${\mathcal G}_0 $.
For this purpose, we need to show that, for any $\widetilde {\vec U} \in {{\mathcal G}_0 = {\mathcal G} }$ and  $\vec U\in \mathcal S$, it holds that
\begin{equation*}%\label{eq:RMHDinG3}
(\widetilde {\vec U} - \vec U ) \cdot \vec n >0,
\end{equation*}
which is equivalent to
\begin{equation*} %\label{eq:RMHDinG3_A}
\widetilde{\mathcal  F} ( \tilde \rho, \tilde p,\tilde {\vec v},\tilde {\vec B}, \vec v,\vec B ) :=\widetilde {\vec U} \cdot \vec n  + p_m >0,
\end{equation*}
because  of \eqref{eq:RMHD:pU=0} and \eqref{eq:RMHD:vecn}.
% gives that $ \vec U  \cdot \vec n= p_m$.
By defining  $\widetilde{\mathcal  F}_0 ( \tilde {\vec v},\tilde {\vec B}, \vec v,\vec B ) := \widetilde{\mathcal  F} ( 0, 0,\tilde {\vec v},\tilde {\vec B}, \vec v,\vec B )$,
one can infer that
\begin{equation*}
\begin{split}
\widetilde{\mathcal  F}_0 ( \tilde {\vec v},\tilde {\vec B}, \vec v,\vec B )
& =
\left( \left| \vec{\tilde B} \right|^2 \vec {\tilde v} - (\vec {\tilde v} \cdot \vec {\tilde B}) \vec {\tilde B} \right) \cdot (-\vec v) + \left( W^{-2} \vec B + (\vec v \cdot \vec B) \vec v \right) \cdot (- \tilde{ \vec B})
\\
&~~~~ + \frac{(1+ {\tilde v}^2) |\vec {\tilde B}|^2 - (\vec{\tilde v}\cdot \vec{\tilde B})^2 }{2} +
\frac{ ( 1- v^2  ) |\vec B|^2 +(\vec v \cdot \vec B)^2 }{2}
\\
& = \frac{ ( 1- v^2 ) |\vec B -\vec{\tilde B}|^2 }{2} +  \frac{ |\vec v - \vec {\tilde v}|^2  |\vec{ \tilde B}|^2} {2}
\\
&~~~~  + \frac{ (\vec v \cdot \vec B)^2 }{2} - (\vec v \cdot \vec B) (\vec v\cdot \vec  {\tilde B} )
- \frac{(\vec {\tilde v} \cdot \vec {\tilde B})^2}{2} +  (\vec {\tilde v} \cdot \vec {\tilde B})  (\vec v \cdot \vec  {\tilde B} )
\\
&
\ge
\frac{ |\vec v - \vec {\tilde v}|^2 |\vec{ \tilde B}|^2  } {2}
+ \frac{ \left[(\vec v \cdot \vec B) - (\vec v\cdot \vec  {\tilde B} ) \right]^2 }{2}
- \frac{\left[ (\vec {\tilde v} \cdot \vec {\tilde B}) - (\vec v\cdot \vec  {\tilde B} ) \right]^2}{2}
\\
&
=
\frac{  |\vec v - \vec {\tilde v}|^2 |\vec{ \tilde B}|^2} {2}
- \frac{ \big( ( \vec v -\vec {\tilde v}) \cdot   \vec  {\tilde B} \big)^2}{2}
+ \frac{ \left(\vec v \cdot (\vec B- \vec {\tilde B}) \right) ^2 }{2} \ge 0.
\end{split}
\end{equation*}
Thus  for any given~$\vec U$ on the hypersurface $\mathcal S$,
%considering the function $\widetilde{\mathcal  F} ( \tilde \rho, \tilde p,\tilde {\vec v},\tilde {\vec B}, \vec v,\vec B ) := \tilde {\vec U} \cdot \vec n  + p_m $,
one has
\begin{equation*}
\begin{split}
 %\widetilde{\mathcal  F}  ( \tilde \rho, & \tilde p,\tilde {\vec v},\tilde {\vec B}, \vec v,\vec B ) =
 \widetilde{\vec U} \cdot \vec n  + p_m
& =\tilde \rho  \tilde W^2( 1-  \vec { \tilde v} \cdot {\vec v} - \tilde W^{-1} W^{-1} ) +
    \tilde p \left ( \frac{\Gamma}{\Gamma-1} \tilde W^2 (1-\vec {\tilde v} \cdot \vec v ) -1    \right) +  \widetilde{\mathcal  F}_0 ( \tilde {\vec v},\tilde {\vec B}, \vec v,\vec B )
    \\
&
\ge \tilde \rho  \tilde W^2 \bigg( 1-  (\tilde v_1,\tilde v_2, \tilde v_3, \tilde W^{-1}) \cdot (v_1,v_2,v_3,W^{-1})  \bigg) +
    \tilde p \left ( 2 \tilde W^2 (1-\vec {\tilde v} \cdot \vec v ) -1    \right)
    %+ \widetilde{\mathcal  F}_0 ( \tilde {\vec v},\tilde {\vec B}, \vec v,\vec B )
\\
&\ge
\tilde \rho  \tilde W^2 \bigg( 1-  \left|(\tilde v_1,\tilde v_2, \tilde v_3, \tilde W^{-1}) \right| \left| (v_1,v_2,v_3,W^{-1}) \right| \bigg) +
    \tilde p \left ( 2 \tilde W^2  (1- \left|\vec {\tilde v} \right| \left| \vec v\right| ) -1    \right)
    %+ \widetilde{\mathcal  F}_0 ( \tilde {\vec v},\tilde {\vec B}, \vec v,\vec B )
\\
& \ge
    \tilde p \left ( 2 \tilde W^2  (1- \left|\vec {\tilde v} \right|  ) -1    \right)
    %+ \widetilde{\mathcal  F}_0 ( \tilde {\vec v},\tilde {\vec B}, \vec v,\vec B )
  =
\frac { \tilde p (1-| \vec {\tilde v}|)} { 1+| \vec {\tilde v}| }
%+ \widetilde{\mathcal  F}_0 ( \tilde {\vec v},\tilde {\vec B}, \vec v,\vec B )
%\\&
 > 0.
 %\widetilde{\mathcal  F}_0 ( \tilde {\vec v},\tilde {\vec B}, \vec v,\vec B ),
\end{split}
\end{equation*}
%where
%\begin{align*}
 %\widetilde{\mathcal  F}_0 ( \tilde {\vec v},\tilde {\vec B}, \vec v,\vec B )
 %& =
 %\left( \left| \vec{\tilde B} \right|^2 \vec {\tilde v} - (\vec {\tilde v} \cdot \vec {\tilde B}) \vec {\tilde B} \right) \cdot (-\vec v) + \left( W^{-2} \vec B + (\vec v \cdot \vec B) \vec v \right) \cdot (- \tilde{ \vec B})
%\\
%&~~~~ + \frac{(1+|\vec {\tilde v}|^2) |\vec {\tilde B}|^2 - (\vec{\tilde v}\cdot \vec{\tilde B})^2 }{2} +
%\frac{ ( 1-|\vec v|^2  ) |\vec B|^2 +(\vec v \cdot \vec B)^2 }{2}
%\\
%& = \frac{ ( 1-|\vec v|^2 ) |\vec B -\vec{\tilde B}|^2 }{2} +  \frac{ |\vec v - \vec {\tilde v}|^2  |\vec{ \tilde B}|^2} {2}
%\\
%&~~~~  + \frac{ (\vec v \cdot \vec B)^2 }{2} - (\vec v \cdot \vec B) (\vec v\cdot \vec  {\tilde B} )
%- \frac{(\vec {\tilde v} \cdot \vec {\tilde B})^2}{2} +  (\vec {\tilde v} \cdot \vec {\tilde B})  (\vec v \cdot \vec  {\tilde B} )
%\\
%&
%\ge
 %\frac{ |\vec v - \vec {\tilde v}|^2 |\vec{ \tilde B}|^2  } {2}
 %+ \frac{ \left[(\vec v \cdot \vec B) - (\vec v\cdot \vec  {\tilde B} ) \right]^2 }{2}
%- \frac{\left[ (\vec {\tilde v} \cdot \vec {\tilde B}) - (\vec v\cdot \vec  {\tilde B} ) \right]^2}{2}
%\\
%&
%=
 %\frac{  |\vec v - \vec {\tilde v}|^2 |\vec{ \tilde B}|^2} {2}
 %- \frac{ \big( ( \vec v -\vec {\tilde v}) \cdot   \vec  {\tilde B} \big)^2}{2}
 %+ \frac{ \left(\vec v \cdot (\vec B- \vec {\tilde B}) \right) ^2 }{2} \ge 0.
%\end{align*}
%This yields \eqref{eq:RMHDinG3_A} and completes the proof.
The proof is completed.
\end{proof}

\subsection{Second equivalent definition}\label{sec:second}

The convexity of the admissible state set $\mathcal G$ can
give its second equivalent form,
whose  importance lies in that all constraints
are  linear  with respect to $\vec U$ so that
it will be very effective in verifying theoretically
the PCP property of  the numerical schemes for the RMHD equations \eqref{eq:RMHD1D}.

\begin{theorem}[Second equivalent definition]\label{theo:RMHD:CYcondition:VecN}
The admissible state set ${\mathcal G}$ or ${\mathcal G}_0$ is equivalent to the   set
\begin{equation}\label{eq:RMHD:CYcondition:VecNG1}
\begin{split}
{\mathcal G}_1 := \Big\{   \vec U=(D,\vec m,\vec B,E)^{\top} \in \mathbb{R}^8 \big|  D>0, \vec U \cdot
{{\vec n^*}} + {p^*_m} >0, \\
 \mbox{for any {${\vec B^*}, {\vec v^*} \in \mathbb{R}^3$} with  $  |\vec v^*|<1$} \Big\},
\end{split}
\end{equation}
where
\begin{align}\label{eq:RMHD:vecns}
&{\vec n}^* = {\left( - \sqrt {1 - {|\vec v^*|}^2} ,~
	- {\vec v}^*,~ - (1 - {|\vec v^*|}^2) {\vec B}^* - ({\vec v}^* \cdot {\vec B}^*) {\vec v}^*,~1 \right)^{\top}},\\
& p_{m}^*  = \frac{ (1-{|\vec v^*|}^2) |{\vec B}^*|^2 +({\vec v}^* \cdot {\vec B}^*)^2 }{2}. \label{eq:RMHD:vecns2}
\end{align}
Here ${\vec U}^*$ denotes any point on the hypersurface $\mathcal S$, and
	$p_{m}^*=-{\vec U}^*\cdot {\vec n}^*$ and ${\vec n}^*$ are  corresponding  magnetic pressure and  inward-pointing vector  to the region ${\mathcal G}_0 $, respectively.
\end{theorem}

\begin{proof}
Theorem \ref{theo:RMHD:convex} and its proof have shown that ${\mathcal G}_0={\mathcal G} \subseteq {\mathcal G}_1$.
The subsequent task is to prove ${\mathcal G}_1 \subseteq {\mathcal G}_0$. For any $\vec U \in {\mathcal G}_1$, the convexity of the hypersurface $\mathcal S$ in \eqref{eq:RMHD:Constraint3=0} implies the constraint ${\Psi}(\vec U)>0$ in ${\mathcal G}_0$. Thus it needs to prove that  the state $\vec U\in {\mathcal G}_1$
 satisfies  $q(\vec U)>0$.
% Because for any $\tilde{\vec v}$ and $ \tilde{\vec B}\in \mathbb{R}^3 $ with $\tilde v<1$, one has
%$$
%\vec U \cdot \tilde{\vec n} + \tilde p_{m}>0,\quad \vec U \in {\mathcal G}_1.
%$$
 If taking the vectors {${\vec B}^*, {\vec v}^*\in \mathbb{R}^3$} as
$${\vec B}^*= \vec 0,\quad  {\vec v}^* = \frac{1}{\sqrt{D^2+|\vec m|^2}} \vec m ,$$
and substituting them into the second inequality in ${\mathcal G}_1$,
one has
\begin{equation*}
\begin{split}
0 & < \vec U \cdot{\vec n}^* +  p_{m}^* = E - \vec m \cdot {\vec v}^* - D \sqrt{1-{|\vec v^*|}^2} \\
& = E - \frac{ |\vec m|^2 } { \sqrt{D^2+|\vec m|^2} } - \frac{ D^2} { \sqrt{D^2+|\vec m|^2} } \\
& = E - \sqrt{D^2+|\vec m|^2} = q(\vec U).
\end{split}
\end{equation*}
The proof is completed.
\end{proof}

\begin{remark}
It is seen that $\vec n^*$ in \eqref{eq:RMHD:vecns} can be rewritten as
$$
\vec n^* = -\sqrt{1-|\vec v^*|^2} \big( 1, u_1^*,u_2^*,u_3^*,b_1^*,b_2^*,b_3^*,u_0^*\big)^{\top},
$$
where $u_\alpha^*$ and $b_\alpha^*$ denote the velocity and magnetic field in 4D space-time, respectively.
\end{remark}

\begin{remark}
Theorems \ref{theo:RMHD:CYconditionFINAL2} and \ref{theo:RMHD:CYcondition:VecN} indicate that
 ${\mathcal G}={\mathcal G}_0={\mathcal G}_1$. Thus they
 will not be deliberately distinguished henceforth.
 %, because they are all equivalent descriptions of the admissible state set.}
\end{remark}

Theorem \ref{theo:RMHD:CYcondition:VecN}  implies the following
orthogonal invariance of the admissible state set  ${\mathcal G}_1$.

\begin{corollary}[Orthogonal invariance] \label{lem:RMHD:zhengjiao}
Let $\vec T :={\rm diag}\{1,\vec T_3,\vec T_3,1\}$, where
$\vec T_3$ denotes any  orthogonal matrix  of size 3.
If $\vec U \in{\mathcal G}_1$, then
$\vec T \vec U \in{\mathcal G}_1$.
\end{corollary}

\begin{proof}
 For any $\vec U=(D,\vec m,\vec B,E)^{\top}\in{\mathcal G}_1$,	
 if denoting $\bar{\vec U}=\vec T \vec U=:(\bar D, \bar {\vec m}, \bar{\vec B}, \bar E)^{\top}$,
 then $\bar D = D >0$.
For any ${\vec B}^*,{\vec v}^* \in  \mathbb{R}^3$ with $|\vec v^*|<1$, if denoting
$ {\hat{\vec B}}^*:=  {\vec B}^* \vec T_3,~{\hat{\vec v}}^*:= {\vec v}^* \vec T_3 $, then
 $|{\hat{\vec  v}}^*|=|{ \vec v}^*|<1$, ${\hat p}_{m}^* =   p_{m}^*$, and ${\hat{\vec n}}^* = \vec T^{-1} {\vec n}^* $.
Using Theorem \ref{theo:RMHD:CYcondition:VecN} for $\vec U \in{\mathcal G}_1$ gives
\begin{align*}
0 & < \vec U \cdot \hat{\vec n}^*+ \hat p_{m}^* = \left ( \vec T^{-1} \bar {\vec U} \right) \cdot \left( \vec T^{-1}  {\vec n}^* \right) +  p_{m}^* = \bar {\vec U} \cdot {\vec n}^* +  p_{m}^*,
\end{align*}
where the orthogonality of $ \vec T^{-1}$ has been used in the last equality.
Hence using Theorem \ref{theo:RMHD:CYcondition:VecN} again
yields $\bar {\vec U} \in{\mathcal G}_1$. The proof is completed.
\end{proof}

\begin{remark}
%It is worth noting that
Corollary \ref{lem:RMHD:zhengjiao}
implies the rotational or symmetric  invariance of the admissible state set $\mathcal G$
if $\vec T_3$ is taken as a rotational or symmetric matrix of size 3.
\end{remark}

\subsection{Generalized Lax-Friedrichs splitting properties} \label{sec:GLFs}

The section utilizes  the second equivalent definition of $\mathcal G$ in Theorem \ref{theo:RMHD:CYcondition:VecN} to present the generalized Lax-Friedrichs (LxF) splitting properties of the admissible state set $\mathcal G$
for the special RMHD equations \eqref{eq:RMHD1D}.

\begin{lemma}[LxF splitting]\label{lemma:2.8}
	If {$\vec B=\vec 0$},  then the special RMHD equations \eqref{eq:RMHD1D} satisfy the {\em LxF splitting property}:
	\begin{align*}
	\mbox{$\vec U \pm \alpha^{-1} \vec F_i (\vec U) \in {\mathcal G}$ for  $\vec U \in {\mathcal G}$},
	\end{align*}
	where   $\alpha \ge \varrho_i$ and $\varrho_i$ denotes a proper upper bound of the spectral radius of the Jacobian matrix $\partial \vec F_i/\partial \vec U$, $i=1,2,3$.
If {$\vec B\neq \vec 0$},	then  the {\em LxF splitting property} does not {always} hold.
\end{lemma}
\begin{proof}
	The first part has been proved in \cite{WuTang2015},
%	the above  {\em Lax-Friedrichs splitting property} holds, see Lemma 2.3(iii) in
%	and has played an important role in analyzing the PCP property of the numerical schemes for the RHD equations. 	
	while the second part  is proved by contradiction as follows.
	
	Assume that the LxF splitting property holds for $\vec U\in \mathcal G$ and $\Gamma \in(1,2]$.
	For any $\vec V=(\rho,\vec v,\vec B,p)^{\top} $ satisfying $\rho>0,~p>0$, and $v<1$,  one has
	$$\vec U(\vec V) \pm \alpha^{-1} \vec F_i (\vec U(\vec V)) \in {\mathcal G},\quad \forall \alpha  \ge \varrho_i,~i=1,2,3.$$
Because	the speed of light $c =1$ is a rigorous bound of the spectral radius of $\partial \vec F_i/\partial \vec U$,
one can specially take $\rho=p=\epsilon>0,\vec v=(0.5,0,0), \vec B = (1,0,0)$, $\alpha = 1/\theta$ for $\theta \in (0, 1]$, and $\Gamma=5/3$, such that
	\begin{equation*}
    \begin{split}
	\vec U^\pm(\epsilon,\theta) := &\vec U \pm \alpha^{-1} \vec F_1 (\vec U)
	\\
	=& \left( \frac{\sqrt{3}}{3} (2+\theta) \epsilon,
	\frac{14 \pm 13 \theta}{6} \epsilon \mp \frac{ \theta}{2}
	,0,0,1,0,0, \frac{ 11 \pm 7 \theta }{3} \epsilon + \frac{1}{2}  \right)^{\top} \in {\mathcal G}={\mathcal G}_0.
    \end{split}
	\end{equation*}
	According to Remark \ref{rem:rmhd:equal}, one has $\tilde q(\vec U^\pm(\epsilon,\theta) )>0$, for all $\epsilon>0$
	 and $\theta \in (0, 1]$.
	The continuity of $\tilde q(\vec U)$ with respect to $\vec U$ further implies that for any fixed $\theta$,
	$\vec U^\pm(\epsilon,\theta)$ is also continuous with respect to $\epsilon$. Therefore
	$$
	0 \le \mathop {\lim }\limits_{\epsilon \to 0^+ } \tilde q(\vec U^\pm(\epsilon,\theta) ) = \tilde q(\vec U^\pm(0,\theta) ) = - \frac{27}{64} \theta^2 (\theta^2+4)^2 <0,
	$$
	which leads to the contradiction. Hence the LxF splitting property does not hold for the admissible state set  $\mathcal G$  for the RMHD equations  \eqref{eq:RMHD1D} in general.
\end{proof}

%For an explicit numerical scheme for \eqref{eq:RMHD1D} with $d=1$, which is assumed to be possibly expressed as
%\begin{equation}\label{eq:a:explicit:scheme}
%{\vec U}_j^{n+1} = {\mathcal E}_h \left( {\vec U}_{j-k}^{n}, {\vec U}_{j-k+1}^{n}, \cdots, {\vec U}_{j+i-1}^{n}, {\vec U}_{j+i}^{n}   \right),
%\end{equation}
%evidently it is extremely difficult to analyze the PCP property of \eqref{eq:a:explicit:scheme}, i.e. checking whether ${\vec U}_j^{n+1} \in {\mathcal G}$ under the assumption that ${\vec U}_{j}^{n}\in{\mathcal G}$ for all $j \in {\mathbb {Z}}$. In \eqref{eq:a:explicit:scheme}, ${\vec U}_{j}^{n}$ stands for the numerical solution at cell $j$ and time level $n$, and  ${\mathcal E}_h$ represents the discrete solver operator of the scheme and is also highly nonlinear due to the nonlinearity of \eqref{eq:RMHD1D}, especially no explicit expression of $\vec F_i(\vec U)$ with respect to $\vec U$.
%In short, the nonlinearities of $\vec V(\vec U)$ and $\vec F_i(\vec U)$ mainly reflect in the description of the admissible state set and the numerical scheme for \eqref{eq:RMHD1D}, respectively.

%In order to analyze the PCP property of a numerical scheme based on the Lax-Firedrichs type flux, the following property referred as the Lax-Friedrichs splitting is usually expected for the admissible state set.
%Unfortunately, this property does not hold for the admissible state set $\mathcal G$ of the RMHDs in general.

Although the LxF splitting property may not  {hold} {for  the nonzero magnetic field},
we {discover} the generalized LxF splitting properties which are coupling
the  convex combination of some LxF splitting terms with a ``discrete divergence-free'' condition for the magnetic
field vector  $\vec B$.
%The absence of the Lax-Friedrichs splitting property makes it difficult to analyze the PCP property of numerical schemes based on the simple Lax-Friedrichs flux. However, through a large number of numerical tests,
%we observe an interesting phenomenon:
%special convex combination of some Lax-Friedrichs splitting terms may become admissible, when those involved states in the Lax-Friedrichs splitting are admissible and satisfy a ``discrete divergence-free'' condition. This leads to the following generalized Lax-Friedrichs splitting properties.
However, it is extremely difficult  and technical %to derive the generalized   Lax-Friedrichs splitting properties
because of the ``discrete divergence-free'' condition for the magnetic field $\vec B$
and the strong nonlinearity in the constraints of the admissible state set and $\vec F_i(\vec U)$, etc.
Their breakthrough is {made by} a constructive inequality in the following lemma. % \ref{theo:RMHD:LLFsplit}.

\begin{lemma}\label{theo:RMHD:LLFsplit}
	If $\vec U \in {\mathcal G}$, then
	for any $\theta \in [-1,1]$ and ${\vec B}^*$, $\vec v^*$ $\in \mathbb{R}^3$ with $|\vec v^*|<1$
	it holds
	\begin{equation}\label{eq:RMHD:LLFsplit}
	\big( \vec U + \theta \vec F_i(\vec U) \big) \cdot \vec n^* +  p_{m}^* + \theta \big( v_{i}^* p_{m}^* - B_i (\vec v^* \cdot \vec B^*)\big)>0,
	\end{equation}
	where $i\in\{1,2,3\}$, and $\vec n^*$ and $p_{m}^*$ are defined in \eqref{eq:RMHD:vecns} and \eqref{eq:RMHD:vecns2}, respectively.
\end{lemma}

\begin{proof}
%	Due to the strong nonlinearity of \eqref{eq:RMHD:LLFsplit}, especially the involved flux $\vec F_i$, its proof is non-trivial and very technical.
{\tt (i)}. First let  us prove the inequality \eqref{eq:RMHD:LLFsplit} for the case of $i=1$, i.e.
	\begin{equation}\label{eq:RMHD:LLFsplit-thz}
{\mathcal H} (\rho,p,\vec v, \vec B,\vec v^*,\vec B^*,\theta ) := \big( \vec U + \theta \vec F_1(\vec U) \big) \cdot \vec n^* + (1+\theta v_{1}^* ) p_{m}^* - \theta B_1 (\vec v^* \cdot \vec B^*)>0.
	\end{equation}
%corresponding primitive variables $\vec V=(\rho,\vec v,\vec B,p)^{\top}$ satisfies $\rho>0,~p >0$, and $|\vec v|<1$.
	Taking partial derivatives of ${\mathcal H}$ with respect to $\rho$ and $p$ respectively gives
%	\begin{align*}
%	\frac{\partial {\mathcal H}}{\partial \rho } &= (1+\theta v_1)  W^2( 1-  \vec {  v} \cdot {\vec v^*} -  W^{-1} (W^{-1})^* ) \\
%	& \ge (1+\theta v_1)    W^2 \left( 1-  ( v_1, v_2,  v_3,  W^{-1}) \cdot (v_{1}^*,v_{2}^*,v_{3}^*,(W^{-1})^*)  \right) \\
%	& \ge (1+\theta v_1)    W^2 \left( 1-  \left|( v_1, v_2,  v_3,  W^{-1}) \right| \left|   (v_{1}^*,v_{2}^*,v_{3}^*,(W^{-1})^*)  \right| \right) =0,\\
%	\frac{\partial {\mathcal H}}{\partial p } &=
%	(1+\theta v_1) \left ( \frac{\Gamma}{\Gamma-1}  W^2 (1-\vec { v} \cdot \vec v^* ) -1    \right)  \\
%	&\ge  (1+\theta v_1)  \left ( 2  W^2  (1- \left|\vec { v} \right| \left|\vec { v}^* \right|  ) -1    \right) \\
%	&\ge  (1+\theta v_1)  \left ( 2  W^2  (1- \left|\vec { v} \right|   ) -1    \right) \\
%	&= \frac { (1+\theta v_1)   (1-| \vec {v}|)} { 1+| \vec {v}| } > 0.
%	\end{align*}
\begin{align*}
\begin{split}
\frac{\partial {\mathcal H}}{\partial \rho } &= (1+\theta v_1)  W^2( 1-  \vec {  v} \cdot {\vec v^*} -  W^{-1} (W^{-1})^* ) \\
& \ge (1+\theta v_1)    W^2 \left( 1-  ( v_1, v_2,  v_3,  W^{-1}) \cdot (v_{1}^*,v_{2}^*,v_{3}^*,(W^{-1})^*)  \right) \\
& \ge (1+\theta v_1)    W^2 \left( 1-  \left|( v_1, v_2,  v_3,  W^{-1}) \right| \left|   (v_{1}^*,v_{2}^*,v_{3}^*,(W^{-1})^*)  \right| \right) =0,
\end{split}
\\
\begin{split}
\frac{\partial {\mathcal H}}{\partial p } &=
\frac{\Gamma}{\Gamma-1}  (1+\theta v_1) W^2 (1-\vec { v} \cdot \vec v^* ) -   (1+\theta v_1^*) \\
&  \ge 2(1+\theta v_1) W^2 (1-\vec { v} \cdot \vec v^* ) -   (1+\theta v_1^*) \\
&   \ge \mathop {\min } \left\{   H_p^+, H_p^-   \right\} > 0,
\end{split}
\end{align*}
where
\begin{equation*}
\begin{split}
H_p^\pm  :=& 2(1 \pm v_1) W^2 (1-\vec { v} \cdot \vec v^* ) -   (1 \pm v_1^*)
\\[2mm]
=& 2(1 \pm v_1) W^2 - 1 - 2(1 \pm v_1) W^2 \left[ v_1^* \left( v_1 \pm \frac{1}{ 2(1 \pm v_1) W^2   }  \right)  + v_2^* v_2 + v_3^*  v_3 \right] \\[2mm]
\ge &  2(1 \pm v_1) W^2 - 1 - 2(1 \pm v_1) W^2   \left| \vec v^*\right| \sqrt{ \left( v_1 \pm \frac{1}{ 2(1 \pm v_1) W^2   }  \right)^2 + v_2^2 + v_3^2  } \\[2mm]
=& \big( 2(1 \pm v_1) W^2 - 1 \big) \big( 1- \left| \vec v^*\right| \big)
\ge  \big( 2(1 \pm v_1)/(1-v_1^2)  - 1 \big) \big( 1- \left| \vec v^*\right| \big)\\[2mm]
=& 2(1 \pm v_1)^2 \big( 1- \left| \vec v^*\right| \big) / (1-v_1^2) > 0.
\end{split}
\end{equation*}
Here we have used that $|\vec v|<1$ because of $\vec U\in\mathcal G$, and the Cauchy-Schwarz inequality.
Thus, together with $\rho>0,~p >0$, one has
	$$
	{\mathcal H} (\rho,p,\vec v, \vec B,\vec v^*,\vec B^*,\theta )  > {\mathcal H} (0,0,\vec v, \vec B,\vec v^*,\vec B^*,\theta ) =:{\mathcal H}_0 (\vec v, \vec B,\vec v^*,\vec B^*,\theta ).
	$$
The subsequent task is to show that ${\mathcal H}_0 \ge 0$. This
is equivalent to the positive semi-definiteness of a symmetric matrix $\vec {\mathcal A}^H (\vec v,\vec v^*,\theta)$ for any
$\theta \in [-1,1]$  and $\vec v,\vec v^*\in {\mathbb{R}^3}$ with $|\vec v|<1$ and $|\vec v^*|<1$, because  ${\mathcal H}_0 (\vec v, \vec B,\vec v^*,\vec B^*,\theta )$ can be reformulated into a quadratic form of $(\vec B,\vec B^*)$, i.e.
	$$
	{\mathcal H}_0 (\vec v, \vec B,\vec v^*,\vec B^*,\theta ) = \frac{1}{2} (\vec B,\vec B^*) \vec {\mathcal A}^H (\vec v,\vec v^*,\theta)  (\vec B,\vec B^*)^{\top}.
	$$
Here the diagonal  and the upper triangular elements of the symmetric matrix $\vec {\mathcal A}^H = \big[{\mathcal A}^H_{jk} (\vec v,\vec v^*,\theta) \big]_{6\times 6}$ are
\begin{align*}
	&
	{\mathcal A}^H_{11} = {2(1 -v_2^* {v_2}-v_3^* {v_3}) + (1 - \theta v_1^*)(v_2^2 + v_3^2 - 1)},\\
	&
	{\mathcal A}^H_{12} ={{v_1}v_2^* + {v_2}v_1^* - {v_1}{v_2} + \theta (v_3^*{v_2}{v_3} - v_2^*v_3^2 + v_2^* - {v_2} + v_1^*{v_1}{v_2})},\\
	&
	{\mathcal A}^H_{13} = {{v_1}v_3^* + {v_3}v_1^* - {v_1}{v_3} + \theta (v_2^*{v_2}{v_3} - v_3^*v_2^2 + v_3^* - {v_3} + v_1^*{v_1}{v_3})},\\
	&
	{\mathcal A}^H_{14}=\theta(v_{1}^{*}v_{2}v_{2}^{*}+v_{1}^{*}v_{3}v_{3}^{*}-v_{1}^{*})+{v_{2}^{*}}^2+{v_{3}^{*}}^2-1,\\
	&
	{\mathcal A}_{15}^{H}=\theta(v_{2}^{*}v_{3}v_{3}^{*}-v_{2}(v_{3}v_{3}^{*}+v_{1}v_{1}^{*}-1)-v_{2}^{*})-v_{1}^{*}v_{2}^{*},\\
	&
	{\mathcal A}_{16}^{H}=\theta(v_{3}^{*}v_{2}v_{2}^{*}-v_{3}(v_{1}v_{1}^{*}+v_{2}v_{2}^{*}-1)-v_{3}^{*})-v_{1}^{*}v_{3}^{*},\\
%		\end{align*}
%			\begin{align*}
	&	{\mathcal A}_{22}^{H}=\theta(-v_{1}^{*}v_{1}^{2}-2v_{1}v_{3}v_{3}^{*}+2v_{1}+v_{1}^{*}v_{3}^{2}-v_{1}^{*})+v_{1}^{2}-2v_{1}v_{1}^{*}+v_{3}^{2}-2v_{3}v_{3}^{*}+1,\\
	&
	{\mathcal A}_{23}^{H}=\theta(v_{3}^{*}v_{1}v_{2}+v_{2}^{*}v_{1}v_{3}-v_{1}^{*}v_{2}v_{3})+v_{2}^{*}v_{3}^{*}-(v_{2}-v_{2}^{*})(v_{3}-v_{3}^{*}),\\
	&
	{\mathcal A}_{24}^{H}=-(1+\theta v_{1})v_{1}^{*}v_{2}^{*},\ \ \ \
	{\mathcal A}_{25}^{H}=(1+\theta v_{1})( { v_{1}^{*}}^2+ {v_{3}^{*}}^2-1),\\
	&
	{\mathcal A}_{26}^{H}=-(1+\theta v_{1})v_{2}^{*}v_{3}^{*},\\
	&
	{\mathcal A}_{33}^{H}=\theta(-v_{1}^{*}v_{1}^{2}-2v_{1}v_{2}v_{2}^{*}+2v_{1}+v_{1}^{*}v_{2}^{2}-v_{1}^{*})+v_{1}^{2}-2v_{1}v_{1}^{*}+v_{2}^{2}-2v_{2}v_{2}^{*}+1,\\
	&
	{\mathcal A}_{34}^{H}=-(1+\theta v_{1})v_{1}^{*}v_{3}^{*},\ \ \ \
	{\mathcal A}_{35}^{H}=-(1+\theta v_{1})v_{2}^{*}v_{3}^{*},\\
	&
	{\mathcal A}_{36}^{H}=(1+\theta v_{1})({v_{1}^{*}}^2+{v_{2}^{*}}^2-1),\\
	&
	{\mathcal A}_{44}^{H}=-(1+\theta v_{1}^{*})({v_{2}^{*}}^2+{v_{3}^{*}}^2-1),\ \ \ \
	{\mathcal A}_{45}^{H}=(1+\theta v_{1}^{*})v_{1}^{*}v_{2}^{*},\\
	&
	{\mathcal A}_{46}^{H}=(1+\theta v_{1}^{*})v_{1}^{*}v_{3}^{*},\ \ \ \
	{\mathcal A}_{55}^{H}=-(1+\theta v_{1}^{*})({v_{1}^{*}}^2+{v_{3}^{*}}^2-1),\\
	&
	{\mathcal A}_{56}^{H}=(1+\theta v_{1}^{*})v_{2}^{*}v_{3}^{*},\ \ \ \
	{\mathcal A}_{66}^{H}=-(1+\theta v_{1}^{*})({v_{1}^{*}}^2+{v_{2}^{*}}^2-1).
 \end{align*}
If taking the upper triangular  matrix
	\[\vec P = \begin{pmatrix}
	~~1~~ & ~~0~~ & ~~0~~ & ~~1~~ & {\frac{{\theta (v_2^* - {v_2})}}{{1 + \theta v_1^*}}} & {\frac{{\theta (v_3^* - {v_3})}}{{1 + \theta v_1^*}}}  \\
	{} & 1 & 0 & 0 & {\frac{{1 + \theta {v_1}}}{{1 + \theta v_1^*}}} & 0  \\
	{} & {} & 1 & 0 & 0 & {\frac{{1 + \theta {v_1}}}{{1 + \theta v_1^*}}}  \\
	{} & {} & {} & 1 & 0 & 0  \\
	{} & {} & {} & {} & 1 & 0  \\
	{} & {} & {} & {} & {} & 1
	\end{pmatrix},\]
one has
	\[
	\vec P   \vec {\mathcal A}^H (\vec v,\vec v^*,\theta) \vec P^{\top} = {\rm diag} \left\{  \frac{1}{{1 + \theta v_1^*}} \vec {\mathcal B}^H (\vec v,\vec v^*,\theta) ,\vec {\mathcal C}^H (\vec v^*,\theta)  \right\},
	\]
	where $	\vec {\mathcal B}^H $ and $\vec {\mathcal C}^H$ are  two symmetric matrices respectively defined by
	\begin{align*}
	\vec {\mathcal B}^H  = %\left[ {\begin{array}{*{20}{c}}
	\begin{pmatrix}
	{{{\mathcal B}^H_{11}}} & {{{\mathcal B}^H_{12}}} & {{{\mathcal B}^H_{13}}}  \\
	{{{\mathcal B}^H_{21}}} & {{{\mathcal B}^H_{22}}} & {{{\mathcal B}^H_{23}}}  \\
	{{{\mathcal B}^H_{31}}} & {{{\mathcal B}^H_{32}}} & {{{\mathcal B}^H_{33}}}
	%\end{array}} \right]
	\end{pmatrix},%\\[3mm]
	\quad
	\vec {\mathcal C}^H  = (1 + \theta v_1^*)\left[ {\left( {1 - {{\left| {{\vec v^*}} \right|}^2}} \right) \vec I + {\vec v^*}^{\top}{\vec v^*}} \right],
	\end{align*}
with
\begin{align*}
 \begin{split}
	{\mathcal B}^H_{11} & = \left( {1,{v_2},{v_3}} \right)
	\begin{pmatrix}
	{(1 - {\theta ^2})({v{_2^*}^2} + {v{_3^*}^2})} & {( {\theta ^2}-1)v_2^*} & {( {\theta ^2}-1)v_3^*}  \\
	{( {\theta ^2}-1)v_2^*} & {1 - {\theta ^2} + {\theta ^2}v{{_3^*}^2}} & { - {\theta ^2}v_2^*v_3^*}  \\
	{({\theta ^2}-1)v_3^*} & { - {\theta ^2}v_2^*v_3^*} & {1 - {\theta ^2} + {\theta ^2}v{{_2^*}^2}}
	\end{pmatrix}
	{\left( {1,{v_2},{v_3}} \right)^{\top}} \\
	&= :\left( {1,{v_2},{v_3}} \right){ \vec {\hat {\mathcal B} }_{11}}{\left( {1,{v_2},{v_3}} \right)^{\top}},
\end{split}
	\\
 \begin{split}
	{\mathcal B}^H_{22} & = \left( {1,{v_1},{v_3}} \right)
	\begin{pmatrix}
	{(1 - {\theta ^2})v{{_1^*}^2} + v{{_3^*}^2}} & {({\theta ^2} - 1)v_1^* + \theta v{{_3^*}^2}} & { - (1 + \theta v_1^*)v_3^*}  \\
	{({\theta ^2} - 1)v_1^* + \theta v{{_3^*}^2}} & {1 - {\theta ^2} + {\theta ^2}v{{_3^*}^2}} & { - \theta (1 + \theta v_1^*)v_3^*}  \\
	{ - (1 + \theta v_1^*)v_3^*} & { - \theta (1 + \theta v_1^*)v_3^*} & {{{(1 + \theta v_1^*)}^2}}
	\end{pmatrix}
	{\left( {1,{v_1},{v_3}} \right)^{\top}} \\
	&= :\left( {1,{v_1},{v_3}} \right){ \vec {\hat {\mathcal B} }_{22}}{ \left( {1,{v_1},{v_3}} \right)^{\top}},
\end{split}
	\\
\begin{split}
	{\mathcal B}^H_{33} & = \left( {1,{v_1},{v_2}} \right)
	\begin{pmatrix}
	{(1 - {\theta ^2})v{{_1^*}^2} + v{{_2^*}^2}} & {({\theta ^2} - 1)v_1^* + \theta v{{_2^*}^2}} & { - (1 + \theta v_1^*)v_2^*}  \\
	{({\theta ^2} - 1)v_1^* + \theta v{{_2^*}^2}} & {1 - {\theta ^2} + {\theta ^2}v{{_2^*}^2}} & { - \theta (1 + \theta v_1^*)v_2^*}  \\
	{ - (1 + \theta v_1^*)v_2^*} & { - \theta (1 + \theta v_1^*)v_2^*} & {{{(1 + \theta v_1^*)}^2}}  \\
	\end{pmatrix}
	{\left( {1,{v_1},{v_2}} \right)^{\top}} \\
	&= :\left( {1,{v_1},{v_2}} \right){ \vec {\hat {\mathcal B} }_{33}}{ \left( {1,{v_1},{v_2}} \right)^{\top}},
\end{split}
	\\
	{\mathcal B}^H_{12} &= {\mathcal B}^H_{21} = (1,\vec v)
	\begin{pmatrix}
	{({\theta ^2} - 1)v_1^*v_2^*} & {(1 - {\theta ^2})v_2^*} & {(1 - {\theta ^2})v_1^* - \theta v{{_3^*}^2}} & {\theta v_2^*v_3^*}  \\
	{} & 0 & {{\theta ^2}(1 - v{{_3^*}^2}) - 1} & {{\theta ^2}v_2^*v_3^*}  \\
	{} & {} & 0 & {(1 + \theta v_1^*)\theta v_3^*}  \\
	{} & {} & {} & { - (1 + \theta v_1^*)\theta v_2^*}
	\end{pmatrix}
	{(1,\vec v)^{\top}},
	\\
	{\mathcal B}^H_{13}& = {\mathcal B}^H_{31} = (1,\vec v)
	\begin{pmatrix}
	{({\theta ^2} - 1)v_1^*v_3^*} & {(1 - {\theta ^2})v_3^*} & {\theta v_2^*v_3^*} & {(1 - {\theta ^2})v_1^* - \theta v{{_2^*}^2}}  \\
	{} & 0 & {{\theta ^2}v_2^*v_3^*} & {{\theta ^2}(1 - v{{_2^*}^2}) - 1}  \\
	{} & {} & { - (1 + \theta v_1^*)\theta v_3^*} & {(1 + \theta v_1^*)\theta v_2^*}  \\
	{} & {} & {} & 0  \\
	\end{pmatrix}
	{(1,\vec v)^{\top}},
	\\
	{\mathcal B}^H_{23} &={\mathcal B}^H_{32} =  - ({v_2} - v_2^* - \theta {v_1}v_2^* + \theta {v_2}v_1^*)({v_3} - v_3^* - \theta {v_1}v_3^* + \theta {v_3}v_1^*).
\end{align*}
Under the hypothesis, one has that $1 + \theta v_1^*\ge 1-|v_1^*| >0$
and thus the matrix $\vec {\mathcal C}^H (\vec v^*,\theta)$ is  positive definite. Therefore,
the subsequent task is to show the positive semi-definiteness of the symmetric matrix  $\vec {\mathcal B}^H$,
	or equivalently, the non-negativity of all principal minors of $\vec {\mathcal B}^H$.
%	However, it is still a hard task to directly check the sign of those principal minors of $\vec {\mathcal B}^H$.
	It is observed that {these minors} can be estimated through the quadratic forms of $(1,\vec v)^{\top}$.
	First check the first-order principal minors of $\vec {\mathcal B}^H$. If taking
	\[{\vec P_1} = \begin{pmatrix}
	1~ & ~{v_2^*}~ & ~{v_3^*}  \\
	{} & ~1~ & ~0 \\
	{} & {} & ~1
	\end{pmatrix},\]
	one has
	\[{\vec P_1}{ \vec {\hat{\mathcal B}_{11}}}{\vec P_1^{\top}} =
	%\left[ {\begin{array}{*{20}{c}}
	\begin{pmatrix}   0 & 0 & 0  \\
	0 & {1 - {\theta ^2} + {\theta ^2}v{{_3^*}^2}} & { - {\theta ^2}v_2^*v_3^*}  \\
	0 & { - {\theta ^2}v_2^*v_3^*} & {1 - {\theta ^2} + {\theta ^2}v{{_2^*}^2}}
	\end{pmatrix}. %\end{array}} \right].
	\]
Then using
\begin{equation*}
 \begin{split}
	&
	1 - {\theta ^2} + {\theta ^2}v{{_3^*}^2}  \ge 0,\quad 1 - {\theta ^2} + {\theta ^2}v{{_2^*}^2}  \ge 0,\\
	&\det \left( {\begin{array}{*{20}{c}}
		{1 - {\theta ^2} + {\theta ^2}v{{_3^*}^2}} & { - {\theta ^2}v_2^*v_3^*}  \\
		{ - {\theta ^2}v_2^*v_3^*} & {1 - {\theta ^2} + {\theta ^2}v{{_2^*}^2}}  \\
		\end{array}} \right) = (1 - {\theta ^2})\left[ {1 - {\theta ^2} + {\theta ^2}(v{{_2^*}^2} + v{{_3^*}^2})} \right] \ge 0,
 \end{split}
\end{equation*}
yields the positive semi-definiteness of
 the matrix ${ \vec {\hat{\mathcal B}_{11}}}$, which follows that ${\mathcal B}^H_{11} \ge 0$.
	Similarly, one has ${\mathcal B}^H_{22} \ge 0 $ and ${\mathcal B}^H_{33} \ge 0$.
	Next we consider the second-order principal minors of $\vec {\mathcal B}^H$. Some algebraic manipulations yield
\begin{align*}
	&
	\det\left( {\begin{array}{*{20}{c}}
		{{{\mathcal B}^H_{11}}} & {{{\mathcal B}^H_{12}}}  \\
		{{{\mathcal B}^H_{21}}} & {{{\mathcal B}^H_{22}}}
		\end{array}} \right) = (v_3 -v_3^*)^2 \Xi ,\quad
	\det\left( {\begin{array}{*{20}{c}}
		{{{\mathcal B}^H_{11}}} & {{{\mathcal B}^H_{13}}}  \\
		{{{\mathcal B}^H_{31}}} & {{{\mathcal B}^H_{33}}}
		\end{array}} \right) =(v_2 -v_2^*)^2 \Xi,\\
	&
	\det\left( {\begin{array}{*{20}{c}}
		{{{\mathcal B}^H_{22}}} & {{{\mathcal B}^H_{23}}}  \\
		{{{\mathcal B}^H_{32}}} & {{{\mathcal B}^H_{33}}}
		\end{array}} \right) = (v_1 -v_1^*)^2 \Xi,
\end{align*}
	where
	\begin{align*}
	\Xi = (1-\theta^2) {\vec z}_3 ^{\top}  {\rm diag}
	\left\{  (1-\theta^2)
	\begin{pmatrix}
	{v{{_1^*}^2}} & { - v_1^*}  \\
	{ - v_1^*} & 1
	\end{pmatrix},(1 + \theta v_1^*)^2 \vec I_2,
	\right\} {\vec z}_3 ,
	\end{align*}
	with
	\[\vec z = \left( {1,{v_1},{v_2} - \frac{{(1 + \theta {v_1})v_2^*}}{{1 + \theta v_1^*}},{v_3} - \frac{{(1 + \theta {v_1})v_3^*}}{{1 + \theta v_1^*}}} \right)^{\top}.\]
	It is not difficult to know that $\Xi \ge 0$ by noting the positive semi-definiteness of
	$$\begin{pmatrix}
	{v{{_1^*}^2}} & { - v_1^*}  \\
	{ - v_1^*} & 1
	\end{pmatrix}.$$
	Therefore  all three second-order principal minors of $\vec {\mathcal B}^H$ are non-negative.
	Finally we consider the third-order principal minor of $\vec {\mathcal B}^H$, i.e. $\det(\vec {\mathcal B}^H)$. Some algebraic manipulations yield
\begin{align*}
	&
	\det\left( {\begin{array}{*{20}{c}}
		{{{\mathcal B}^H_{21}}} & {{{\mathcal B}^H_{23}}}  \\
		{{{\mathcal B}^H_{31}}} & {{{\mathcal B}^H_{33}}}
		\end{array}} \right) = (v_1^* -v_1)(v_2 -v_2^*)   \Xi ,\quad
	\det\left( {\begin{array}{*{20}{c}}
		{{{\mathcal B}^H_{21}}} & {{{\mathcal B}^H_{22}}}  \\
		{{{\mathcal B}^H_{31}}} & {{{\mathcal B}^H_{32}}}
		\end{array}} \right) =(v_1 -v_1^*)(v_3 -v_3^*) \Xi,\\
	&
	\det\left( {\begin{array}{*{20}{c}}
		{{{\mathcal B}^H_{11}}} & {{{\mathcal B}^H_{13}}}  \\
		{{{\mathcal B}^H_{21}}} & {{{\mathcal B}^H_{23}}}
		\end{array}} \right) = (v_2^* -v_2)(v_3 -v_3^*) \Xi,
\end{align*}
	Based on those first-  and second-order principal minors %in \eqref{eq:2nd:minor}
	of the symmetric matrix $\vec {\mathcal B}^H$, one obtains
	the adjoint matrix of $\vec {\mathcal B}^H$
	\begin{align*}
    \begin{split}
	{\rm adj} \left( \vec {\mathcal B}^H \right) &= \Xi \begin{pmatrix}
	(v_1 -v_1^*)^2 & (v_1 -v_1^*)(v_2 -v_2^*) & (v_1 -v_1^*)(v_3 -v_3^*)  \\
	(v_1 -v_1^*)(v_2 -v_2^*) & (v_2 -v_2^*)^2 & (v_2 -v_2^*)(v_3 -v_3^*)  \\
	(v_1 -v_1^*)(v_3 -v_3^*) & (v_2 -v_2^*)(v_3 -v_3^*) & (v_3 -v_3^*)^2
	\end{pmatrix}
	\\
	&
	= \Xi  (\vec v - \vec v^*)^{\top} (\vec v - \vec v^*),
    \end{split}
	\end{align*}
	which is also a symmetric matrix of size 3 and has rank of at most one, such that
	 ${\rm adj} \left( \vec {\mathcal B}^H \right)$ and  $\vec {\mathcal B}^H$ are
	irreversible and  $\det(\vec {\mathcal B}^H)=0$.
	In conclusion, the matrix $\vec {\mathcal B}^H$ is positive semi-definite, and  the inequality
	\eqref{eq:RMHD:LLFsplit} for the case of $i=1$, i.e. \eqref{eq:RMHD:LLFsplit-thz}, does hold.
	
{\tt (ii)}. The inequality
\eqref{eq:RMHD:LLFsplit} for the case of $i=2$ or 3 can be verified by using \eqref{eq:RMHD:LLFsplit-thz} and the
 orthogonal invariance in Corollary \ref{lem:RMHD:zhengjiao}.

%The following will use \eqref{eq:RMHD:LLFsplit}  for $i=1$ to show  \eqref{eq:RMHD:LLFsplit}
For the case of $i=2$, we introduce a symmetric matrix $\vec T = {\rm diag} \{ 1, \vec T_3, \vec T_3, 1\}$ with
the orthogonal matrix
	$$
	{\vec T_3} =  \begin{pmatrix}
	0~ & ~1~ & ~0  \\
	1~ & ~0~ & ~0 \\
	0~ & ~0~ & ~1
	\end{pmatrix}.
	$$
Regarding the conservative vector $\vec U$ as a vector function of the primitive variables $\vec V$,
	i.e. $\vec U(\vec V)$, then one has
	$\vec U(\vec T \vec V)=\vec T \vec U=:\tilde{ \vec U}$, $\vec F_1(\vec U(\vec T\vec V)) = \vec T \vec F_2 (\vec U)$,
	and
	$$\vec F_1 (\tilde{ \vec U} ) = \vec F_1 (\vec T\vec U )=  \vec F_1(\vec U(\vec T\vec V)) = \vec T \vec F_2 (\vec U).$$
	Thanks to Corollary \ref{lem:RMHD:zhengjiao}, one obtains $\tilde {\vec U} \in {\mathcal G}$.
Let $\tilde{\vec v}^* = \vec  v^* \vec T_3$ and $\tilde{\vec B}^* =   \vec B^* \vec T_3$, then
	$$|\vec {\tilde v}^*| = |\vec v|<1,\quad |\vec {\tilde B}^*| = |\vec B^*|, \quad  \vec {\tilde v}^* \cdot \vec {\tilde B}^* = \vec v^* \cdot \vec B^*. $$
	It follows from \eqref{eq:RMHD:vecns} and \eqref{eq:RMHD:vecns2} that $\tilde p_m^* =   p_m^*$ and $\tilde {\vec n}^* = \vec T {\vec n}^*$.
	Using the inequality  \eqref{eq:RMHD:LLFsplit-thz} with $\tilde{ \vec U} \in{\mathcal G},~\tilde {\vec v}^*$, and $\tilde {\vec B}^*$ gives
	\begin{equation*}
    \begin{split}
	0 & < \big( \tilde {\vec U} + \theta \vec F_1(\tilde { \vec U}) \big) \cdot \tilde {\vec n}^* +  \tilde p_{m}^* + \theta \big( \tilde v_1^* \tilde p_{m}^* - \tilde B_1 ( \tilde {\vec v}^* \cdot \tilde {\vec B}^*)\big) \\
	%& = \big( \vec T {\vec U} + \theta \vec F_1({ \vec T \vec U}) \big) \cdot  \left( \vec T {\vec n}^* \right) +   p_{m}^* + \theta \big(  v_2^*  p_{m}^* - B_2 (  {\vec v}^* \cdot  {\vec B}^*)\big)\\
	&
	= \big( \vec T {\vec U} + \theta  \vec T \vec F_2({ \vec U}) \big) \cdot  \left( \vec T {\vec n}^* \right) +   p_{m}^* + \theta \big(  v_2^*  p_{m}^* - B_2 (  {\vec v}^* \cdot  {\vec B}^*)\big)
	\\
	&
	= \big(  {\vec U} + \theta  \vec F_2({ \vec U}) \big) \cdot    {\vec n}^*  +   p_{m}^* + \theta \big(  v_2^*  p_{m}^* - B_2 (  {\vec v}^* \cdot  {\vec B}^*)\big),
    \end{split}
	\end{equation*}
	where the orthogonality of the matrix $\vec T$ has been used in the last equality.
	This verifies the inequality \eqref{eq:RMHD:LLFsplit} for the case of $i=2$.
	The case of $i=3$ can be  similarly derived  by taking the orthogonal matrix
	$$
	{\vec T_3} =  \begin{pmatrix}
	0~ & ~0~ & ~1  \\
	0~ & ~1~ & ~0  \\
	1~ & ~0~ & ~0
	\end{pmatrix}.
	$$
	The proof is completed.
	\end{proof}

\begin{remark}
	Thanks to to the second equivalent form of $\mathcal G$, Lemma \ref{lemma:2.8} tells us that
%	the absence of Lax-Friedrichs splitting property implies that, for $\vec U \in {\mathcal G}$,
the inequality
	\begin{equation}\label{eq:RMHD:LLFsplitNO}
	\big( \vec U + \theta \vec F_i(\vec U) \big) \cdot \vec n^* +  p_{m}^* > 0,
	\end{equation}
does	not always hold for any  $\theta \in [-1,1]$ and $ {\vec B}^*, \vec v^* \in \mathbb{R}^3$ with $ |\vec v^*|<1$,
where $i\in \{1,2,3\}$.
Compared \eqref{eq:RMHD:LLFsplitNO} to the inequality  \eqref{eq:RMHD:LLFsplit},
 the third term at the left-hand side of \eqref{eq:RMHD:LLFsplit} is extremely technical and crucial in deriving
the generalized LxF splitting properties.
Although {this term} is not always positive or negative, it can be canceled out dexterously
with the help of the following ``discrete divergence-free'' condition \eqref{eq:descrite1DDIV} or \eqref{eq:descrite2DDIV},
see  the proofs of  following theorems. % \ref{theo:RMHD:LLFsplit1D} and \ref{theo:RMHD:LLFsplit2D}.

%{\color{red}We emphasize again that the construction of this term is extremely non-trivial, any attempts to guess this term is
%very difficult to verified.}

%	We constructively add the third term to remedy the insufficient part of the first two terms in the left-hand side of  \eqref{eq:RMHD:LLFsplit}.
%	Note this ``constructive'' term is not always positive or negative, but it is necessary and
%	can exactly remedy the insufficiency of \eqref{eq:RMHD:LLFsplitNO}.
%	Moreover, as it will be seen in the following proofs of Theorem \ref{theo:RMHD:LLFsplit1D} and \ref{theo:RMHD:LLFsplit2D}, this ``constructive'' term will
%	be canceled out dexterously under the ``discrete divergence-free'' condition \eqref{eq:descrite1DDIV} and \eqref{eq:descrite2DDIV}, respectively.
%	This idea is a key in proving the generalized Lax-Friedrichs splitting properties.
%	However, constructing the third term in the left-hand side of \eqref{eq:RMHD:LLFsplit} is non-trivial.
%	Before getting the final result, one can only guess that this term is $g(\theta,\vec v^*,\vec B^*,B_i)$, which is
%	expected to be a odd function with respect to $\theta$ and a quadratic form of $(\vec B^*,B_i)$. For each possible attempt for $g$, it is difficult to check
%	whether it is satisfying. This highlights the result.
\end{remark}

Based on the above lemma,  we derive the following  generalized   LxF splitting properties.

\begin{theorem}[1D generalized LxF splitting]\label{theo:RMHD:LLFsplit1D}
If $\tilde {\vec U}=(\tilde D, \tilde{\vec m}, \tilde{\vec B}, \tilde{E})^{\top}\in {\mathcal G}$
and $\hat{\vec U}=(\hat D, \hat{\vec m}, \hat{\vec B}, \hat{E})^{\top} \in {\mathcal G}$ satisfy
1D ``discrete divergence-free'' condition
\begin{equation}\label{eq:descrite1DDIV}
\tilde {B_i} - \hat {B_i}=0,
\end{equation}
for a given $i\in \{1,2,3\}$, then for any  $\alpha \ge c=1$  it holds
\begin{equation}\label{eq:RMHD:LLFsplit1D}
\bar{\vec U}:=\frac{1}{2} \Big(
\tilde { \vec U } -  \alpha^{-1} { \vec F_i ( \tilde { \vec U})  }
+\hat{
\vec U} +  \alpha^{-1} { \vec F_i (\hat{\vec U})}
\Big)
\in {\mathcal G}.
\end{equation}
\end{theorem}
%{\bf Proof of Theorem \ref{theo:RMHD:LLFsplit1D}.}
\begin{proof}
	It is obvious that  $$
\frac{1}{2} \left(  \tilde D ( 1 - \tilde v_i/\alpha ) +  \hat D ( 1 + \hat v_i / \alpha ) \right) >0,
$$
that is to say, the first component of $\bar{\vec U}$ satisfies
 the first constraint in ${\mathcal G}_1$, see Theorem \ref{theo:RMHD:CYcondition:VecN}.

Next, let us check the second constraint in ${\mathcal G}_1$.
For any ${\vec B}^*, \vec v^* \in \mathbb{R}^3$ with $  v^*<1$, using Lemma \ref{theo:RMHD:LLFsplit} gives
\begin{equation*}
\begin{split}
\bar{\vec U} \cdot \vec n^* + p_{m}^*
%& =
%& \frac{1}{2} \Big(
%\tilde { \vec U } -  \alpha^{-1} { \vec F_i ( \tilde { \vec U})  }
%+\hat{
%	\vec U} +  \alpha^{-1} { \vec F_i (\hat{\vec U})}
%\Big) \cdot \vec n^* + p_{m}^*
%\\[2mm]
& =  \frac{1}{2} \Big( \left(
\tilde { \vec U } -  \alpha^{-1} { \vec F_i ( \tilde { \vec U})  }
\right)
\cdot \vec n^* + p_{m}^*\Big) +  \frac{1}{2} \Big( \left(
\hat{
	\vec U} +  \alpha^{-1} { \vec F_i (\hat{\vec U})} \right)
\cdot \vec n^* + p_{m}^*\Big)\\[2mm]
& \overset{\eqref{eq:RMHD:LLFsplit}}{>}
-  \frac{1}{2}\alpha^{-1} \Big(  \tilde B_i (\vec v^* \cdot \vec B^*) -  v_{i}^* p_{m}^* \Big)
+ \frac{1}{2}\alpha^{-1} \Big( \hat B_i (\vec v^* \cdot \vec B^*) -  v_{i}^* p_{m}^* \Big)
\overset{\eqref{eq:descrite1DDIV}}{=}
0,
\end{split}
\end{equation*}
where $\vec n^*$ and $p_{m}^*$ are defined in \eqref{eq:RMHD:vecns} and \eqref{eq:RMHD:vecns2}, respectively.
%\begin{align*}%\label{eq:RMHD:vecns}
%&\vec n^* = {\left( - \sqrt {1 - |\vec v^*{|^2}} , - \vec v^*, - (1 - {|\vec v^*|^2}) \vec B^* - (\vec v^* \cdot \vec B^*) \vec v^*,1 \right)^{\top}},\\
%& p_{m}^*  = \frac{ (1-|\vec v^*|^2)|\vec B^*|^2 +(\vec v^* \cdot \vec B^*)^2 }{2}.
%\end{align*}
Using Theorem \ref{theo:RMHD:CYcondition:VecN}  completes the proof.
\end{proof}

\begin{theorem}[2D generalized LxF splitting]\label{theo:RMHD:LLFsplit2D}
If $\tilde{\vec U}^{i}$, $\hat{\vec U}^i$,
$\bar{\vec U}^i$,
$\breve{\vec U}^i
\in {\mathcal G}$ for $i=1,2,\cdots,L$ satisfy 2D  ``discrete divergence-free'' condition
\begin{equation}\label{eq:descrite2DDIV}
\frac{{\sum\limits_{i=1}^L {{\omega _i}({\tilde B_1}^i - {\hat B_1}^i)} }}{{\Delta x}}
+ \frac{{\sum\limits_{i=1}^L {{\omega _i}({\bar B_2}^i - {\breve B_2}^i)} }}{{\Delta y}}
=0,
\end{equation}
where $\Delta x, \Delta y >0$, and the sum of all positive numbers
$\left\{\omega _i \right\}_{i=1}^L$ is equal to 1,
then  for all $\alpha \ge c=1$ it holds
\begin{equation} \label{eq:RMHD:LLFsplit2D}
\begin{split}
\bar{\vec U}:=&\frac{1}{2 \left(  \frac{1}{\Delta x} + \frac{1}{\Delta y} \right)}
\sum\limits_{i=1}^L {{\omega _i}}
\bigg[
\frac{1}{\Delta x} \Big(
\tilde { \vec U }^i -  \alpha^{-1} { \vec F_1 ( \tilde { \vec U}^i)  }
+\hat{
\vec U}^i +  \alpha^{-1} { \vec F_1 (\hat{\vec U}^i)}
\Big)\\
&\quad \quad \quad \quad \quad
+
\frac{1}{\Delta y} \Big(
\bar { \vec U }^i -  \alpha^{-1} { \vec F_2 ( \bar { \vec U}^i)  }
+\breve{
\vec U}^i +  \alpha^{-1} { \vec F_2 (\breve{\vec U}^i)}
\Big)
\bigg]
\in {\mathcal G}.
\end{split}
\end{equation}
\end{theorem}

\begin{proof}
The first component of $\bar{\vec U}$ satisfies
	 the first constraint in ${\mathcal G}_1$, i.e.
\begin{align*}
\begin{split}
&  \frac{1}{2 \left(  \frac{1}{\Delta x} + \frac{1}{\Delta y} \right)}
\sum\limits_{i=1}^L {{\omega _i}}
\bigg[
\frac{1}{\Delta x} \Big(
\tilde { D }^i (1 -  \alpha^{-1} { {\tilde v_1}^i  } )
+\hat{
	D }^i (1 +  \alpha^{-1} {{\hat v_1}^i} )
\Big)\\
& \quad \quad \quad \quad \quad \quad
+
\frac{1}{\Delta y} \Big(
\bar { D }^i ( 1 -  \alpha^{-1} { { \bar v_2}^i  } )
+\breve{
	D}^i (1 +  \alpha^{-1} {{\breve v_2}^i)}
\Big)
\bigg] >0.
\end{split}
\end{align*}

For any ${\vec B}^*,\vec v^* \in \mathbb{R}^3$ with $|\vec v^*|<1$, utilizing Lemma \ref{theo:RMHD:LLFsplit} and
\eqref{eq:descrite2DDIV} gives
\begin{equation*}
\begin{split}
&\bar{\vec U} \cdot \vec n^* + p_{m}^* \\
& =
\frac{1}{2 \left(  \frac{1}{\Delta x} + \frac{1}{\Delta y} \right)}
\sum\limits_{i=1}^L  {{\omega _i}}
\bigg[
\frac{1}{\Delta x} \bigg( \Big(
\tilde { \vec U }^i -  \alpha^{-1} { \vec F_1 ( \tilde { \vec U}^i)  } \Big) \cdot \vec n^* + p_{m}^*
+\Big( \hat{
	\vec U}^i +  \alpha^{-1} { \vec F_1 (\hat{\vec U}^i)}
\Big) \cdot \vec n^* + p_{m}^* \bigg) \\[2mm]
&\quad \quad \quad
+
\frac{1}{\Delta y} \bigg( \Big(
\bar { \vec U }^i -  \alpha^{-1} { \vec F_2 ( \bar { \vec U}^i)  } \Big) \cdot \vec n^* + p_{m}^*
+ \Big( \breve{
	\vec U}^i +  \alpha^{-1} { \vec F_2 (\breve{\vec U}^i)}
\Big) \cdot \vec n^* + p_{m}^* \bigg)
\bigg]\\[2mm]
& \overset{\eqref{eq:RMHD:LLFsplit}}{>}
\frac{1}{2 \left(  \frac{1}{\Delta x} + \frac{1}{\Delta y} \right)}
\sum\limits_{i=1}^L  {{\omega _i}}
\bigg[
\frac{1}{\Delta x} \bigg(
-\alpha^{-1} \Big( \tilde B_1^i (\vec v^* \cdot \vec B^*) -  v_{1}^* p_{m}^* \Big)
+ \alpha^{-1} \Big( \hat B_1^i (\vec v^* \cdot \vec B^*) -  v_{1}^* p_{m}^* \Big)  \bigg) \\[2mm]
&\quad \quad \quad
+
\frac{1}{\Delta y} \bigg(
-\alpha^{-1} \Big( \bar B_2^i (\vec v^* \cdot \vec B^*) -  v_{2}^* p_{m}^* \Big)
+ \alpha^{-1} \Big( \breve B_2^i (\vec v^* \cdot \vec B^*) -  v_{2}^* p_{m}^* \Big)  \bigg)
\bigg]
\\[2mm]
& = - \frac{ \vec v^* \cdot \vec B^*  }{2 \alpha \left(  \frac{1}{\Delta x} + \frac{1}{\Delta y} \right)}
\sum\limits_{i=1}^L   {{\omega _i}} \left( \frac{{{\tilde B_1}^i - {\hat B_1}^i} }{{\Delta x}} + \frac{ {\bar B_2}^i - {\breve B_2}^i} {{\Delta y}} \right) \overset{\eqref{eq:descrite2DDIV}}{=} 0. \nonumber%\label{eq:WKL20161226}
\end{split}
\end{equation*}
Thus {$\bar{\vec U}$} also satisfies the second constraint in  ${\mathcal G}_1$. Using Theorem \ref{theo:RMHD:CYcondition:VecN} completes the proof.
\end{proof}

%The equations \eqref{eq:descrite1DDIV} and \eqref{eq:descrite2DDIV} are
 %``discrete divergence free'' conditions of $\vec B$ in 1D and 2D respectively.

%Especially, the expressions of $\vec U$ involve many variables with constraints, and $\vec F_i$ in the expressions are highly nonlinear functions and can not be explicitly expressed in term of conservative variables. It is also a difficulty to appropriately utilize the ``discrete divergence-free'' condition in the proof.
%A breakthrough of the proof comes from the discovery of the following lemma, where a constructive inequality is established. Before providing the proof of Theorem \ref{theo:RMHD:LLFsplit1D} and \ref{theo:RMHD:LLFsplit2D}, we first introduce this important lemma.

%The following gives the proofs of Theorem \ref{theo:RMHD:LLFsplit1D} and \ref{theo:RMHD:LLFsplit2D} by using Lemma \ref{theo:RMHD:LLFsplit}.

\begin{theorem}[3D generalized LxF splitting]\label{theo:RMHD:LLFsplit3D}
If $\tilde {
		\vec U}^{i}$, $\hat{\vec U}^i$,
	$\bar{\vec U}^i$,
	$\breve{\vec U}^i$, $\bar { \bar{\vec U} }^i$,
	$\breve{ \breve{\vec U} }^i
	\in {\mathcal G}$ for $i=1,2,\cdots, L$, and they satisfy the 3D ``discrete divergence-free'' condition
	\begin{equation*}%\label{eq:descrite3DDIV}
	\frac{{\sum\limits_{i=1}^L {{\omega _i}({\tilde B_1}^i - {\hat B_1}^i)} }}{{\Delta x}} + \frac{{\sum\limits_{i=1}^L {{\omega _i}({\bar B_2}^i - {\breve B_2}^i)} }}{{\Delta y}}
	+ \frac{{\sum\limits_{i=1}^L {{\omega _i} \left( \bar {\bar {B_3^i} } - \breve {\breve {B_3^i} } \right)} }}{{\Delta z}}
	=0,
	\end{equation*}
where $\Delta x, \Delta y, \Delta z >0$, and the sum of all positive numbers
$\left\{\omega _i \right\}_{i=1}^L$ is equal to 1,
	then for any  $\alpha \ge c=1$ it holds
	\begin{equation*}
    \begin{split}
	&\frac{1}{2 \left(  \frac{1}{\Delta x} + \frac{1}{\Delta y} + \frac{1}{\Delta z} \right)}
	\sum\limits_{i=1}^L {{\omega _i}}
	\bigg[
	\frac{1}{\Delta x} \Big(
	\tilde { \vec U }^i -  \alpha^{-1} { \vec F_1 ( \tilde { \vec U}^i)  }
	+\hat{
		\vec U}^i +  \alpha^{-1} { \vec F_1 (\hat{\vec U}^i)}
	\Big)\\
	&\quad \quad \quad \quad \quad
	+
	\frac{1}{\Delta y} \Big(
	\bar { \vec U }^i -  \alpha^{-1} { \vec F_2 ( \bar { \vec U}^i)  }
	+\breve{
		\vec U}^i +  \alpha^{-1} { \vec F_2 (\breve{\vec U}^i)}
	\Big)
	\\
	&\quad \quad \quad \quad \quad
	+
	\frac{1}{\Delta z} \Big(
	\bar{\bar { \vec U }}^i -  \alpha^{-1} { \vec F_3 \big( \bar{\bar { \vec U}}^i \big)  }
	+\breve {\breve{
			\vec U}}^i +  \alpha^{-1} { \vec F_3 \big( \breve {\breve{\vec U}}^i\big)}
	\Big)
	\bigg]
	\in {\mathcal G}.
    \end{split}
	\end{equation*}
\end{theorem}
\begin{proof}
	The proof is similar to that of Theorem \ref{theo:RMHD:LLFsplit2D} and omitted here.
	\end{proof}

\begin{remark}
	%	{\color{red}
	Because the convex combination $\bar{\vec U}$ in the above
	generalized LxF splitting properties depends on
	several	states, it is very difficult to directly check whether $\bar{\vec U}$ belongs to the set $\mathcal G$.
	It is subtly and fortunately overcame by  using the inequality \eqref{eq:RMHD:LLFsplit} in Lemma \ref{theo:RMHD:LLFsplit} and  the ``discrete divergence-free'' condition, which is
	an approximation to \eqref{eq:2D:BxBy0}. For example,
	the ``discrete divergence-free'' condition \eqref{eq:descrite2DDIV}  can be derived
	by  using some quadrature rule for the integrals at the left hand side of
	\begin{equation} \label{eq:div000}
    \begin{split}
	& \frac{1}{{\Delta x}} \left( \frac{1}{{\Delta y}}\int_{y_0 }^{y_0  + \Delta y} { \big( B_1 (x_0  + \Delta x,y) - B_1 (x_0 ,y) \big) dy} \right)  \\
	&+ \frac{1}{{\Delta y}} \left( \frac{1}{{\Delta x}}\int_{x_0 }^{x_0  + \Delta x} {  \big( B_2 (x,y_0  + \Delta y)-B_2 (x,y_0 )\big)   dx} \right)   \\
	& = \frac{1}{{\Delta x\Delta y}}\int_{I} {\left( {\frac{{\partial B_1 }}{{\partial x}} + \frac{{\partial B_2 }}{{\partial y}}} \right)dxdy}=0,
	\end{split}
    \end{equation}
	where $(x,y)=(x_1,x_2)$ and $I := [x_0 ,x_0  + \Delta x] \times [y_0 ,y_0  + \Delta y] $.
	% is any rectangle with length $\Delta x$ and width $\Delta y$ in 2D spatial plane, and \eqref{eq:div000} is a weak form of the divergence free condition \eqref{eq:2D:BxBy0} for $d=2$.
\end{remark}

The above generalized LxF splitting properties %\ref{theo:RMHD:LLFsplit1D} and \ref{theo:RMHD:LLFsplit2D}
are important tools in developing and analyzing the PCP numerical schemes on uniform meshes
if the numerical flux is taken as the LxF type flux \eqref{eq:LFflux}.
Moreover, it is easy to extend them on non-uniform or unstructured meshes.
For example,  the following	theorem  shows %\ref{theo:RMHD:LLFsplit2Dus}
an extension to the  case of 2D arbitrarily convex polygon mesh.
%Besides, the generalized Lax-Friedrichs splitting property can also be extended to three-dimensional case, see e.g. the version shown in Theorem \ref{theo:RMHD:LLFsplit3D}.

\begin{theorem}\label{theo:RMHD:LLFsplit2Dus} %[2D generalized LxF splitting]
If for $i=1,2,\cdots,L$ and $j=1,2,\cdots,J$, ${\vec U}^{ij}\in {\mathcal G}$ and satisfy 2D  ``discrete divergence-free'' condition over
an $J$-sided convex polygon
\begin{equation}\label{eq:descrite2DDIVus}
\sum\limits_{j = 1}^J {\left[ {\sum\limits_{i = 1}^L {\omega _i \left( { B_1^{ij} {\mathcal N}_1^j  + B_2^{ij} {\mathcal N}_2^j  } \right)} } \right]} \ell_j = 0,
\end{equation}
where $\ell_j>0$ and $\left( {\mathcal N}_1^j, {\mathcal N}_2^j \right)$ are the length and the unit outward normal vector of the $j$-th edge of the polygon, respectively, and the sum of all positive numbers
$\left\{\omega _i \right\}_{i=1}^L$ is equal to 1,
then  for all $\alpha \ge c=1$ it holds
\begin{align*}
\bar{\vec U}: = \frac{1}{{\sum\limits_{j = 1}^J {\ell_j } }}\sum\limits_{j = 1}^J {\left[ {\sum\limits_{i = 1}^L {\omega _i \bigg( {{\vec U}^{ij}  - \alpha ^{ - 1} \Big(\vec F_1 (\vec U^{ij} ) {\mathcal N}_1^j  + \vec F_2 (\vec U^{ij} ) {\mathcal N}_2^j \Big)} \bigg)} } \right]} \ell_j  \in {\mathcal G}.
%\label{eq:RMHD:LLFsplit2Dus}
\end{align*}
\end{theorem}
\begin{proof}
%\color{red}
The rotational invariance property of the 2D RMHD equations \eqref{eq:RMHD1D} yields
$$
\vec F_1 (\vec U^{ij} ) {\mathcal N}_1^j  + \vec F_2 (\vec U^{ij} ) {\mathcal N}_2^j = \vec T^{-1}_j \vec F_1(\vec T_j {\vec U}^{ij}),
$$
where $\vec T_j:={\rm diag} \left\{1,\vec T_{3,j},\vec T_{3,j},1 \right\}$ with the rotational matrix $\vec T_{3,j}$ defined by
	$$
	{\vec T}_{3,j} :=  \begin{pmatrix}
	{\mathcal N}_1^j~ & ~{\mathcal N}_2^j~ & ~0  \\
	-{\mathcal N}_2^j~ & ~{\mathcal N}_1^j~ & ~0 \\
	0~ & ~0~ & ~1
	\end{pmatrix}.
	$$
For each $j$ and any $\vec B^*,\vec v^*\in {\mathbb{R}}^3$ with $|\vec v^*|<1$, let $\hat {\vec v}^*= {\vec v}^* \vec T_{3,j}$ and $\hat {\vec B}^*= {\vec B}^* \vec T_{3,j}$, one has $|\vec {\hat v}^*| = |\vec v|<1,~\vec {\hat v}^* \cdot \vec {\hat B}^* = \vec v^* \cdot \vec B^*,~\hat p_m^* =   p_m^*$, and $\hat {\vec n}^* = \vec T_{j} {\vec n}^*$. Utilizing Lemma \ref{theo:RMHD:LLFsplit} for $\vec T_j \vec U^{ij}$, $\hat {\vec v}^*$,  and $\hat {\vec B}^*$ gives
\begin{equation}\label{eq:wkl000}
\begin{split}
0 &< \left( \vec T_j \vec U^{ij} - \alpha^{-1}  \vec F_1(\vec T_j {\vec U}^{ij}) \right) \cdot \hat{\vec n}^* + \hat p_m^* - \alpha^{-1} \big( \hat v_1^* \hat p_m^* - \left( B_1^{ij} {\mathcal N}_1^j  + B_2^{ij} {\mathcal N}_2^j \right) (\hat {\vec v}^* \cdot \hat {\vec B}^* ) \big) \\
& = \left(  \vec U^{ij} - \alpha^{-1} \vec T_j^{-1}  \vec F_1(\vec T_j {\vec U}^{ij}) \right) \cdot {\vec n}^* +  p_m^* \\
&\qquad  - \alpha^{-1} \Big(  \big(v_1^* {\mathcal N}_1^j  + v_2^{*} {\mathcal N}_2^j  \big)  p_m^* - \big( B_1^{ij} {\mathcal N}_1^j  + B_2^{ij} {\mathcal N}_2^j \big) ( {\vec v}^* \cdot  {\vec B}^* ) \Big),
\end{split}
\end{equation}
where the orthogonality of $\vec T_j$ has been used. Hence, one has
\begin{equation*}
\begin{split}
&\bar{\vec U} \cdot \vec n^* + p_{m}^* \\
& = \frac{1}{{\sum\limits_{j = 1}^J {\ell_j } }}\sum\limits_{j = 1}^J {\left[ {\sum\limits_{i = 1}^L {\omega _i \Big( \big( {{\vec U}^{ij}  - \alpha ^{ - 1} \vec T_j^{-1}  \vec F_1(\vec T_j {\vec U}^{ij})  } \big) \cdot  \vec n^* + p_{m}^*   \Big)    } } \right]} \ell_j\\
& \overset{\eqref{eq:wkl000}} {>}
\frac{1}{\alpha{\sum\limits_{j = 1}^J {\ell_j } }}\sum\limits_{j = 1}^J {\left[ {\sum\limits_{i = 1}^L {\omega _i \Big(  \big(v_1^* {\mathcal N}_1^j  + v_2^{*} {\mathcal N}_2^j  \big)  p_m^* - \big( B_1^{ij} {\mathcal N}_1^j  + B_2^{ij} {\mathcal N}_2^j \big) ( {\vec v}^* \cdot  {\vec B}^* )  \Big) } } \right]} \ell_j\\
& \overset{\eqref{eq:descrite2DDIVus}} {=}   \frac{p_m^*}{\alpha{\sum\limits_{j = 1}^J {\ell_j } }}\sum\limits_{j = 1}^J {\left[ {\sum\limits_{i = 1}^L {\omega _i   \left(v_1^* {\mathcal N}_1^j  + v_2^{*} {\mathcal N}_2^j  \right)    } } \right]} \ell_j
=  \frac{p_m^*}{\alpha{\sum\limits_{j = 1}^J {\ell_j } }}\sum\limits_{j = 1}^J    \left(v_1^* {\mathcal N}_1^j  + v_2^{*} {\mathcal N}_2^j  \right)  \ell_j =0,
\end{split}
\end{equation*}
which implies that $\bar{\vec U}$ satisfies the second constraint in  ${\mathcal G}_1$.
On the other hand,  $\bar{\vec U}$  satisfies the first constraint in ${\mathcal G}_1$  because
\begin{equation*}
\begin{split}
 & \frac{1}{{\sum\limits_{j = 1}^J {\ell_j } }}\sum\limits_{j = 1}^J {\left[ {\sum\limits_{i = 1}^L {\omega _i D^{ij} \Big( { 1 - \alpha ^{ - 1} \big( v_1^{ij} {\mathcal N}_1^j  + v_2^{ij} {\mathcal N}_2^j \big)} \Big)} } \right]} \ell_j  \\
 &\ge
  \frac{1}{{\sum\limits_{j = 1}^J {\ell_j } }}\sum\limits_{j = 1}^J {\left[ {\sum\limits_{i = 1}^L {\omega _i D^{ij} \Big(  1 - \alpha ^{ - 1} \sqrt{ \big( v_1^{ij}\big)^2 + \big( v_2^{ij}\big)^2 }    \Big)} } \right]} \ell_j  >0.
\end{split}
\end{equation*}
Thus, the proof is completed by using Theorem  \ref{theo:RMHD:CYcondition:VecN}.
\end{proof}

\section{Physical-constraints-preserving schemes}\label{sec:app}
This section applies the previous theoretical results on the admissible state set
$\mathcal G$ to  develope  PCP numerical schemes for the 1D and 2D
special RMHD equations \eqref{eq:RMHD1D}.

\subsection{1D PCP schemes}\label{sec:1Dpcp}

For the sake of convenience, this subsection will use the symbol $x$
to replace the independent variable $x_1$ in \eqref{eq:RMHD1D}.
Assume that the spatial domain is divided into a uniform mesh with a constant spatial step-size $\Delta x$
and the $j$-th cell $I_j=(x_{j-\frac{1}{2}},x_{j+\frac{1}{2}})$,
and the time interval is also divided into the (non-uniform) grid $\{t_0=0, t_{n+1}=t_n+\Delta t_{n}, n\geq 0\}$
with the time step size $\Delta t_{n}$ determined by the CFL type condition.
Let $\overline {\vec U}_j^n $ be the numerical approximation to the cell average value
of the exact solution $\vec U(x,t)$ over the cell $I_j$ at $t=t_n$.
Our aim is to seek numerical schemes of the 1D RMHD equations \eqref{eq:RMHD1D},
whose solution  $\overline {\vec U}_j^n$ belongs to the set ${\mathcal G}$ if $\overline {\vec U}_j^0\in {\mathcal G}$.

\subsubsection{First-order accurate schemes}

Consider the first-order accurate LxF type scheme
\begin{equation}\label{eq:1DRMHD:LFscheme}
\overline {\vec U}_j^{n+1} = \overline {\vec U}_j^n - \frac{\Delta t_n}{\Delta x} \Big( \hat {\vec F}_1 ( \overline {\vec U}_j^n ,\overline {\vec U}_{j+1}^n) - \hat {\vec F}_1 ( \overline {\vec U}_{j-1}^n, \overline {\vec U}_j^n   \Big) ,
\end{equation}
where the numerical flux $\hat {\vec F}_1$ is defined by
\begin{equation}\label{eq:LFflux}
\hat{\vec F}_i(\vec U^-,\vec U^+) = \frac{1}{2} \Big( \vec F_i(\vec U^-) + \vec F_i(\vec U^+) - \varrho_i (\vec U^+ - \vec U^-) \Big), \quad i=1,2,3.
\end{equation}
Here $\varrho_i$ is an appropriate upper bound of the spectral radius of the Jacobian matrix $\partial \vec F_i(\vec U)/\partial \vec U$ and may be taken as $\varrho_i=c=1$.

Thanks to the generalized LxF splitting property shown in Theorem \ref{theo:RMHD:LLFsplit1D},  one can prove that the scheme
\eqref{eq:1DRMHD:LFscheme} is PCP under a CFL type condition.
%if the initial data satisfy $\overline {\vec U}_j^0 \in{\mathcal G}$ and $\overline B_{1,j}^0=B_{\rm const}$ for all $j$.

\begin{theorem}\label{theo:1DRMHD:LFscheme}
	If  $\overline {\vec U}_j^0 \in{\mathcal G}$ and $\overline B_{1,j}^0 = B_{\rm const}$ for all $j$,
	then $\overline {\vec U}_j^n$, calculated by using \eqref{eq:1DRMHD:LFscheme} under the CFL type condition
\begin{equation}\label{eq:CFL:LF}
0< \Delta t_n \le  \Delta x/c,
\end{equation}
belongs to ${\mathcal G}$ and satisfies $\overline B_{1,j}^n = B_{\rm const} $ for all $j$
and $n\in {\mathbb{N}}$, where $c=1$ is the speed of light.
\end{theorem}

\begin{proof}
Here the induction argument is used for the time level number $n$.
It is obvious that the conclusion holds for $n=0$ because of the hypothesis on  the initial data.
Now assume that $\overline {\vec U}_j^n\in {\mathcal G}$ with $\overline B_{1,j}^n = B_{\rm const} $ for all $j$,
and check whether the conclusion holds for $n+1$.
Thanks to the numerical flux in \eqref{eq:LFflux}, the fifth equation in \eqref{eq:1DRMHD:LFscheme} gives
\begin{equation*}
\begin{split}
\overline B_{1,j}^{n+1} &= \overline B_{1,j}^{n} - \frac{\Delta t_n}{2\Delta x} \Big(  2 \overline B_{1,j}^{n} - \overline B_{1,j+1}^{n} - \overline B_{1,j-1}^{n} \Big)\\
& = B_{\rm const} - \frac{\Delta t_n}{2\Delta x} \Big(  2 B_{\rm const} - B_{\rm const} - B_{\rm const} \Big) = B_{\rm const},
\end{split}
\end{equation*}
for all $j$. If substituting \eqref{eq:LFflux} into \eqref{eq:1DRMHD:LFscheme},
 one can rewrite \eqref{eq:1DRMHD:LFscheme} in the following form
 \begin{equation*}
 \begin{split}
\overline {\vec U}_j^{n+1} &= (1-\lambda) \overline {\vec U}_j^n
+ \frac{\lambda}{2 }\Big( \overline {\vec U}_{j+1}^n -\vec F_1\left( \overline {\vec U}_{j+1}^n\right) +
                         \overline {\vec U}_{j-1}^n +\vec F_1\left( \overline {\vec U}_{j-1}^n\right) \Big)
\\
& =: (1-\lambda) \overline {\vec U}_j^n  + \lambda \Xi,
\end{split}
\end{equation*}
where  $\lambda := \Delta t_n / \Delta x \in (0,1]$ due to \eqref{eq:CFL:LF}.  With the induction hypothesis and Theorem \ref{theo:RMHD:LLFsplit1D}, one has $\Xi \in {\mathcal G}$.
The convexity of $\mathcal G$ further yields $\overline {\vec U}_j^{n+1} \in {\mathcal G}$.
The proof is completed.
\end{proof}

\subsubsection{High-order accurate  schemes}\label{sec:High1D}

This subsection discusses the high-order accurate PCP schemes for the 1D RMHD equations \eqref{eq:RMHD1D}.

Let us consider the  high-order {(spatially)} accurate scheme
for the 1D RMHD equations \eqref{eq:RMHD1D}
\begin{equation}\label{eq:1DRMHD:cellaverage}
\overline {\vec U}_j^{n+1} = \overline {\vec U}_j^{n} - \frac{\Delta t_n}{\Delta x}
 \left(  \hat {\vec F}_1 (\vec U_{j+ \frac{1}{2}}^-,\vec U_{j+ \frac{1}{2}}^+ )
  - \hat {\vec F}_1 (\vec U_{j- \frac{1}{2}}^-,\vec U_{j-\frac{1}{2}}^+)
  \right) ,
\end{equation}
where the numerical flux $\hat {\vec F}_1$ is  defined by \eqref{eq:LFflux}.
Eq. \eqref{eq:1DRMHD:cellaverage} may be derived from high-order accurate finite volume schemes  or the discrete evolution equation for the cell average value $\{\overline {\vec U}_j^n\}$ in the discontinuous Galerkin (DG) methods.
The quantities  $\vec U_{j + \frac{1}{2}}^-$ and $\vec U_{j + \frac{1}{2}}^+$ are the $(K+1)$th-order  accurate approximations of the point values $\vec U\big( x_{j + \frac{1}{2}} ,t_n \big)$ within the cells $I_j$ and $I_{j+1}$ respectively, and given by
$$
\vec U_{j + \frac{1}{2}}^- = \vec U_j^n \big(x_{j + \frac{1}{2}}-0 \big), \quad \vec U_{j + \frac{1}{2}}^+ = \vec U_{j+1}^n \big( x_{j + \frac{1}{2}}+0 \big),
$$
where the polynomial vector function ${\vec U}_j^n(x)$ is
with the cell average value of $\overline {\vec U}_j^n$, approximating $\vec U(x,t_n)$ within the cell $I_j$, and either reconstructed in the finite volume methods from $\{\overline {\vec U}_j^n\}$ or directly evolved in the DG methods with degree $K \ge 1$.
%The cell average value of ${\vec U}_j(x)$ over the cell $I_j$  is equal to
% $\overline {\vec U}_j^n$.
The evolution equations for the high-order ``moments'' of ${\vec U}_j(x)$ in the DG methods is omitted because  we are only concerned with  the PCP property of the numerical schemes here.
%If the  first-order accurate time discretization  \eqref{eq:1DRMHD:cellaverage} may be replaced with the high-order accurate strong stability preserving  time discretization,
%the following discussion is still true.  see  the end of this  subsection.
%Similar to  \cite{zhang2010b},

Generally, the solution $\overline {\vec U}_j^{n+1}$ of the high-order accurate scheme \eqref{eq:1DRMHD:cellaverage} may not belong to
${\mathcal G}$
%with $\vec U_{j + \frac{1}{2}}^\pm$ given by a traditional finite volume or %DG method is not PCP, that is to say, it is possible to meet
 %$\overline {\vec U}_j^{n+1} \notin {\mathcal G}$
even if $\overline {\vec U}_j^{n} \in {\mathcal G}$ for all $j$.
%In other words, it is possible to generate the nonphysical solution by the scheme \eqref{eq:1DRMHD:cellaverage}. Especially,
Thus if the scheme \eqref{eq:1DRMHD:cellaverage} is used to solve some ultra-relativistic problems with low density or pressure, or very large velocity, it may  break down after some time steps due to the nonphysical numerical solutions generated by \eqref{eq:1DRMHD:cellaverage}.
To cure such defect, the positivity-preserving limiters  devised
%for high-order accurate finite volume or DG method for non-relativistic Euler equations
in \cite{Xing2010,zhang2010,zhang2010b} %, and the high-order accurate DG method for shallow water equations \cite{}.
will be  extended to our RMHD case and   ${\vec U}_j(x)$ is limited as $\tilde{\vec U}_j(x)$ such that the values of $\tilde{\vec U}_j(x)$ at some critical points in $I_j$ belong to $\mathcal G$ .
Let $\{ \hat x_j^\alpha \}_{\alpha=1} ^ L$ be the Gauss-Lobatto nodes transformed into the interval $I_j$, and $\{\hat \omega_\alpha\}_{\alpha=1} ^ L$ be  associated Gaussian  quadrature weights satisfying $\sum \limits_{\alpha=1}^L \hat\omega_\alpha = 1$. Here we take $2L-3\ge K$ in order that the algebraic precision of corresponding quadrature  is at least $K$. In particular, one can take $L$ as the smallest integer not less than  $\frac{K+3}{2}$.

\begin{theorem} \label{thm:PCP:1DRMHD}
If  the polynomial vector ${\vec U}^n_j(x)=:( D_j(x), \vec m_j( x),\vec B_j( x), E_j( x) )^{\top}$ satisfy:
\\
{\rm (i)}. $B_{1,j}(x)=B_{\rm const} $ for any $x\in I_j$ and all $j$, and
\\
{\rm (ii)}. $\vec U_j^n ( \hat x_j^\alpha ) \in {\mathcal G}$ for all $j$ and $\alpha=1,2,\cdots,L$,
\\
then under the CFL type condition
\begin{equation}\label{eq:CFL:1DRMHD}
0<\Delta t_n \le \hat \omega_1 \Delta x,
\end{equation}
it holds that  $\overline{\vec U}_j^{n+1} \in {\mathcal G}$ for the numerical
scheme \eqref{eq:1DRMHD:cellaverage}.
\end{theorem}

\begin{proof}
The exactness of the Gauss-Lobatto quadrature rule with $L$ nodes for the polynomials of degree $K$ yields
$$
\overline{\vec U}_j^n = \frac{1}{\Delta x} \int_{I_j} \vec U_j^n (x) dx = \sum \limits_{\alpha=1}^L \hat \omega_\alpha \vec U_j^n (\hat x_j^\alpha ).
$$
Because $\hat \omega_1 = \hat \omega_L$, one has
\begin{align} \nonumber
& \overline{\vec U}_j^{n+1}   = \sum \limits_{\alpha=1}^L \hat \omega_\alpha \vec U_j^n ( \hat x_j^\alpha )
-\lambda \left( \hat {\vec F}_1 (\vec U_{j+ \frac{1}{2}}^-,\vec U_{j+ \frac{1}{2}}^+ ) -
 \hat {\vec F}_1 (\vec U_{j- \frac{1}{2}}^-,\vec U_{j-\frac{1}{2}}^+)  \right) \\ \nonumber
 \begin{split}
& = \sum \limits_{\alpha=2}^{L-1} \hat \omega_\alpha \vec U_j ^n ( \hat x_j^\alpha ) + \frac{\lambda}{2} \left( \vec U_{j+\frac{1}{2}}^+ - \vec F_1 ( \vec U_{j+\frac{1}{2}}^+ )
+ \vec U_{j-\frac{1}{2}}^- + \vec F_1 ( \vec U_{j-\frac{1}{2}}^- ) \right)\\ \nonumber
& \quad
 + \hat \omega_1 \vec U_{j-\frac{1}{2}}^+
+ \hat \omega_L \vec U_{j+ \frac{1}{2}}^-
- \frac{\lambda}{2} \left(   \vec U_{j+\frac{1}{2}}^- + \vec F_1 ( \vec U_{j+\frac{1}{2}}^- )
+ \vec U_{j-\frac{1}{2}}^+ - \vec F_1 ( \vec U_{j-\frac{1}{2}}^+ ) \right)
\end{split}
\\ \nonumber
\begin{split}
& = \sum \limits_{\alpha=2}^{L-1} \hat \omega_\alpha \vec U_j^n ( \hat x_j^\alpha ) + \frac{\lambda}{2} \left( \vec U_{j+\frac{1}{2}}^+ - \vec F_1 ( \vec U_{j+\frac{1}{2}}^+ )
+ \vec U_{j-\frac{1}{2}}^- + \vec F_1 ( \vec U_{j-\frac{1}{2}}^- ) \right)\\ \nonumber
& \quad + \left(\hat \omega_1 -\frac{\lambda}{2} \right)
\left[ \vec U_{j-\frac{1}{2}}^+ - \left(  \frac{2 \hat\omega_1}{\lambda} -1 \right)^{-1} \vec F_1 ( \vec U_{j-\frac{1}{2}}^+  )
+  \vec U_{j+\frac{1}{2}}^- + \left(  \frac{2 \hat\omega_1}{\lambda} -1 \right)^{-1} \vec F_1 ( \vec U_{j+\frac{1}{2}}^- )
\right]
\end{split}
\\ \label{eq:1DRMHD:convexsplit}
& =: \sum \limits_{\alpha=2}^{L-1} \hat \omega_\alpha \vec U_j^n ( \hat x_j^\alpha ) + \lambda \Xi_1 + \left(\hat \omega_1 + \hat \omega_L -\lambda \right)  \Xi_2,
\end{align}
where $\lambda = \Delta t_n/\Delta x \in (0,\hat \omega_1 ]$, and
\begin{align*}
& \Xi_1 := \frac{1}{2} \left( \vec U_{j+\frac{1}{2}}^+ - \vec F_1 ( \vec U_{j+\frac{1}{2}}^+ )
+ \vec U_{j-\frac{1}{2}}^- + \vec F_1 ( \vec U_{j-\frac{1}{2}}^- ) \right),\\
& \Xi_2 := \frac{1}{2} \left[ \vec U_{j-\frac{1}{2}}^+ - \left(  \frac{2 \hat\omega_1}{\lambda} -1 \right)^{-1} \vec F_1 ( \vec U_{j-\frac{1}{2}}^+)
+  \vec U_{j+\frac{1}{2}}^- + \left(  \frac{2 \hat\omega_1}{\lambda} -1 \right)^{-1} \vec F_1 ( \vec U_{j+\frac{1}{2}}^- )
\right].
\end{align*}
Since $B_{1,j}(x)$ is a constant and $2 \hat\omega_1/\lambda -1 \ge 1$,
the generalized LxF property  in Theorem \ref{theo:RMHD:LLFsplit1D} tell us that $\Xi_1,\Xi_2 \in {\mathcal G}$.
Using  $\hat \omega_1 + \hat \omega_L -\lambda>0$,
\eqref{eq:1DRMHD:convexsplit}, and the convexity of $\mathcal G$
gives $\overline{\vec U}_j^{n+1} \in {\mathcal G} $.
The proof is completed.
\end{proof}

Theorem \ref{thm:PCP:1DRMHD} gives two sufficient
 conditions on
the approximate function ${\vec U}^n_j(x)$
for that the   scheme \eqref{eq:1DRMHD:cellaverage} is PCP.
The first condition is easily ensured in practice since
the flux for $B_1$ is zero and the divergence-free condition \eqref{eq:2D:BxBy0} in the case of $d=1$ implies that $B_1$ is always a constant for the exact solution to \eqref{eq:RMHD1D}.
To meet the second condition, we need  a PCP limiting procedure, in which
${\vec U}^n_j(x)$ is limited as $\tilde{\vec U}_j(x)$ satisfying $\tilde{\vec U}_j ( \hat x_j^\alpha ) \in {\mathcal G} $.

To avoid the effect of the rounding error, we introduce a sufficiently small positive number\footnote{In practice, $\epsilon$ can be chosen as $10^{-13}$, and certainly it may be different for three constraints in \eqref{set-G-epsilon}. However, for the extreme problems with $E\gg1$,
  $\epsilon=10^{-13}\overline E_j^n$  is a good choice for the last constraint.} $\epsilon$ such that $ \overline {\vec U}_j^n \in {\mathcal G}_\epsilon$, where
\begin{align} \label{set-G-epsilon}
{\mathcal G}_\epsilon = \left\{   \vec U=(D,\vec m,\vec B,E)^{\top} \Big|  D\ge\epsilon,~q(\vec U)\ge\epsilon,~{\Psi}_\epsilon(\vec U) \ge 0\right\},
\end{align}
with
$$
{\Psi}_\epsilon(\vec U) : = {\Psi}(\vec U_\epsilon),\quad \vec U_\epsilon : = \big(D,\vec m,\vec B,E-\epsilon\big)^{\top}.
%\big( \Phi_\epsilon-2(|\vec B|^2-E+\epsilon) \big) \sqrt{\Phi_\epsilon+(|\vec B|^2-E+\epsilon)} - \sqrt{ \frac{27}{2} \bigg( D^2|\vec B|^2+(\vec m \cdot \vec B)^2 \bigg)},
$$
%and $\Phi_\epsilon := \sqrt{({|\vec B|^2} - E+\epsilon)^2 + 3({(E-\epsilon)^2} - {D^2} - |\vec m|^2)}$.
%
%
%{\tt Algorithm 1}:  limiting procedure}
%
Then the {\tt 1D PCP limiting procedure} is divided into the following three steps.

\noindent
{\bf Step (i)}: Enforce the positivity of $D(\vec U)$. Let $D_{\min} = \min \limits_{x \in {\mathcal S}_j}^{} D_j ( x )$,
where ${\mathcal S}_j:=\{\hat{x}_j^\alpha\}_{\alpha=1}^L$.
 If $D_{\min} < \epsilon$,
then  $D_j ( x )$ is limited as
$$
\hat D_j(x) = \theta_1 \big( D_j(x) - \overline D_j^n \big) + \overline D_j^n,
$$
where $\theta_1 = (\overline D_j^n - \epsilon)/ ( \overline D_j^n - D_{\min} ) $. Otherwise, take $\hat D_j(x) =  D_j(x)$ and $\theta_1 = 1$.
Denote $\hat {\vec U}_j(x) :=  ( \hat D_j(x), \vec m_j(x),\vec B_j(x), E_j(x) )^{\top}$.

\noindent
{\bf Step (ii)}: Enforce the positivity of $q(\vec U)$. Let
$q_{\min} = \min \limits_{x \in {\mathcal S}_j}^{} q(\hat {\vec U}_j ( x ))$. If $q_{\min} < \epsilon$, then  $\hat {\vec U}_j ( x )$ is limited as
$$
\breve {\vec U}_j(x) = \Big(  \theta_2 \big( \hat {D}_j (x) - \overline {D}_j^n \big) + \overline {D}_j^n,
\theta_2 \big( \hat {\vec m}_j (x) - \overline {\vec m}_j^n \big) + \overline {\vec m}_j^n,
 \hat {\vec B}_j (x) ,
\theta_2 \big( \hat {E}_j (x) - \overline {E}_j^n \big) + \overline {E}_j^n \Big)^\top,
$$
where $\theta_2 = (q(\overline {\vec U}_j^n) - \epsilon)/ ( q(\overline {\vec U}_j^n) - q_{\min} ) $. Otherwise, set $\breve {\vec U}_j(x) =  \hat {\vec U}_j(x)$ and $\theta_2 =1$.

\noindent
{\bf Step (iii)}: Enforce the positivity of ${\Psi}(\vec U)$.
For each $x \in {\mathcal S}_j$, if ${\Psi}_\epsilon ( \breve {\vec U}_j(x) ) < 0$, then define $\tilde \theta(x)$ by solving the nonlinear equation
\begin{equation}\label{eq:limiterEq}
{\Psi}_\epsilon \Big( (1- \tilde  \theta) \overline{\vec U}_j^n + \tilde  \theta \breve {\vec U}_j(x) \Big) =0, \quad  \tilde  \theta \in [0,1).
\end{equation}
Otherwise, %If ${\Psi}_\epsilon ( \breve {\vec U}_j(x) ) \ge 0$, then
set $\tilde  \theta(x)=1$. Let $\theta_3 = \min \limits_{x\in {\mathcal S}_j} \{\tilde  \theta(x)\}$  and
\begin{equation}\label{eq:PCPpolynomial}
\tilde {\vec U}_j(x) = \theta_3 \big( \breve {\vec U}_j (x) - \overline {\vec U}_j^n \big) + \overline {\vec U}_j^n.
\end{equation}

%\begin{remark}
%{\color{red}Because the function $q(\vec U)$ only involves the $D$, $\vec m$ and $E$ components of $\vec U$, in {\bf Step (ii)} of the above limiting procedure, one can also only limits these five components of $\hat {\vec U}_j ( x )$.}
%\end{remark}
\begin{remark}
For some  high-order finite volume methods,
it only needs to reconstruct
the limiting values $\vec U_{j+\frac{1}{2}}^\pm$,
instead of the  polynomial vector
${\vec U}_j(x)$. For this case,  due to  the proof of Theorem \ref{thm:PCP:1DRMHD},
it is sufficient that the limiting values satisfy
$$
\vec U_{j-\frac{1}{2}}^+,~\vec U_{j+\frac{1}{2}}^-,~\frac{\overline {\vec U}_j^n -  \hat \omega_1 \vec U_{j-\frac{1}{2}}^+
- \hat \omega_L \vec U_{j+ \frac{1}{2}}^- }{1-2 \hat \omega_1}  \in {\mathcal G},
$$
for all $j$. Similar to the discussions in Section 5 of \cite{zhang2011b},
the previous PCP limiting procedure can be easily revised to meet such condition.
\end{remark}

If replacing ${\vec U}_{j+\frac12}^\pm$ in \eqref{eq:1DRMHD:cellaverage} respectively by
$$
{ \tilde {\vec U}}_{j + \frac{1}{2}}^- = { \tilde {\vec U}}_j \big(x_{j + \frac{1}{2}} \big), \quad { \tilde {\vec U}}_{j + \frac{1}{2}}^+ = { \tilde {\vec U}}_{j+1} \big( x_{j + \frac{1}{2}} \big),
$$
 then the resulting scheme is PCP under the CFL type condition \eqref{eq:CFL:1DRMHD}, according to the conclusion (ii) of the coming Lemma \ref{lam:limiter}. The above PCP limiter satisfies
$$
   \overline {\vec U}_j^n = \frac{1}{\Delta x} \int_{I_j} {\vec U}_j( x)  dx = \frac{1}{\Delta x} \int_{I_j} \hat{\vec U}_j( x)  dx
   = \frac{1}{\Delta x} \int_{I_j} \breve{\vec U}_j( x)  dx = \frac{1}{\Delta x} \int_{I_j} \tilde{\vec U}_j( x)  dx,
$$
and  that $\tilde {B}_{1,j}(x) $ remains constant for any $x\in I_j$ and  $j$ if $B_{1,j}(x) $ is  constant for any $x\in I_j$ and  $j$.
It also preserves the accuracy for smooth solutions, similar to \cite{zhang2010b}.
 The  scheme \eqref{eq:1DRMHD:cellaverage} is  only first-order accurate in time. To achieve high-order accurate PCP scheme in time and space, one can replace the forward Euler time discretization in \eqref{eq:1DRMHD:cellaverage} with high-order accurate strong stability preserving (SSP)  methods \cite{Gottlieb2009}.  For example, utilizing the third-order accurate SSP Runge-Kutta method obtains
 \begin{align} \label{eq:RK1} \begin{aligned}
 & \overline {\vec U}^ *_j   = \overline {\vec U}^n_j  + \Delta t_n {\mathcal L}(  \tilde {\vec U}_j(x);j ), \\
 & \overline {\vec U}^{ *  * }_j  = \frac{3}{4} \overline {\vec U}^n_j  + \frac{1}{4}\Big( \overline {\vec U}^ *_j
  + \Delta t_n {\mathcal L}( \tilde {\vec U}^ *_j(x);j  )\Big), \\
 & \overline {\vec U}^{n+1}_j  = \frac{1}{3} \overline {\vec U}^n_j  + \frac{2}{3}\Big( \overline {\vec U}^{ *  * }_j
 + \Delta t_n {\mathcal L}( \tilde {\vec U}^{ *  * }_j(x);j)\Big),
 \end{aligned}\end{align}
 where $\tilde {\vec U}_j(x)$, $\tilde {\vec U}^ *_j(x)$, and $\tilde {\vec U}^ {**}_j(x)$ denote the PCP limited versions of the reconstructed or evolved polynomial vector function at each Runge-Kutta stage, and
 $$
 {\mathcal L}(  \vec U_j(x);j ) :=  - \frac{1}{\Delta x}
  \left(  \hat {\vec F}_1 (\vec U_{j+ \frac{1}{2}}^-,\vec U_{j+ \frac{1}{2}}^+ )
   - \hat {\vec F}_1 (\vec U_{j- \frac{1}{2}}^-,\vec U_{j-\frac{1}{2}}^+)
   \right).
 $$
 Since such SSP method is a convex combination of the forward Euler method, the resulting high-order scheme is still PCP under the CFL type condition \eqref{eq:CFL:1DRMHD} by the convexity of $\mathcal G$.
Moreover, similar to \cite{WuTang2016}, the PCP schemes hold a discrete $L^1$-type stability for the solution components $\tilde D_j(x)$, $\tilde {\vec m}_j(x)$ and $\tilde E_j(x)$. It is worth noting  that
the set ${\mathcal G}_\epsilon$ in \eqref{set-G-epsilon} is  convex  thanks to the convexity of ${\mathcal G}_0$ so that  the solution to \eqref{eq:limiterEq} is the unique.
This allows that one can use some root-finding methods such as the bisection method to numerically solve \eqref{eq:limiterEq}.
Moreover, one can show ${\mathcal G}_\epsilon \subset {\mathcal G}_0$ and $\mathop {\lim }\limits_{\epsilon  \to 0^ +  }  {\mathcal G}_\epsilon = \overline {\mathcal G}_0$.

\begin{lemma}\label{lam:limiter}
{\rm (i)} ${\mathcal G}_\epsilon \subset {\mathcal G}_0$.
{\rm (ii)} If $ \overline {\vec U}_j^n \in {\mathcal G}_\epsilon$, then
  $\tilde {\vec U}_j(x)$ defined in \eqref{eq:PCPpolynomial} belongs to $ {\mathcal G}_\epsilon$ for all $x \in {\mathcal S}_j$.
\end{lemma}

\begin{proof}
%\begin{itemize}[\hspace{0em}$\bullet$]
(i) Let us first  prove  ${\mathcal G}_\epsilon \subset {\mathcal G}_0$. For any $\vec U=(D,\vec m,\vec B,E)^{\top} \in {\mathcal G}_\epsilon$, one has
$D\ge \epsilon >0, q(\vec U)\ge \epsilon >0$, and ${\Psi}(\vec U_\epsilon)\ge0$ with $\vec U_\epsilon : = \big(D,\vec m,\vec B,E-\epsilon\big)^{\top}$.
Taking partial derivative of ${\Psi}(\vec U)$ with respect to $E$ gives
\begin{equation*}
\begin{split}
\frac{\partial {\Psi} } {\partial E} &= \left( \frac{\partial \Phi(\vec U) }{\partial E} +2 \right) \sqrt{  \Phi(\vec U) + |\vec B|^2-E }
+ \frac{ \Phi(\vec U) - 2(|\vec B|^2 - E)}{2 \sqrt{ \Phi(\vec U)+ |\vec B|^2 - E }} \left(  \frac{\partial \Phi(\vec U) }{\partial E} - 1 \right)
\\
& = \frac{3}{2 \sqrt{ \Phi(\vec U) + |\vec B|^2 - E }} \left(   \Phi(\vec U)    \frac{\partial \Phi(\vec U) }{\partial E} +  \Phi(\vec U) + 2 (|\vec B|^2 - E)  \right) \\
& = \frac{3}{2 \sqrt{ \Phi(\vec U) + |\vec B|^2 - E }} \left(   \Phi(\vec U)    \frac{ 4 E - |\vec B|^2 } { \Phi(\vec U) }
+  \Phi(\vec U) + 2 (|\vec B|^2 - E)  \right)
\\
& = \frac{3 ( 2E + |\vec B|^2 + \Phi(\vec U)  ) }{2 \sqrt{ \Phi(\vec U) + |\vec B|^2 - E }} > 0 ,
\end{split}
\end{equation*}
for any $\vec U\in {\mathcal G}_2$. This implies $ {\Psi}(\vec U) > {\Psi}(\vec U_\epsilon) \ge 0$, and concludes that $\vec U \in {\mathcal G}_0$. Therefore
${\mathcal G}_\epsilon \subseteq {\mathcal G}_0$.
Because $\left( \frac{\epsilon}{2},\vec 0,\vec 0, \epsilon\right)^{\top}$ belongs to ${\mathcal G}_0$, but it does not in ${\mathcal G}_\epsilon$,
one has ${\mathcal G}_\epsilon \subset {\mathcal G}_0$.

(ii) Next we prove $\tilde{\vec U}_j ( x ) \in {\mathcal G}_\epsilon$ for any $x \in {\mathcal S}_j$. The above PCP limiting {procedure} yields
$$
\hat D_j (x) \ge \epsilon,\quad q\big( \breve{ \vec U}_j(x) \big) \ge \epsilon, \quad {\Psi}_\epsilon \big( \tilde { \vec U}_j(x) \big) \ge 0,
$$
for   $x\in {\mathcal S}_j$. For any $x\in {\mathcal S}_j$, one has
\begin{align*}
\begin{split}
\tilde D_j(x)  &= \theta_3 \left( \breve D_j( x) - \overline D_j^n \right) +  \overline D_j^n = \theta_2\theta_3 \left( \hat D_j( x) - \overline D_j^n \right) +  \overline D_j^n\\
& \ge \theta_2\theta_3 \left( \epsilon - \overline D_j^n \right) +  \overline D_j^n \ge \epsilon,
\end{split}
\end{align*}
by noting $\theta_2,\theta_3 \in[0,1]$.
Similarly, making use of the concavity of $q(\vec U)$ gives
\begin{align*}
\begin{split}
q\big( \tilde { \vec U}_j(x) \big) & = q\Big(  \theta_3  \breve {\vec U}_j (x) + (1-\theta_3) \overline {\vec U}_j^n \Big)
 \ge  \theta_3 q\Big(  \breve {\vec U}_j (x) \Big) + (1-\theta_3) q \Big( \overline {\vec U}_j^n \Big)\\
 & \ge \theta_3 \epsilon + (1-\theta_3) \epsilon = \epsilon.
\end{split}
\end{align*}
The proof is completed.
\end{proof}

\subsection{2D PCP schemes}\label{sec:2Dpcp}

For the sake of convenience, this subsection will use
the symbols $(x,y)$ to replace the independent variables $(x_1,x_2)$ in \eqref{eq:RMHD1D}.
Assume that the spatial domain is divided into a uniform rectangular mesh with cells $\left\{I_{ij}=\left(x_{i-\frac{1}{2}},x_{i+\frac{1}{2}}\right)\times
\left(y_{j-\frac{1}{2}},y_{j+\frac{1}{2}}\right) \right\}$ and the spatial step-sizes $\Delta x$ and $\Delta y$ in $x$ and $y$ directions respectively,
and
the time interval is also divided into the (non-uniform) mesh $\{t_0=0, t_{n+1}=t_n+\Delta t_{n}, n\geq 0\}$
with the time step size $\Delta t_{n}$ determined by the CFL type condition.
Let $\overline {\vec U}_{ij}^n $ be the numerical approximation to the cell average value of the exact solution $\vec U(x,y,t)$ over $I_{ij}$ at $t=t_n$.
Our aim is to seek  numerical schemes of the 2D RMHD equations \eqref{eq:RMHD1D},
whose solution  $\overline {\vec U}_{ij}^n$ {stays at} ${\mathcal G}$ if $\overline {\vec U}_{ij}^0\in {\mathcal G}$.

\subsubsection{First-order accurate  schemes} \label{sec:2D:FirstOrder}
Consider the first-order accurate LxF type scheme
\begin{equation} \label{eq:2DRMHD:LFscheme}
\begin{split}
\overline {\vec U}_{ij}^{n+1} = \overline {\vec U}_{ij}^n &- \frac{\Delta t_n}{\Delta x} \Big( \hat {\vec F}_1 ( \overline {\vec U}_{ij}^n ,\overline {\vec U}_{i+1,j}^n) - \hat {\vec F}_1 ( \overline {\vec U}_{i-1,j}^n, \overline {\vec U}_{ij}^n   \Big) ,\\
&- \frac{\Delta t_n}{\Delta y} \Big( \hat {\vec F}_2 ( \overline {\vec U}_{ij}^n ,\overline {\vec U}_{i,j+1}^n) - \hat {\vec F}_2 ( \overline {\vec U}_{i,j-1}^n, \overline {\vec U}_{ij}^n   \Big)
\end{split}
\end{equation}
where the $x$- and $y$-directional numerical fluxes $\hat {\vec F}_1$ and $\hat {\vec F}_2$ are   defined as \eqref{eq:LFflux}.
If $\overline {\vec U}_{ij}^n$  belongs to $\mathcal G$ for all $i,j$,
but the magnetic field $\overline{\vec B}_{ij}^n$ is not  divergence-free
in the discrete sense,
then the solution $\overline {\vec U}_{ij}^{n+1}$ of \eqref{eq:2DRMHD:LFscheme} may not belong to $\mathcal G$, see Example \ref{Counterexample}.
It means that the scheme \eqref{eq:2DRMHD:LFscheme} is not PCP in general
when the divergence of magnetic field is nonzero.

\begin{example}\label{Counterexample} %{\bf Counterexample}
 For any $\epsilon>0$, take  the primitive variable vectors  $\hat{\vec V}=(\epsilon,0.5,0,0,0,0,0,\epsilon)^{\top}$ and $\tilde{\vec V}=(\epsilon,0.5,0,0,1,0,0,\epsilon)^{\top}$,
 and let $\hat{\vec U}=\hat{\vec U}(\hat{\vec V})$ and $\tilde{\vec U}=\tilde{\vec U}(\tilde{\vec V})$ be corresponding conservative  vectors.
If taking $\vec U_{i+1,j}^n = \tilde{\vec U} \in{\mathcal G}$ and $\vec U_{ij}^n = \vec U_{i,j\pm1}^n = \vec U_{i-1,j}^n = \hat{\vec U} \in{\mathcal G} $, then
substituting them into \eqref{eq:2DRMHD:LFscheme} gives
$$
\overline {\vec U}_{ij}^{n+1} (\epsilon) =
 \left( \frac{2\sqrt{3}}{3}\epsilon,
\left(\frac{4}{3}+\frac{2}{3(\Gamma-1)} \right) \epsilon + \frac{\Delta t_n }{4\Delta x} ,0,0, \frac{\Delta t_n }{2\Delta x}  , 0,0,
\left( \frac{5}{3}+\frac{4}{3(\Gamma-1)} \right) \epsilon + \frac{\Delta t_n }{4\Delta x} \right)^{\top} .
$$
Because of the continuity of $\tilde q\left( \overline {\vec U}_{ij}^{n+1} (\epsilon)  \right)$ with respect to $\epsilon$, one has
$$
\mathop {\lim }\limits_{\epsilon \to 0^+ }  \tilde q\left( \overline {\vec U}_{ij}^{n+1} (\epsilon)  \right)
=  \tilde q\left( \overline {\vec U}_{ij}^{n+1} (0)  \right)
= 27 \left( \frac{\Delta t_n }{4\Delta x}  \right)^7 \left( \frac{2\Delta t_n }{\Delta x} +1  \right)^2 \left( \frac{\Delta t_n }{\Delta x} -4 \right) <0,
$$
for any time step-size $\Delta t_n$ satisfying the linear stability condition $\frac{\Delta t_n}{\Delta x} + \frac{\Delta t_n}{\Delta y}\le 1$. The locally sign-preserving property for continuous functions implies that there is a small positive number $\epsilon_0$ such that  $ \tilde q\left( \overline {\vec U}_{ij}^{n+1} (\epsilon_0)  \right) <0 $.  Hence  $ \overline {\vec U}_{ij}^{n+1} (\epsilon_0)  \notin {\mathcal G}$ thanks to Remark \ref{rem:rmhd:equal}.
\end{example}

%Thus we guess that some extra conditions are required
%to obtain a PCP scheme for 2D RMHD equations \eqref{eq:RMHD1D}.
%Since the exact solution to the RMHD equations \eqref{eq:RMHD1D} must satisfy the divergence free condition \eqref{eq:2D:BxBy0},
%it is natural to conjecture that these extra conditions may relate to the divergence free condition for the magnetic field.
The above example shows clearly that it is necessary for a PCP RMHD code to preserve the discrete divergence-free condition,
 and the locally divergence-free condition of magnetic field within the cell can not ensure the PCP property even for a first-order accurate scheme.
The  divergence-free MHD code is  very important in the MHDs, see e.g. \cite{Brackbill1980,Evans1988,Toth2000} etc.
 The nonzero divergence of the numerical magnetic
field may lead to the generation of non-physical wave or the negative pressure or density \cite{Brackbill1980,Rossmanith2006}.
Although  some works, e.g. \cite{Balsara2012,cheng,Christlieb,Miyoshi}, have discussed
the positivity-preserving schemes for the non-relativistic MHD equations,
up to now no any multi-dimensional MHD numerical scheme is rigorously proved to be PCP in theory.
%no one  revealed the connection between the PCP property and the divergence free condition in theory.
%5 it is still  open  that develop provable PCP schemes for multi-dimensional
%magnetohydrodynamics,

If the scheme \eqref{eq:2DRMHD:LFscheme} satisfies
a discrete divergence-free condition, then one can use
the generalized LxF splitting property in Theorem \ref{theo:RMHD:LLFsplit2D} to
prove
that the scheme \eqref{eq:2DRMHD:LFscheme} is PCP.
%, we  for the magnetic field
%can ensure its PCP property, see the following theorem.

\begin{theorem} \label{theo:2DDivB:LFscheme}
 The solution $\overline {\vec U}_j^n$ of  \eqref{eq:2DRMHD:LFscheme} satisfies
the discrete divergence-free condition
\begin{equation}\label{eq:DisDivB}
\mbox{\rm div} _{ij} \overline {\vec B}^n := \frac{ \left( \overline  B_1\right)_{i+1,j}^n - \left( \overline  B_1 \right)_{i-1,j}^n } {2\Delta x} + \frac{ \left( \overline  B_2 \right)_{i,j+1}^n - \left( \overline B_2 \right)_{i,j-1}^n } {2\Delta y} = 0,
\end{equation}
for all $n \in \mathbb{N} $ and $j$, if \eqref{eq:DisDivB} holds for the discrete initial data $\{\overline {\vec U}_j^0\}$.
\end{theorem}

\begin{proof}
It is proved by  the induction argument  for the time level number $n$.
The conclusion is true for $n=0$ due to the
hypothesis.
Now assume that \eqref{eq:DisDivB} holds for a non-negative integer $n$ and all $j$, and then  check whether the conclusion holds
for $n+1$.
% the following shows \eqref{eq:DisDivB} also holds for $n+1$ and all $j$.
Using \eqref{eq:LFflux}  and noting that the fifth component of $\vec F_1$ and the sixth component of  $\vec F_2$ are zero, one can rewrite
the fifth and sixth equations in \eqref{eq:2DRMHD:LFscheme} as
\begin{align} \label{eq:BxE}
\begin{split}
 \left( \overline B_1\right)  _{i,j}^{n+1} & = \left( 1 - \lambda  \right)  \left(\overline B_1\right)  _{i,j}^{n}
+ \frac{\Delta t_n}{2\Delta x} \left(
\left(\overline B_1\right)  _{i+1,j}^{n}  + \left(\overline B_1\right)  _{i-1,j}^{n}
        \right)  \\
        &
\quad
+ \frac{\Delta t_n}{2\Delta y} \left(
\left(\overline B_1\right)  _{i,j+1}^{n}     +  \left(\overline B_1\right)  _{i,j-1}^{n}
        \right)
+  \frac{  \Delta t_n}{2\Delta y} \Big(
  \Omega_{i,j+1}   - \Omega_{i,j-1}     \Big) ,
\end{split}
\\
%%%%%%%By
\begin{split}
 \left( \overline B_2\right)  _{i,j}^{n+1} & = \left( 1 - \lambda \right)  \left(\overline B_2\right)  _{i,j}^{n}
+ \frac{\Delta t_n}{2\Delta x} \left(
\left(\overline B_2\right)  _{i+1,j}^{n}  + \left(\overline B_2\right)  _{i-1,j}^{n}
        \right)  \\
        &
\quad
+ \frac{\Delta t_n}{2\Delta y} \left(
\left(\overline B_2\right)  _{i,j+1}^{n}     +  \left(\overline B_2\right)  _{i,j-1}^{n}
        \right)
+  \frac{  \Delta t_n}{2\Delta x} \Big(
 - \Omega_{i+1,j}   + \Omega_{i-1,j}     \Big) ,
\end{split}\label{eq:ByE}
\end{align}
where $ \Omega_{ij} $ denotes the sixth component of $\vec F_1 \left( \overline {\vec U}_{ij}^{n} \right)$, and the fact that $ \Omega_{ij} $
is equal to the opposite number of the fifth component of $\vec F_2 \left( \overline {\vec U}_{ij}^{n} \right)$ has been used.
Since the operator $\mbox{div}_{ij}$ in \eqref{eq:DisDivB} is linear, using \eqref{eq:BxE} and \eqref{eq:ByE} gives
\begin{equation} \label{eq:DivB:n+1}
\begin{split}
\mbox{div}_{ij} \overline {\vec B}^{n+1}  =& \left( 1 - \lambda  \right)  \mbox{div}_{ij} \overline {\vec B}^{n}
+ \frac{\Delta t_n}{2\Delta x} \left(
\mbox{div}_{i+1,j} \overline {\vec B}^{n}  +  \mbox{div}_{i-1,j} \overline {\vec B}^{n}
        \right)
        \\  %\nonumber
&
+ \frac{\Delta t_n}{2\Delta y} \left(
\mbox{div}_{i,j+1} \overline {\vec B}^{n}
 +  \mbox{div}_{i,j-1} \overline {\vec B}^{n}
        \right)
        \\  %\nonumber
&  +  \frac{  \Delta t_n}{2\Delta x\Delta y} \bigg[ \Big(
  \Omega_{i+1,j+1}   - \Omega_{i+1,j-1}     \Big)
  - \Big(
  \Omega_{i-1,j+1}   - \Omega_{i-1,j-1}     \Big)
   \bigg]
\\  %\nonumber
&
+\frac{  \Delta t_n}{2\Delta x \Delta y} \bigg[ \Big(
 - \Omega_{i+1,j+1}   + \Omega_{i-1,j+1}
  \Big)
  - \Big(
 - \Omega_{i+1,j-1}   + \Omega_{i-1,j-1}
  \Big)
   \bigg]\\
 =& \left( 1 - \lambda  \right)  \mbox{div}_{ij} \overline {\vec B}^{n}
+ \frac{\Delta t_n}{2\Delta x} \left(
\mbox{div}_{i+1,j} \overline {\vec B}^{n}  +  \mbox{div}_{i-1,j} \overline {\vec B}^{n}
        \right)
%\nonumber
\\
&+ \frac{\Delta t_n}{2\Delta y} \left(
\mbox{div}_{i,j+1} \overline {\vec B}^{n}     +  \mbox{div}_{i,j-1} \overline {\vec B}^{n}
        \right)
 = 0,
 \end{split}
\end{equation}
where the induction hypothesis has used in the last {equal} sign.  Hence \eqref{eq:DisDivB} holds
for all $n \in \mathbb{N} $ and $j$.
\end{proof}

\begin{theorem} \label{theo:2DRMHD:LFscheme}
If  $\overline {\vec U}_{ij}^n =: \left(\overline  D_{ij}^n, \overline { \vec m}_{ij}^n, \overline  {\vec B}_{ij}^n, \overline  E _{ij}^n  \right)^{\top} \in {\mathcal G}$ satisfies the discrete divergence-free condition
\eqref{eq:DisDivB}
%\begin{equation}\label{eq:DisDivB}
%\delta _{ij} \overline {\vec B}^n := \frac{ \left( \overline  B_1\right)_{i+1,j}^n - \left( \overline  B_1 \right)_{i-1,j}^n } {2\Delta x} + \frac{ \left( \overline  B_2 \right)_{i,j+1}^n - \left( \overline B_2 \right)_{i,j-1}^n } {2\Delta y} = 0,
%\end{equation}
for all $i$ and $j$, then under the CFL type condition
\begin{equation}\label{eq:CFL:LF2D}
0< \frac{ c \Delta t_n}{\Delta x} + \frac{ c \Delta t_n}{\Delta y}  \le  1,
\end{equation}
the solution $ \overline {\vec U}_{ij}^{n+1}$ of \eqref{eq:2DRMHD:LFscheme} belongs to ${\mathcal G}$ for all $i$ and $j$, where $c=1$ is the speed of light.
\end{theorem}

\begin{proof}
Substituting  \eqref{eq:LFflux} into \eqref{eq:2DRMHD:LFscheme} gives
 \begin{equation*}
 \begin{split}
\overline {\vec U}_{ij}^{n+1} & =
\frac{\lambda}{ 2\left( \frac{1}{\Delta x} + \frac{1}{\Delta y} \right) }
\bigg[
\frac{1}{\Delta x}
\Big( \overline {\vec U}_{i+1,j}^n -\vec F_1\left( \overline {\vec U}_{i+1,j}^n\right) +
                         \overline {\vec U}_{i-1,j}^n +\vec F_1\left( \overline {\vec U}_{i-1,j}^n\right) \Big)\\
& \quad +  \frac{1}{\Delta y}\Big( \overline {\vec U}_{i,j+1}^n -\vec F_2\left( \overline {\vec U}_{i,j+1}^n\right) +
                         \overline {\vec U}_{i,j-1}^n +\vec F_2\left( \overline {\vec U}_{i,j-1}^n\right) \Big) \bigg]
                         + (1-\lambda) \overline {\vec U}_j^n
\\
& =:    \lambda \Xi + (1-\lambda) \overline {\vec U}_{ij}^n,
\end{split}
\end{equation*}
where $\lambda := \Delta t_n \left( \frac{1}{\Delta x} + \frac{1}{\Delta y} \right) \in (0,1]$ due to \eqref{eq:CFL:LF2D}.  Using the condition \eqref{eq:DisDivB} and
Theorem \ref{theo:RMHD:LLFsplit2D} gives $\Xi \in {\mathcal G}$.
The convexity of $\mathcal G$ further yields $\overline {\vec U}_{ij}^{n+1} \in {\mathcal G}$. The proof is completed.
\end{proof}

%However, it is usually non-trivial to preserve the discrete divergence free condition \eqref{eq:DisDivB} in the computation at the same time.
%An exciting discovery indicates that the scheme \eqref{eq:2DRMHD:LFscheme} can also preserve the condition \eqref{eq:DisDivB}, if the initial data
%satisfy \eqref{eq:DisDivB} for $n=0$ and all $j$, see the following theorem.

%Theorems \ref{theo:2DRMHD:LFscheme} and  \ref{theo:2DDivB:LFscheme} yield that the Lax-Friedrichs scheme \eqref{eq:2DRMHD:LFscheme} is PCP
%and preserve the discrete divergence free condition \eqref{eq:DisDivB},
Let us discuss how to get the discrete initial data which are admissible, i.e.  $\overline {\vec U}_{ij}^0 \in{\mathcal G}$, and satisfy the condition \eqref{eq:DisDivB} for all $j$.
%corresponding to} the discrete initial cell average values
%$\{\overline {\vec U}_{ij}^0\}$  are calculated by
After giving initial data $(\rho,\vec v,\vec B,p)(x,y,0)$,
calculate the cell average values of the initial primitive variables $(\rho,\vec v,\vec B,p)$ by
\begin{equation}\label{eq:intitialRHO}
\Big( \overline \rho_{ij}^0, \overline {\vec v}_{ij}^0, \left(\overline B_3\right)_{ij}^0, \overline p_{ij}^0 \Big)
= \frac{1}{\Delta x \Delta y} \iint_{I_{ij}} \big( \rho,\vec v, B_3,p \big) (x,y,0)~ dx dy,
\end{equation}
and
\begin{equation}\label{eq:initialBxBy}
\left( \overline B_1 \right)_{ij}^0 = \frac{1}{2 \Delta y} \int_{ y_{j-1} }^{ y_{j+1 } }  B_1(x_i,y,0)~ dy,
\quad
\left( \overline B_2 \right)_{ij}^0 = \frac{1}{2 \Delta x} \int_{ x_{i-1} }^{ x_{i+1 } }  B_2(x,y_j,0)~  dx,
\end{equation}
for  each $i$ and $j$, then $\overline {\vec U}_{ij}^0=  {\vec U}(\overline {\vec V}_{ij}^0)$
 belongs to ${\mathcal G}$ and satisfies the condition \eqref{eq:DisDivB} for all $j$.
In fact, one has $\overline \rho_{ij}^0>0$, $\overline p_{ij}^0 > 0 $, and
\begin{equation*}
\begin{split}
\left|  \overline {\vec v}_{ij}^0 \right|^2
& =  \sum\limits_{k = 1}^3
\left(
\iint_{I_{ij}}  \frac{1}{\Delta x \Delta y}  \cdot { v_k} (x,y,0) ~dx dy \right)^2\\
 &
 \le
 \sum\limits_{k = 1}^3
 \left(
 \iint_{I_{ij}}
\left(  \frac{1}{\Delta x \Delta y} \right)^2  dx dy \right)
\left(
 \iint_{I_{ij}}
{ v_k^2}  (x,y,0)~  dx dy \right)
\\
&
=
 \frac{1}{\Delta x \Delta y}
\left(
 \iint_{I_{ij}}
v^2 (x,y,0)~  dx dy \right) < 1,
\end{split}
\end{equation*}
where the Cauchy-Schwarz inequality {has} been used in the penultimate inequality. Moreover, with \eqref{eq:initialBxBy}, it holds
\begin{equation*}
\begin{split}
\mbox{div}_{ij} \overline{ \vec B}^0 & =
 \frac{1}{2 \Delta x } \left(
 \frac{1}{2 \Delta y} \int_{ y_{j-1} }^{ y_{j+1 } }  B_1(x_{i+1},y,0) dy
 -  \frac{1}{2 \Delta y} \int_{ y_{j-1} }^{ y_{j+1 } }  B_1(x_{i-1},y,0) dy
 \right)
 \\
 & \quad +
  \frac{1}{2 \Delta y } \left(
 \frac{1}{2 \Delta x} \int_{ x_{i-1} }^{ x_{i+1 } }  B_2(x,y_{j+1},0)  dx
- \frac{1}{2 \Delta x} \int_{ x_{i-1} }^{ x_{i+1 } }  B_2(x,y_{j-1},0)  dx
 \right)
 \\
 & =   \frac{1}{4\Delta x \Delta y } \int_{ x_{i-1} }^{ x_{i+1 } }   \int_{ y_{j-1} }^{ y_{j+1 } }
 \left( \frac{\partial B_1}{\partial x} + \frac{\partial B_2}{\partial y}
  \right) ~dxdy= 0 ,
\end{split}
\end{equation*}
where the divergence theorem and \eqref{eq:2D:BxBy0} for $t=0$ has been used.% in the last two equalities respectively.
%Using these approximate
%admissible primitive variables can obtain the admissible
 In practical computations, the integrals in \eqref{eq:intitialRHO} and \eqref{eq:initialBxBy} can be approximately calculated by some numerical quadratures so that the condition \eqref{eq:DisDivB} may not hold exactly due to the numerical error.
 Fortunately, the discrete divergence error
$ {\mathcal E}_{\infty}^n : = \max_{ij} \left\{\left| \mbox{div}_{ij} \overline {\vec B}^{n} \right| \right\}$ %at time $t=t_n$
  does not grow with $n$
 under the condition \eqref{eq:CFL:LF2D}, because
  using \eqref{eq:DivB:n+1} and the triangular inequality gives
$$
{\mathcal E}_{\infty}^{n+1} \le {\mathcal E}_{\infty}^n.
$$

\subsubsection{High-order accurate  schemes}\label{sec:High2D}

This subsection discusses the high-order accurate PCP schemes for the 2D RMHD equations \eqref{eq:RMHD1D}.
%For the sake of convenience, the  time discretization is limited to the first-order accurate forward Euler method, while the high-order accurate SSP time discretization can be considered similar to the 1D case in Section \ref{sec:High1D}.

Assume that  the approximate solution $\vec U_{ij}^n (x,y)$ at time $t=t_n$ within the cell $I_{ij}$ is either reconstructed in the finite volume methods
from the cell average values $\{\overline {\vec U}_{ij}^n\}$ or evolved in the DG methods. The function $\vec U_{ij}^n (x,y)$ is a vector of the polynomial of degree $K$,  and its cell average value over the cell $I_{ij}$ is equal to $\overline {\vec U}_{ij}^{n}$.
Moreover, let
$\vec U_{i\pm \frac{1}{2},j}^\mp (y)$  %$\vec U_{i+\frac{1}{2},j}^- (y)$,
and  $\vec U_{i,j\pm \frac{1}{2}}^\mp (x)$
%and $\vec U_{i,j+\frac{1}{2}}^+ (x)$
denote the traces of $\vec U_{ij}^n(x,y)$ on the four edges
$
\big\{x_{i\pm \frac{1}{2}} \big\}\times
\big(y_{j-\frac{1}{2}},y_{j+\frac{1}{2}}\big)
$
%,$
%\big\{x_{i+\frac{1}{2}} \big\}\times
%\big(y_{j-\frac{1}{2}},y_{j+\frac{1}{2}}\big)
%$,
and $
\big(x_{i-\frac{1}{2}} , x_{i+\frac{1}{2}} \big)\times
\big\{y_{j\pm \frac{1}{2}}\big\}
$
%, and $
%\big(x_{i-\frac{1}{2}} , x_{i+\frac{1}{2}} \big)\times
%\big\{y_{j+\frac{1}{2}}\big\}$
of the cell $I_{ij}$ respectively.

For the 2D RMHD equations \eqref{eq:RMHD1D},
the finite volume scheme or discrete equation for the cell average value in the DG method  may be given by
\begin{equation}\label{eq:2DRMHD:cellaverage}
\begin{split}
\overline {\vec U}_{ij}^{n+1} = \overline {\vec U}_{ij}^{n}
& - \frac{\Delta t_n}{\Delta x} \sum\limits_{\beta =1}^{Q}  \omega_\beta \bigg(
  \vec {\hat F}_1( \vec U^-_{i+\frac{1}{2},\beta}, \vec U^+_{i+\frac{1}{2},\beta} ) -
  \vec {\hat F}_1( \vec U^-_{i-\frac{1}{2},\beta}, \vec U^+_{i-\frac{1}{2},\beta})
  \bigg)  \\
&-  \frac{ \Delta t_n }{\Delta y} \sum\limits_{\alpha =1}^{Q} \omega_\alpha \bigg(
 \vec {\hat F}_2( \vec U^-_{\alpha,j+\frac{1}{2}} , \vec U^+_{\alpha,j+\frac{1}{2}}   ) -
 \vec {\hat F}_2( \vec U^-_{\alpha,j-\frac{1}{2}} , \vec U^+_{\alpha,j-\frac{1}{2}}  )
\bigg),
\end{split}
\end{equation}
which is an approximation of
the equation
\begin{align*} \nonumber
\overline {\vec U}_{ij}^{n+1} & = \overline {\vec U}_{ij}^{n}
 - \frac{\Delta t_n}{\Delta x\Delta y}
\int_{y_{j-\frac{1}{2}}}^{y_{j+\frac{1}{2}}}
  \vec {\hat F}_1( \vec U^-_{i+\frac{1}{2},j}(y), \vec U^+_{i+\frac{1}{2},j}(y) ) -
  \vec {\hat F}_1( \vec U^-_{i-\frac{1}{2},j}(y), \vec U^+_{i-\frac{1}{2},j}(y)) ~dy
  \\
&\quad\quad \ \ -  \frac{ \Delta t_n }{ \Delta x  \Delta y}
\int_{x_{i-\frac{1}{2}}}^{x_{i+\frac{1}{2}}}
 \vec {\hat F}_2( \vec U^-_{i,j+\frac{1}{2}} (x), \vec U^+_{i,j+\frac{1}{2}} (x)) -
 \vec {\hat F}_2( \vec U^-_{i,j-\frac{1}{2}} (x), \vec U^+_{i,j-\frac{1}{2}}  (x))~ dx,
% \label{eq:2DschemeIntergal}
\end{align*}
 by using
%In general, the integrals in \eqref{eq:2DschemeIntergal} should be approximated by quadratures with sufficient accuracy.
%Taking
the Gaussian quadrature for each integral with $Q$ nodes
and the weights $\{ \omega_\alpha\}_{\alpha=1} ^ Q$ satisfying $\sum \limits_{\alpha=1}^Q \omega_\alpha = 1$.
In \eqref{eq:2DRMHD:cellaverage},
$\hat {\vec F}_1$ and  $\hat {\vec F}_2$ denote
 the numerical fluxes in $x$- and $y$-directions  respectively,
 and are taken as the LxF flux defined in \eqref{eq:LFflux}.
Moreover, as shown schematically in Fig. \ref{fig:limit}, the limiting values $ \vec U^\pm_{i+\frac{1}{2},\beta}= \vec U^\pm_{i+\frac{1}{2},j} \big( y_j^\beta \big)  $ and $ \vec U^\pm_{\alpha,j+\frac{1}{2}}= \vec U^\pm_{ i,j+\frac{1}{2}} \big( x_i^{\alpha} \big) $,
where $\{x_i^\alpha\}_{\alpha=1}^{Q}$ and $\{y_j^\alpha \} _{\alpha=1}^{Q}$  denote the Gaussian nodes transformed into the intervals
 $\left[x_{i-\frac{1}{2}},x_{i+\frac{1}{2}} \right]$ and $\left[y_{j-\frac{1}{2}},y_{j+\frac{1}{2}} \right]$, respectively.
For the accuracy requirement,  $Q$ should satisfy:
$Q \ge K+1$ for a $\mathbb{P}^K$-based DG method \cite{Cockburn0},
or $Q \ge (K+1)/2$ for a $(K+1)$-th order accurate finite volume scheme.

  \begin{figure}[htbp]
    \centering
    \includegraphics[height=0.4\textwidth]{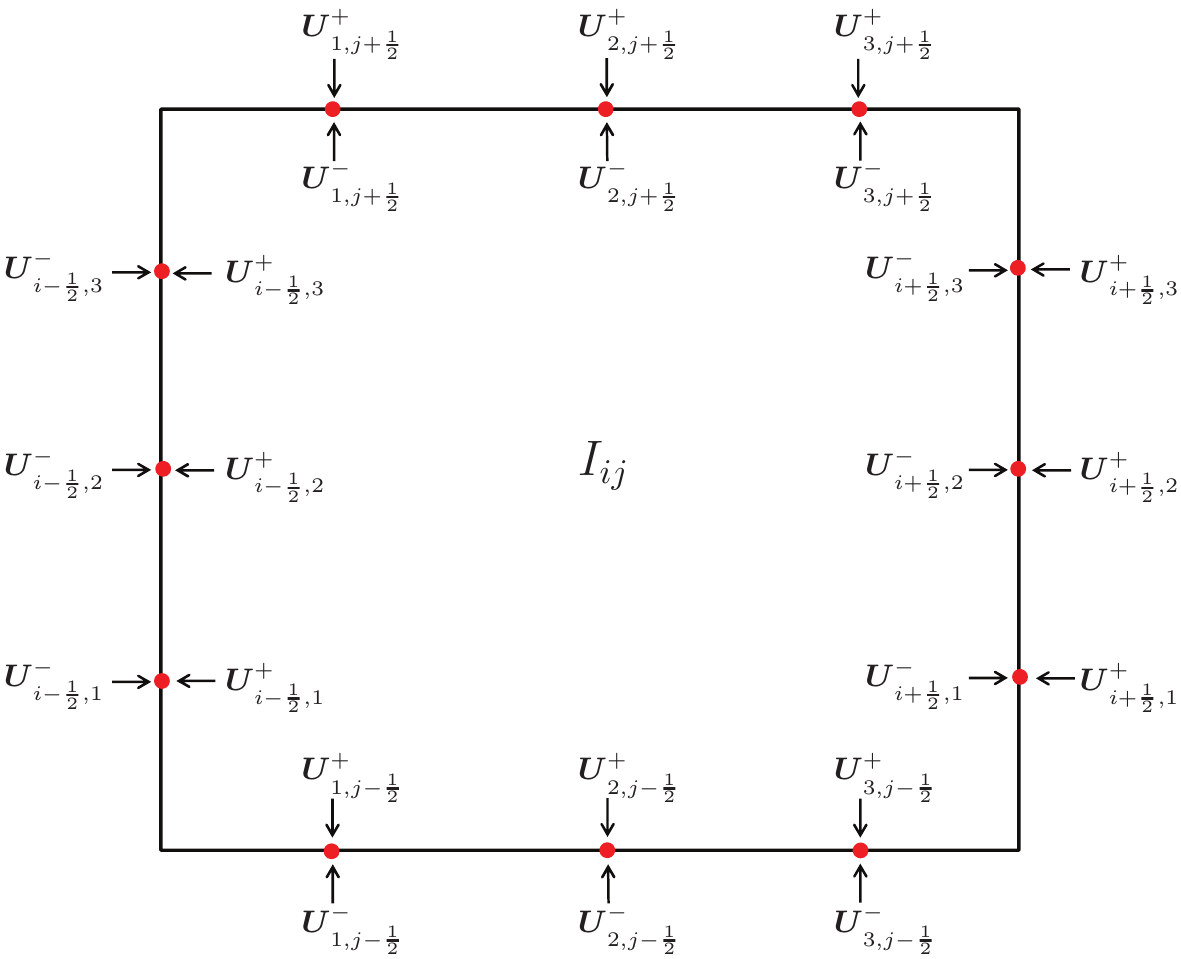}
    \caption{The limiting values at $Q$ Gaussian nodes on  four edges of the cell $I_{ij}$ with $Q=3$.}
    \label{fig:limit}
  \end{figure}

Let $\{ \hat x_i^\alpha \}_{\alpha=1} ^ L$ and $\{ \hat y_j^\alpha \}_{\alpha=1} ^ L$ be the Gauss-Lobatto nodes transformed into the intervals
 $\left[x_{i-\frac{1}{2}},x_{i+\frac{1}{2}} \right]$ and $\left[y_{j-\frac{1}{2}},y_{j+\frac{1}{2}} \right]$ respectively, and
$ \{\hat \omega_\alpha\}_{\alpha=1} ^ L$ be  associated weights satisfying $\sum \limits_{\alpha=1}^L \hat\omega_\alpha = 1$.
Here  $L\ge (K+3)/2$ such that the algebraic precision degree of corresponding quadrature is at least $K$.
%In general, the high-order accurate scheme \eqref{eq:2DRMHD:cellaverage} is not PCP, that is to say, it is possible to meet $\overline {\vec U}_{ij}^{n+1} \notin {\mathcal G}$ even if $\overline {\vec U}_{ij}^{n} \in {\mathcal G}$ for all $i$ and $j$.
Similar to Theorem \ref{thm:PCP:1DRMHD} for the 1D case,
we have the following sufficient conditions for that
the high-order accurate scheme \eqref{eq:2DRMHD:cellaverage} is PCP.

\begin{theorem} \label{thm:PCP:2DRMHD}
If   ${\vec U}_{ij}^n(x,y)=:\big( D_{ij}(x,y), \vec m_{ij}( x,y),\vec B_{ij}( x,y), E_{ij}( x,y) \big)^{\top}$ satisfy:
\\
{\rm (i)}. the discrete divergence-free conditions
\begin{align}\label{eq:DivB:cst1}
&
\mbox{\rm div} _{ij}^{\mbox{\tiny \rm in}} {\vec B} \triangleq \frac{1}{\Delta x} \sum \limits_{\beta=1}^{Q} \omega_\beta \Big(  (B_1)_{i+\frac{1}{2},\beta}^-
 - (B_1)_{i-\frac{1}{2},\beta} ^+  \Big) + \frac{1}{\Delta y} \sum \limits_{\beta=1}^{Q} \omega_\beta \Big(  (B_2)_{\beta,j+\frac{1}{2}}^-
 - (B_2)_{\beta,j-\frac{1}{2}} ^+  \Big) = 0,
 \\ &
\mbox{\rm div} _{ij}^{\mbox{\tiny \rm out}} {\vec B} \triangleq \frac{1}{\Delta x} \sum \limits_{\beta=1}^{Q} \omega_\beta \Big(  ( B_1)_{i+\frac{1}{2},\beta}^+
 - ( B_1)_{i-\frac{1}{2},\beta} ^-  \Big) + \frac{1}{\Delta y} \sum \limits_{\beta=1}^{Q} \omega_\beta \Big(  ( B_2)_{\beta,j+\frac{1}{2}}^+
 - ( B_2)_{\beta,j-\frac{1}{2}} ^-  \Big)  = 0,
\label{eq:DivB:cst2}
\end{align}
for all $i$ and $j$, and
\\
{\rm (ii)}. $\vec U_{ij}^n ( \hat x_i^\alpha,  y_j^\beta ), \vec U_{ij}^n ( x_i^\beta,  \hat y_j^\alpha ) \in {\mathcal G}$, for all $i,j,\alpha,\beta$,
\\
then under the CFL type condition
\begin{equation}\label{eq:CFL:2DRMHD}
0< \frac{\Delta t_n}{\Delta x} + \frac{\Delta t_n}{\Delta y}  \le \hat \omega_1,
\end{equation}
the solution  $\overline{\vec U}_j^{n+1}$ of the scheme \eqref{eq:2DRMHD:cellaverage} belongs to  ${\mathcal G}$.
\end{theorem}

\begin{proof}
The exactness of the Gauss-Lobatto quadrature rule with $L$ nodes and the Gauss quadrature rule with $Q$ nodes for the polynomials of degree $K$ yields
\begin{equation} \label{eq:2D:Gauss1}
\begin{split}
\overline{\vec U}_{ij}^n &= \frac{1}{\Delta x \Delta y} \iint_{I_{ij}} {\vec U}_{ij}^n (x,y)~ dx dy
= \frac{1}{\Delta x} \int_{x_{i-\frac{1}{2}}}^{x_{i+\frac{1}{2}}} \left( \sum \limits_{\beta = 1}^{Q} \omega_\beta \vec U_{ij}^n\left(x,y_j^\beta\right)  \right) dx
\\ %\nonumber
&=  \sum \limits_{\beta = 1}^{Q} \omega_\beta \left( \frac{1}{\Delta x} \int_{x_{i-\frac{1}{2}}}^{x_{i+\frac{1}{2}}}   \vec U_{ij}^n\left(x,y_j^\beta\right) dx  \right)
= \sum \limits_{\beta = 1}^{Q} \omega_\beta \left( \sum \limits_{\alpha = 1}^{L} \hat \omega_\alpha  \vec U_{ij}^n\left(\hat x_i^\alpha,y_j^\beta\right)  \right)\\
&
=  \sum \limits_{\alpha = 1}^{L} \sum \limits_{\beta = 1}^{Q}  \hat \omega_\alpha \omega_\beta  \vec U_{ij}^n\left(\hat x_i^\alpha,y_j^\beta\right) .
\end{split}
\end{equation}
Similarly, one has
\begin{equation}\label{eq:2D:Gauss2}
\overline{\vec U}_{ij}^n = \sum \limits_{\alpha = 1}^{L} \sum \limits_{\beta = 1}^{Q}  \hat \omega_\alpha \omega_\beta  \vec U_{ij}^n\left( x_i^\beta,\hat y_j^\alpha \right).
\end{equation}
By combining \eqref{eq:2D:Gauss1} and   \eqref{eq:2D:Gauss2}, and using $ \vec U_{ij}^n\left(\hat x_i^1,y_j^\beta\right) = \vec U_{i-\frac{1}{2},\beta}^+$,
$\vec U_{ij}^n\left(\hat x_i^L,y_j^\beta\right) = \vec U_{i+\frac{1}{2},\beta}^-$,
$\vec U_{ij}^n\left( x_i^\beta,\hat y_j^1 \right) = \vec U_{\beta,j-\frac{1}{2}}^+$,
$\vec U_{ij}^n\left( x_i^\beta,\hat y_j^L \right) = \vec U_{\beta,j+\frac{1}{2}}^-
 $, and $\hat \omega_1 = \hat \omega_L$, one has
\begin{align*}
\overline{\vec U}_{ij}^n &= \frac{\lambda_x}{\lambda_x+\lambda_y} \overline{\vec U}_{ij}^n + \frac{\lambda_y}{\lambda_x+\lambda_y} \overline{\vec U}_{ij}^n
\\
&= \frac{\lambda_x}{\lambda_x+\lambda_y}  \sum \limits_{\alpha = 1}^{L} \sum \limits_{\beta = 1}^{Q}  \hat \omega_\alpha \omega_\beta  \vec U_{ij}^n\left(\hat x_i^\alpha,y_j^\beta\right) + \frac{\lambda_y}{\lambda_x+\lambda_y} \sum \limits_{\alpha = 1}^{L} \sum \limits_{\beta = 1}^{Q}  \hat \omega_\alpha \omega_\beta  \vec U_{ij}^n\left( x_i^\beta,\hat y_j^\alpha \right) \\
\begin{split}
&= \frac{\lambda_x}{\lambda_x+\lambda_y}  \sum \limits_{\alpha = 2}^{L-1} \sum \limits_{\beta = 1}^{Q}  \hat \omega_\alpha \omega_\beta  \vec U_{ij}^n\left(\hat x_i^\alpha,y_j^\beta\right) + \frac{\lambda_y}{\lambda_x+\lambda_y} \sum \limits_{\alpha = 2}^{L-1} \sum \limits_{\beta = 1}^{Q}  \hat \omega_\alpha \omega_\beta  \vec U_{ij}^n\left( x_i^\beta,\hat y_j^\alpha \right) \\
&\quad + \frac{\lambda_x \hat \omega_1}{\lambda_x+\lambda_y} \sum \limits_{\beta = 1}^{Q}  \omega_\beta \left( \vec U_{i-\frac{1}{2},\beta}^+ +
\vec U_{i+\frac{1}{2},\beta}^- \right)
+ \frac{\lambda_y \hat \omega_1}{\lambda_x+\lambda_y} \sum \limits_{\beta = 1}^{Q}  \omega_\beta \left(  \vec U_{\beta,j-\frac{1}{2}}^+ +
\vec U_{\beta,j+\frac{1}{2}}^- \right) ,
\end{split}
\end{align*}
 where $\lambda_x := \Delta t_n / \Delta x$ and $\lambda_y := \Delta t_n / \Delta y $. Substituting the above identity and  \eqref{eq:LFflux} into \eqref{eq:2DRMHD:cellaverage} gives
\begin{align} \nonumber
\begin{split}
\overline{\vec U}_{ij}^{n+1}
& = \frac{\lambda_x}{\lambda_x+\lambda_y}  \sum \limits_{\alpha = 2}^{L-1} \sum \limits_{\beta = 1}^{Q}  \hat \omega_\alpha \omega_\beta  \vec U_{ij}^n\left(\hat x_i^\alpha,y_j^\beta\right) + \frac{\lambda_y}{\lambda_x+\lambda_y} \sum \limits_{\alpha = 2}^{L-1} \sum \limits_{\beta = 1}^{Q}  \hat \omega_\alpha \omega_\beta  \vec U_{ij}^n\left( x_i^\beta,\hat y_j^\alpha \right) \\ \nonumber
&\quad + \frac{\lambda_x \hat \omega_1}{\lambda_x+\lambda_y} \sum \limits_{\beta = 1}^{Q}  \omega_\beta \left( \vec U_{i-\frac{1}{2},\beta}^+ +
\vec U_{i+\frac{1}{2},\beta}^- \right)
+ \frac{\lambda_y \hat \omega_1}{\lambda_x+\lambda_y} \sum \limits_{\beta = 1}^{Q}  \omega_\beta \left(  \vec U_{\beta,j-\frac{1}{2}}^+ +
\vec U_{\beta,j+\frac{1}{2}}^- \right) \\ \nonumber
&\quad
-\lambda_x \sum\limits_{\beta =1}^{Q}  \omega_\beta \bigg(
 \vec {\hat F}_1( \vec U^-_{i+\frac{1}{2},\beta}, { \vec U}^+_{i+\frac{1}{2},\beta} ) -
 \vec {\hat F}_1( {\vec U}^-_{i-\frac{1}{2},\beta}, \vec U^+_{i-\frac{1}{2},\beta})    \bigg)  \\ \nonumber
&\quad  -  \lambda_y \sum\limits_{\beta =1}^{Q} \omega_\beta \bigg(       \vec {\hat F}_2 ( \vec U^-_{\beta,j+\frac{1}{2}} ,{ \vec U}^+_{\beta,j+\frac{1}{2}}   )
- \vec {\hat F}_2( {\vec U}^-_{\beta,j-\frac{1}{2}} , \vec U^+_{\beta,j-\frac{1}{2}}  )
\bigg)
\end{split}
\\  %\nonumber
%%%%
 \label{eq:2DRMHD:split:proof}
& = (1-2\hat \omega_1) \Xi_1 + \left(\lambda_x + \lambda_y \right) \Xi_2 +  \big( 2\hat \omega_1 - \left(\lambda_x + \lambda_y \right) \big)   \Xi_3,
\end{align}
with
\begin{align*}
\Xi_1
& := \sum \limits_{\alpha = 2}^{L-1} \frac{\hat \omega_\alpha}{1-2\hat \omega_1} \left[
\frac{\lambda_x}{\lambda_x+\lambda_y}  \sum \limits_{\beta = 1}^{Q} \omega_\beta  \vec U_{ij}\left(\hat x_i^\alpha,y_j^\beta\right) + \frac{\lambda_y}{\lambda_x+\lambda_y}  \sum \limits_{\beta = 1}^{Q} \omega_\beta  \vec U_{ij}\left( x_i^\beta,\hat y_j^\alpha \right)
\right],
\\ \nonumber
\begin{split}
\Xi_2
& :=  \frac{1}{2 \left(  \frac{1}{\Delta x} + \frac{1}{\Delta y} \right)}
\sum\limits_{\beta=1}^Q {{\omega _\beta}}
\bigg[
\frac{1}{\Delta x} \bigg(
{ \vec U }_{i+\frac{1}{2},\beta}^+ - \vec F_1 \left(   { \vec U }_{i+\frac{1}{2},\beta}^+ \right)
+  { \vec U }_{i-\frac{1}{2},\beta}^- + \vec F_1 \left(   { \vec U }_{i-\frac{1}{2},\beta}^- \right)
\bigg)\\ \nonumber
&\quad %\quad %\quad \quad \quad \quad
+
\frac{1}{\Delta y} \bigg(
 { \vec U }_{\beta,j+\frac{1}{2}}^+ - \vec F_2 \left(  { \vec U }_{\beta,j+\frac{1}{2}}^+ \right)
+  { \vec U }_{\beta,j-\frac{1}{2}}^- + \vec F_2 \left(  { \vec U }_{\beta,j-\frac{1}{2}}^- \right)
\bigg)
\bigg],
\end{split}
 \\ \nonumber
 \begin{split}
 \Xi_3
 & := \frac{1}{2 \left(  \frac{1}{\Delta x} + \frac{1}{\Delta y} \right)}
\sum\limits_{\beta=1}^Q {{\omega _\beta}}
\bigg[
\frac{1}{\Delta x} \bigg(
 { \vec U }_{i+\frac{1}{2},\beta}^- - \theta^{-1} \vec F_1 \left(   { \vec U }_{i+\frac{1}{2},\beta}^- \right)
+  { \vec U }_{i-\frac{1}{2},\beta}^+ + \theta^{-1} \vec F_1 \left(   { \vec U }_{i-\frac{1}{2},\beta}^+ \right)
\bigg)\\ \nonumber
&\quad % \quad %\quad \quad \quad \quad
+
\frac{1}{\Delta y} \bigg(
 { \vec U }_{\beta,j+\frac{1}{2}}^- - \theta^{-1} \vec F_2 \left(   { \vec U }_{\beta,j+\frac{1}{2}}^- \right)
+  { \vec U }_{\beta,j-\frac{1}{2}}^+ + \theta^{-1} \vec F_2 \left(   { \vec U }_{\beta,j-\frac{1}{2}}^+ \right)
\bigg)
\bigg],
\end{split}
\end{align*}
where $\theta :=\frac{2\hat\omega_1}{\lambda_x+\lambda_y} - 1 \ge 1$.
Thanks to the convexity of $\mathcal G$ and the condition (ii), $\Xi_1 \in {\mathcal G}$.
With ${\vec U}_{i\pm \frac{1}{2},\beta}^\mp , {\vec U}_{\beta,j\pm \frac{1}{2}}^\mp \in {\mathcal G}$, $\theta\ge 1$, and the condition \eqref{eq:DivB:cst1},
one has $\Xi_3 \in{\mathcal G}$ by the generalized LxF splitting property in Theorem \ref{theo:RMHD:LLFsplit2D}.
Similarly, utilizing the condition \eqref{eq:DivB:cst2} gives  $\Xi_2 \in{\mathcal G}$.
Thus using \eqref{eq:2DRMHD:split:proof} and the convexity of $\mathcal G$ yields $ \overline{\vec U}_{ij}^{n+1} \in {\mathcal G}$, and completes the proof.
\end{proof}

\begin{remark}
%\color{red}
Both  \eqref{eq:DivB:cst1} or \eqref{eq:DivB:cst2} in the condition (i)
are approximating \eqref{eq:div000} by replacing
$x_0$ and $y_0$ with $x_{i-\frac12}$ and $y_{j-\frac12}$ respectively.
%\begin{align*}
%& \frac{1}{{\Delta x}} \left( \frac{1}{{\Delta y}}\int_{y_{j-\frac12} }^{y_{j+\frac12}} { \big( B_1 (x_{i+\frac12},y) - B_1 (x_{i-\frac12} ,y) \big) dy} \right)   %\nonumber
%\\ &
%+ \frac{1}{{\Delta y}} \left( \frac{1}{{\Delta x}}\int_{x_{i-\frac12} }^{x_{i+\frac12}} {  \big( B_2 (x,y_{j+\frac12})-B_2 (x,y_{j-\frac12} )\big)   dx} \right)   %\nonumber
%\\
% & = \frac{1}{{\Delta x\Delta y}}\int_{I_{ij}} {\left( {\frac{{\partial B_1 }}{{\partial x}} + \frac{{\partial B_2 }}{{\partial y}}} \right)dxdy},  % \label{eq:div111}
% \end{align*}
%which is zero by the divergence free condition \eqref{eq:2D:BxBy0} for $d=2$.
\end{remark}

\begin{remark}
For some high-order finite volume methods,
it only needs to reconstruct
the limiting values $\vec U_{i-\frac{1}{2},\beta}^+$, $\vec U_{i+\frac{1}{2},\beta}^-$, $\vec U_{\beta,j-\frac{1}{2}}^+$,
$\vec U_{\beta,j+\frac{1}{2}}^-$,
instead of the  polynomial vector
${\vec U}_{ij}^n(x,y)$. In this case, the condition {\rm(ii)} in Theorem \ref{thm:PCP:2DRMHD} can
be replaced with the following condition
\begin{equation*}
\begin{split}
&\vec U_{i-\frac{1}{2},\beta}^+, \vec U_{i+\frac{1}{2},\beta}^-, \vec U_{\beta,j-\frac{1}{2}}^+,
\vec U_{\beta,j+\frac{1}{2}}^- \in {\mathcal G},\quad \beta=1,2,\cdots,Q,\\
&
\Xi_1 = \frac{1}{1-2\hat \omega_1} \Bigg(  \overline{\vec U}_{ij}^n - \hat \omega_1\sum \limits_{\beta = 1}^{Q}  \frac{ \omega_\beta }{ {\lambda_x+\lambda_y}}
\bigg(   \lambda_x \Big( \vec U_{i-\frac{1}{2},\beta}^+ +
\vec U_{i+\frac{1}{2},\beta}^-  \Big) + \lambda_y \Big( \vec U_{\beta,j-\frac{1}{2}}^+ +
\vec U_{\beta,j+\frac{1}{2}}^- \Big)   \bigg) \Bigg) \in {\mathcal G},
\end{split}
\end{equation*}
for all $i$ and $j$.
\end{remark}

It is worth emphasizing the above discussions can be extended to non-uniform  or unstructured meshes by using Theorem \ref{theo:RMHD:LLFsplit2Dus}.
Theorem \ref{thm:PCP:2DRMHD} provides
two sufficient conditions (i) and (ii) on the function $\vec U_{ij}^n(x,y)$ reconstructed in the finite volume method or evolved in the DG method
in order to ensure that the  numerical schemes \eqref{eq:2DRMHD:cellaverage} is PCP.
%Those conditions
%are very similar to those for the non-relativistic Euler case in \cite{zhang2010b} but with
The condition (ii) can be easily met by using the PCP limiter similar to that in Section \ref{sec:High1D}, but
Eqs. \eqref{eq:DivB:cst1} and \eqref{eq:DivB:cst2} in the condition (i) are two constraints on the discrete divergence.
 % and  necessary, , see Example \ref{Counterexample}.
%in Section \ref{sec:2D:FirstOrder} is also a argument.
By using the {\em divergence theorem}, it can be seen that the discrete divergence-free condition \eqref{eq:DivB:cst1} may be met if the
 the reconstructed or evolved polynomial vector $(B_1,B_2)_{ij}(x,y)$ is locally divergence-free,
 see e.g. \cite{Li2005,ZhaoTang2016}. The locally divergence-free property
 of $(B_1,B_2)_{ij}(x,y)$ is  not destroyed in the PCP limiting procedure
  since the PCP limiter modifies the vectors $\vec U_{ij}^n(x,y)$ only with a simple scaling.
The condition \eqref{eq:DivB:cst2}  is necessary for a PCP numerical scheme for the RMHD equations,
{see} Example \ref{Counterexample}, where the magnetic vector satisfies
\eqref{eq:DivB:cst1} and the condition {(ii)}.  Numerical results in Section \ref{sec:examples} will further demonstrate
the importance of condition \eqref{eq:DivB:cst2}.
However, it is not easy to meet the condition \eqref{eq:DivB:cst2} because \eqref{eq:DivB:cst2}
 depends on the limiting values of  the magnetic field calculated from the neighboring cells $I_{i\pm1,j}$ and $I_{i,j\pm1}$
 of  $I_{ij}$.
If the polynomials $(B_1,B_2)_{ij}(x,y)$ are globally or exactly divergence-free,
in other words, it is locally divergence-free in $I_{ij}$ with
normal magnetic component continuous across the cell interface,
then \eqref{eq:DivB:cst1} and \eqref{eq:DivB:cst2} are satisfied.
But the PCP limiter with local scaling may destroy the globally or exactly divergence-free property of $(B_1,B_2)_{ij}(x,y)$. Hence, it is nontrivial and still open to design a limiting procedure for the polynomial vector $\vec U_{ij}^n(x,y)$  satisfying two sufficient conditions in Theorem \ref{thm:PCP:2DRMHD} at the same time.

\begin{remark}\label{remark:diveonpcp}
 As the mesh is refined, it can be weakened that violating the condition \eqref{eq:DivB:cst2} impacts on the PCP property,  if the reconstructed or evolved polynomial vector $(B_1,B_2)_{ij}(x,y)$ is locally divergence-free, i.e. $\mbox{\rm div} _{ij}^{\mbox{\tiny \rm in}} {\vec B}=0$.
 In fact, the proof of Theorem \ref{thm:PCP:2DRMHD} shows that the condition \eqref{eq:DivB:cst2} is only related to  $\Xi_2 \in {\mathcal G}$. If assuming that $(B_1,B_2)_{ij}(x,y)$ approximates the exact solution $(B_1,B_2)(x,y,t_n)$ with at least first order, then the continuity of $B_1(x,y,t_n)$ across the edge $\big\{x_{i + \frac{1}{2}} \big\}\times \big(y_{j-\frac{1}{2}},y_{j+\frac{1}{2}}\big)$ implies
$$
(B_1)_{i+\frac{1}{2},\beta}^- = B_1\big(x_{i + \frac{1}{2}},y_j^\beta,t_n\big) + {\mathcal O} (\Delta),\quad
(B_1)_{i+\frac{1}{2},\beta}^+ = B_1 \big(x_{i + \frac{1}{2}},y_j^\beta,t_n \big) + {\mathcal O} (\Delta),
$$
where $\Delta = \min\{\Delta x,\Delta y\}$. Similarly, one has
$$
(B_2)_{\beta,j+\frac{1}{2}}^- = B_2\big( x_i^\beta, y_{j + \frac{1}{2}},t_n\big) + {\mathcal O} (\Delta),\quad
(B_2)_{\beta,j+\frac{1}{2}}^- = B_2\big( x_i^\beta, y_{j + \frac{1}{2}},t_n\big) + {\mathcal O} (\Delta).
$$
It follows that $\mbox{\rm div} _{ij}^{\mbox{\tiny \rm out}} {\vec B}= \mbox{\rm div} _{ij}^{\mbox{\tiny \rm in}} {\vec B} + {\mathcal O}(1)={\mathcal O}(1)$, so that $\Xi_2$ may not belong to ${\mathcal G}$.
%. If $\mbox{\rm div} _{ij}^{\mbox{\tiny \rm out}} {\vec B} \neq 0$
However, $\Xi_2$ is very close to ${\mathcal G}$ in the   sense of that the first component of $\Xi_2$ is positive, and for any ${\vec B}^*,\vec v^* \in \mathbb{R}^3$ with $|\vec v^*|<1$, it holds
$$
\Xi_2 \cdot \vec n^* + p_{m}^* \ge
- \frac{ \vec v^* \cdot \vec B^*  }{2 \left(  \frac{1}{\Delta x} + \frac{1}{\Delta y} \right)} \Big( \mbox{\rm div} _{ij}^{\mbox{\tiny \rm out}} {\vec B} \Big)
= -{\mathcal O} (\Delta),
$$
whose derivation is similar to that of Theorem \ref{theo:RMHD:LLFsplit2D}. Therefore, as $\Delta$ approaches $0$, the convex combination in \eqref{eq:2DRMHD:split:proof} becomes more possibly in ${\mathcal G}$.

\end{remark}

\section{Numerical experiments}\label{sec:examples}

This section conducts numerical experiments on several 1D and 2D challenging RMHD problems
with either large Lorentz factors, strong discontinuities, low plasma-beta $\beta:=p/p_m$, or low rest-mass density or pressure,
to demonstrate our theoretical analyses and the performance of the proposed PCP  limiter.
Without loss of generality, we take the (third-order accurate) ${\mathbb{P}}^2$-based, locally divergence-free DG methods \cite{ZhaoTang2016}, together with
the third-order SSP Runge-Kutta time disretization \eqref{eq:RK1}, as
our base schemes.
According to the analysis in Section \ref{sec:High1D}, the 1D base scheme with the proposed PCP limiter results in a PCP DG scheme.
As discussed in Section \ref{sec:High2D}, the 2D base scheme with such limiter may not be PCP in general, because the discrete divergence-free condition \eqref{eq:DivB:cst2} in Theorem \ref{thm:PCP:2DRMHD} is not strictly satisfied even though the locally divergence-free property can ensure the condition \eqref{eq:DivB:cst1}.
However, it will be shown in the following that the PCP limiter can still improve the robustness of   2D DG method.
To meet the conditions \eqref{eq:CFL:1DRMHD} and \eqref{eq:CFL:2DRMHD}, the time step-sizes in 1D and 2D will be taken as $0.15\Delta x$ and $0.15\left({1}/{\Delta x}+{1}/{\Delta y}\right)^{-1}$, respectively.
Unless otherwise stated,
all the computations are restricted to the equation of state \eqref{eq:EOS} with the adiabatic index $\Gamma = 5/3$.

\begin{example}[Smooth problems] \label{example1D2Dsmooth}\rm
A 1D and a 2D smooth problems are respectively solved within the domain $[0,1]^d$ on the uniform meshes of $N^d$ cells  to test the accuracy of the ${\mathbb{P}}^2$-based DG methods with the proposed PCP limiter.

The 1D problem describes Alfv\'en waves propagating periodically with large velocity of 0.99 and low pressure, and has the exact solution
$$\vec V(x,t)=\big( 1,0, v_2, v_3, 1,
 \kappa v_2,  \kappa v_3, 10^{-2} \big)^{\top}, \quad (x,t)\in [0,1] \times {\mathbb{R}}^+,
$$
where $v_2=0.99 \sin(2\pi \big(x+t/\kappa)\big)$, $v_3 = 0.99 \cos(2\pi \big(x+t/\kappa)\big)$, and $\kappa = \sqrt{1+\rho h W^2}$. While the 2D problem's exact solution is given by
$$\vec V(x,y,t)=\big(1+0.99999999 \sin \big(2 \pi (x+y)\big), 0.9, 0.2, 0, 1, 1, 1, 10^{-2} \big)^{\top}, \quad (x,y,t)\in [0,1]^2 \times {\mathbb{R}}^+,
$$
which describes a RMHD sine wave fast propagating with low density and pressure.

Fig. \ref{fig:smooth} displays the  $l^1$- and $l^2$-errors at $t=1$ and corresponding orders
obtained  by using the proposed DG methods, respectively.
The results show that the theoretical orders are obtained in both 1D and 2D cases, and
the PCP limiting procedure does not destroy the accuracy.

\begin{figure}[htbp]
  \centering
    \subfigure[Errors in $v_2$ for 1D smooth problem]
    {\includegraphics[width=0.48\textwidth]{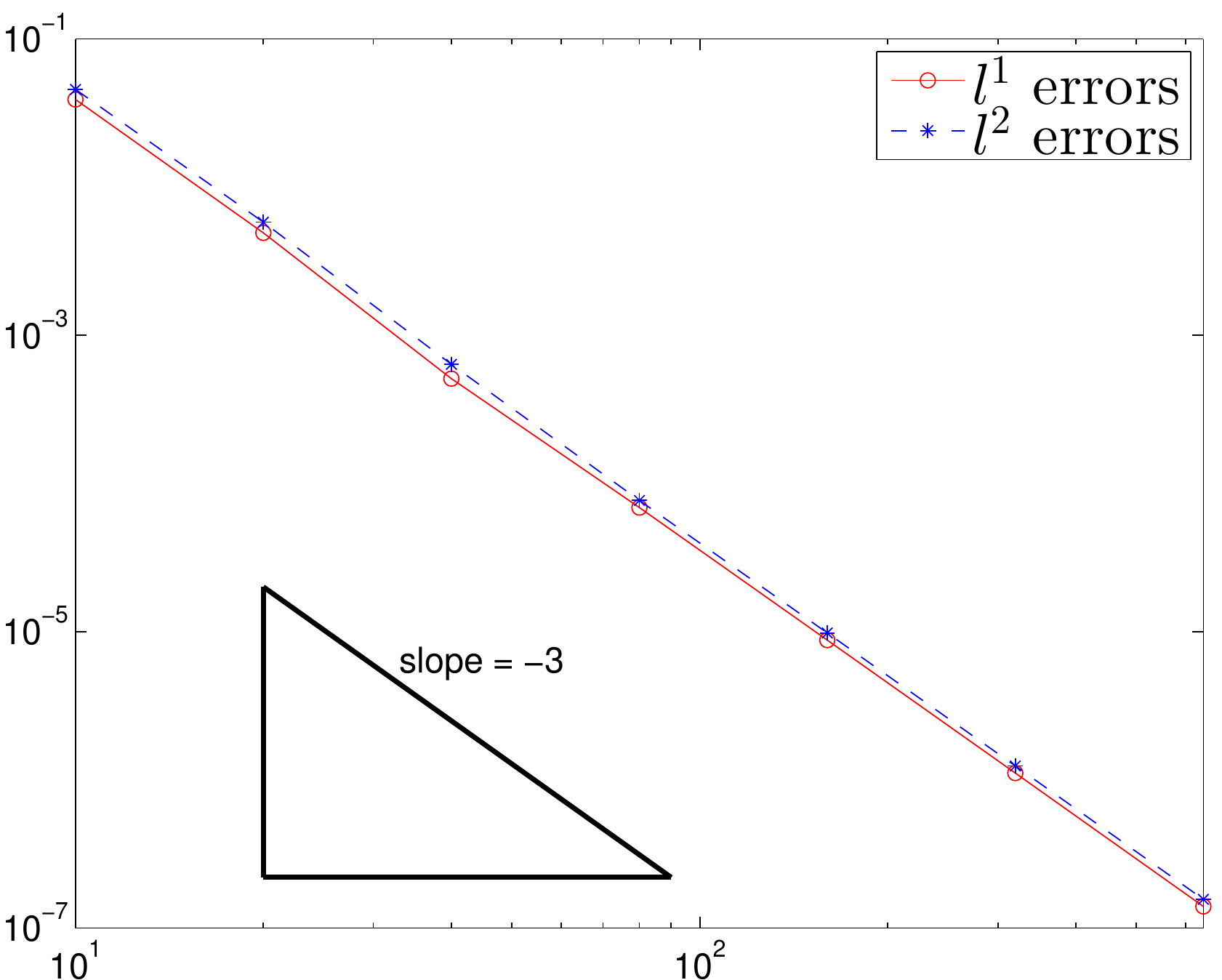}}
  \subfigure[Errors in $\rho$ for 2D smooth problem]
  {\includegraphics[width=0.48\textwidth]{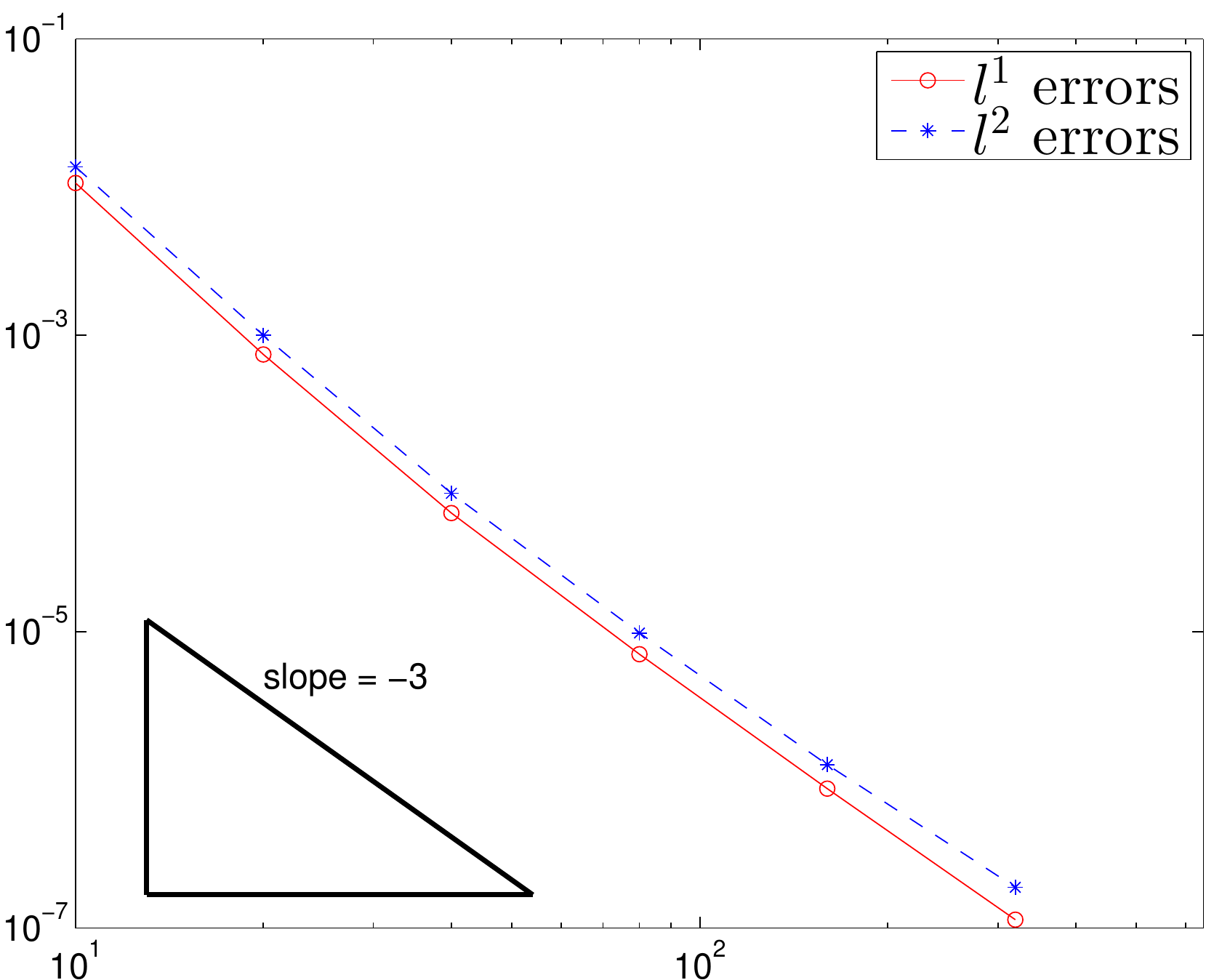}}
  \caption{\small Example \ref{example1D2Dsmooth}: Numerical $l^1$- and $l^2$-errors at $t=1$. The horizontal axis represents the value of $N$.
 }
  \label{fig:smooth}
\end{figure}

\end{example}

To verify the capability of  DG methods with PCP limiter in resolving 1D and 2D ultra-relativistic wave configurations,
three 1D Riemann problems, a 2D rotor problem, a 2D shock and cloud interaction problem, and
 several 2D blast problems will be solved. For those problems, before using  the PCP limiting procedure,
the WENO limiter \cite{Qiu2005,ZhaoTang2016} will be implemented  with the aid of the local
characteristic decomposition \cite{Anton2010} to enhance the numerical stability of high-order DG methods in resolving
the strong discontinuities as well as their interactions.
 Different from \cite{ZhaoTang2016},  the improved WENO proposed in \cite{Borges2008} and the ``trouble'' cell indicator in \cite{Krivodonova} are used here.

\begin{example}[1D Riemann problems] \label{example1DRiemanns}\rm
This example verifies the robustness and effectiveness of the PCP DG method by simulating three 1D Riemann problems (RP), whose initial data comprise two different constant states separated by the initial discontinuity at $x=0$, see Table \ref{tab:riemann1d}. The computational domain is $[-0.5, 0.5]$ and divided into $1000$ uniform cells.

\begin{table}[htbp]
  \centering %\vspace{0.1cm}
    \caption{\small Initial data of the three 1D RPs in Example \ref{example1DRiemanns}.
  }
\begin{tabular}{c|c|cccccccc}
\hline
\multicolumn{2}{c|}{} & $\quad \rho \quad$   & $\quad v_1 \quad$ &  $\quad v_2 \quad$   &  $\quad v_3 \quad$  &  $\quad B_1 \quad$  & $\quad B_2 \quad$  &  $\quad B_3 \quad$  &  $\quad p \quad$  \\
\hline
\hline
\multirow{2}{32pt}{RP I}
&
left state  & 1 & 0  & 0
 & 0 & 5 & 26 & 26 & 30 \\
\cline{2-10}
& right state & 1 & 0  & 0
 & 0 & 5 & 0.7 & 0.7 & 1  \\
\hline
\multirow{2}{32pt}{RP II}
&
left state  & 1 & 0  & 0
 & 0 & 10 & 7 & 7 & $10^4$ \\
\cline{2-10}
& right state & 1 & 0  & 0
 & 0 & 10 & 0.7 & 0.7 & $10^{-8}$  \\
\hline
\multirow{2}{32pt}{RP III}
&
left state  & 1 & 0.99999  & 0
 & 0 & 100 & 70 & 70 & $0.1$ \\
\cline{2-10}
& right state & 1 & -0.99999  & 0
 & 0 & 100 & -70 & -70 & $0.1$ \\
\hline
\end{tabular}\label{tab:riemann1d}
\end{table}

\begin{figure}[htbp]
  \centering
  {\includegraphics[width=0.48\textwidth]{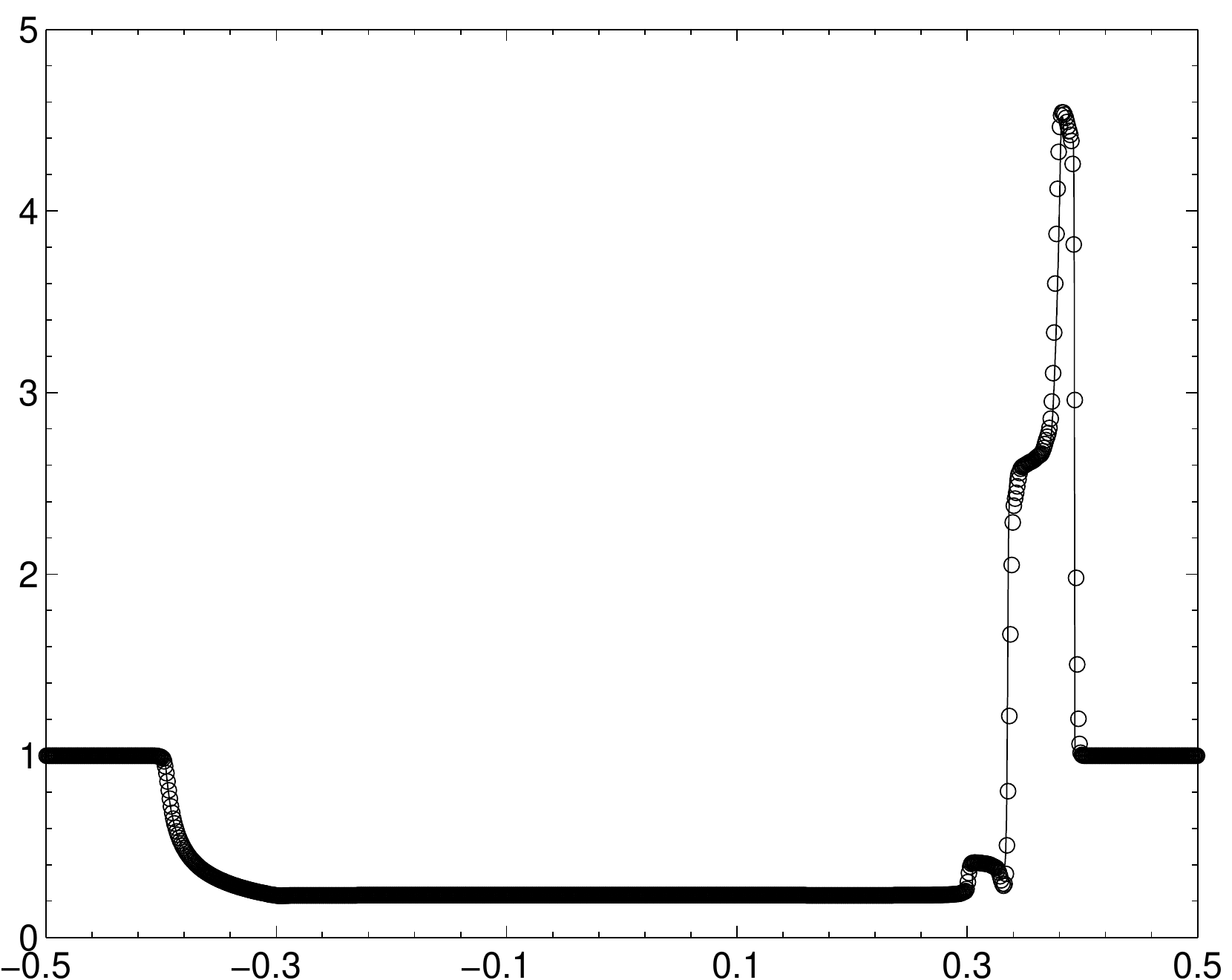}}
  {\includegraphics[width=0.48\textwidth]{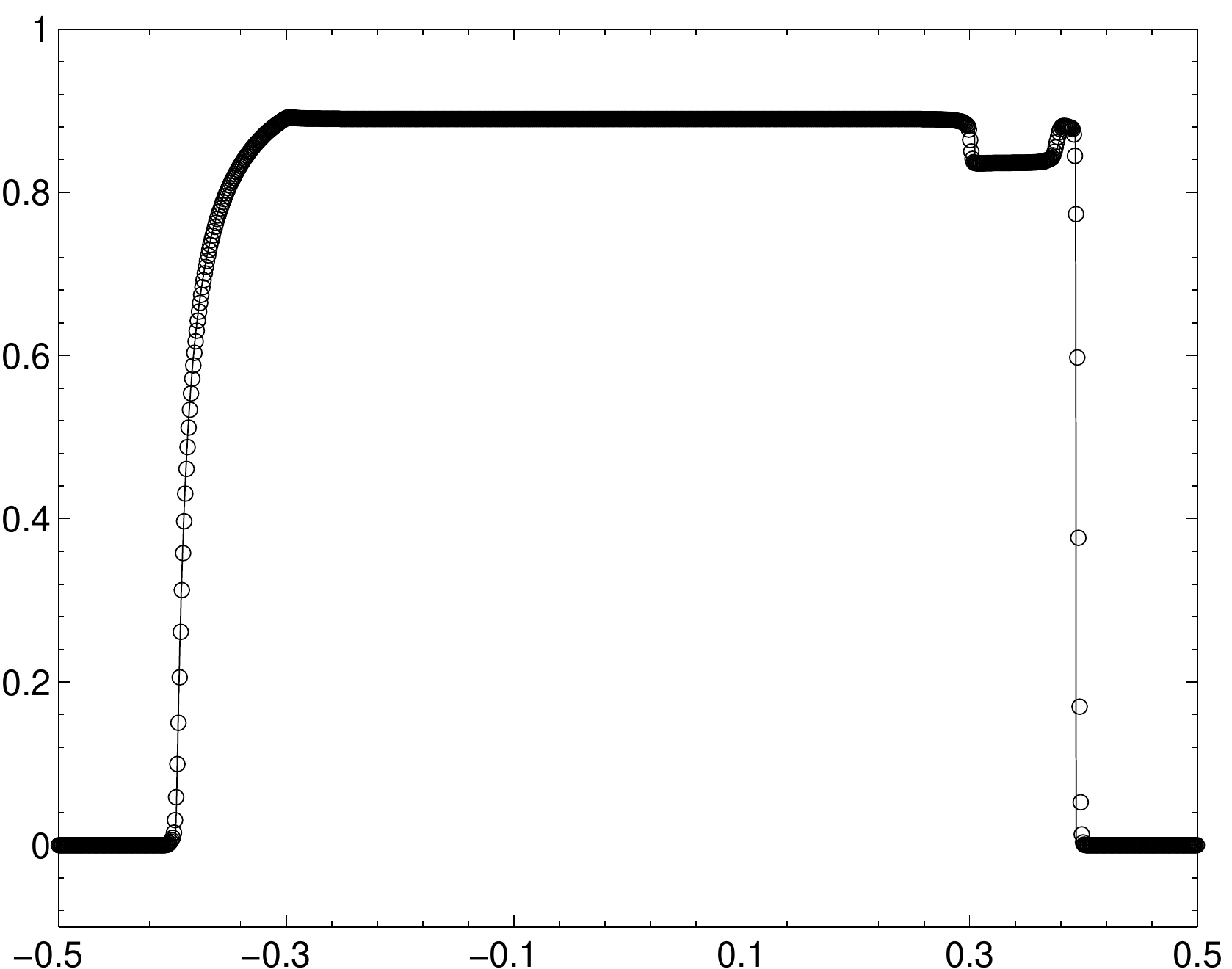}}
  {\includegraphics[width=0.48\textwidth]{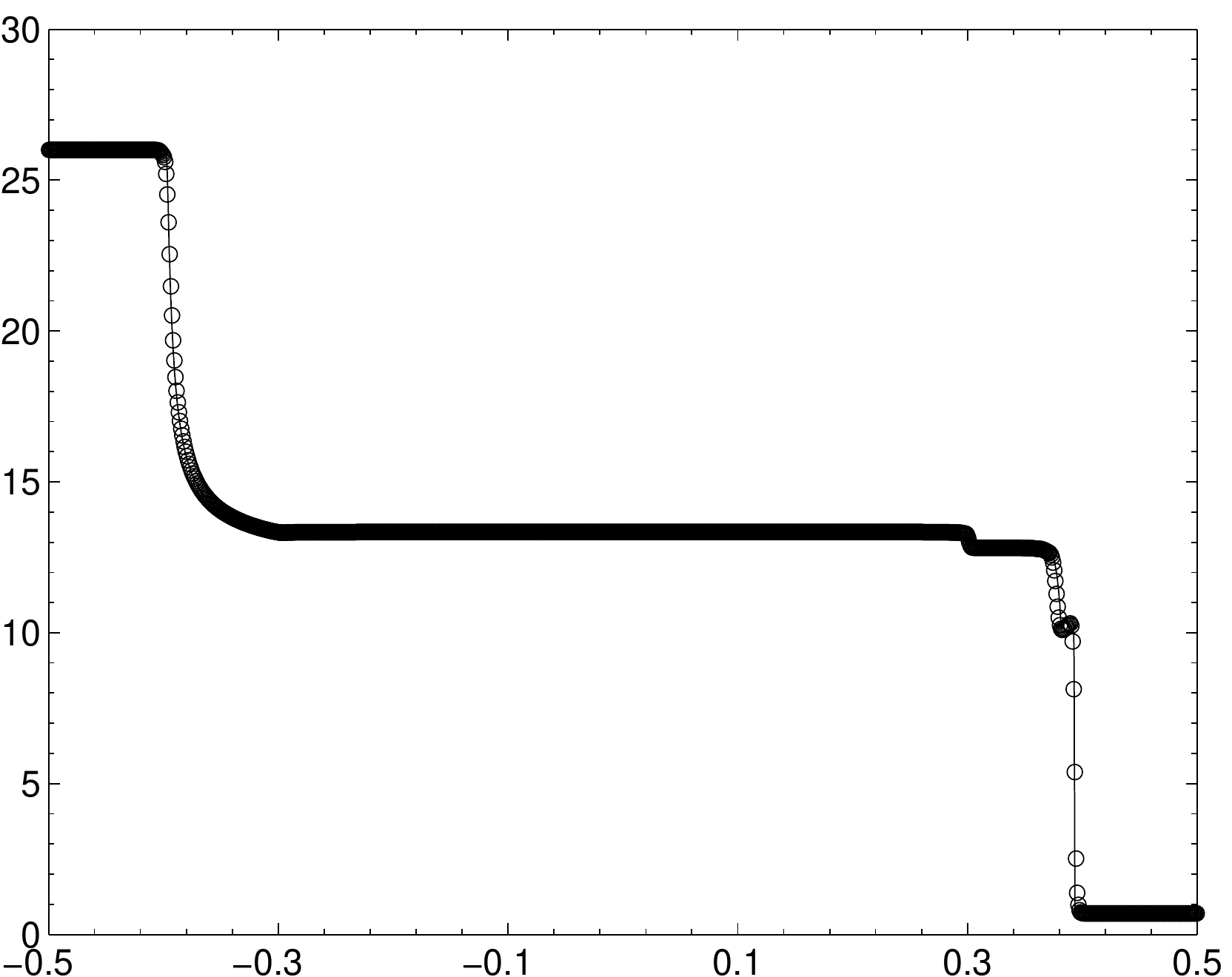}}
  {\includegraphics[width=0.48\textwidth]{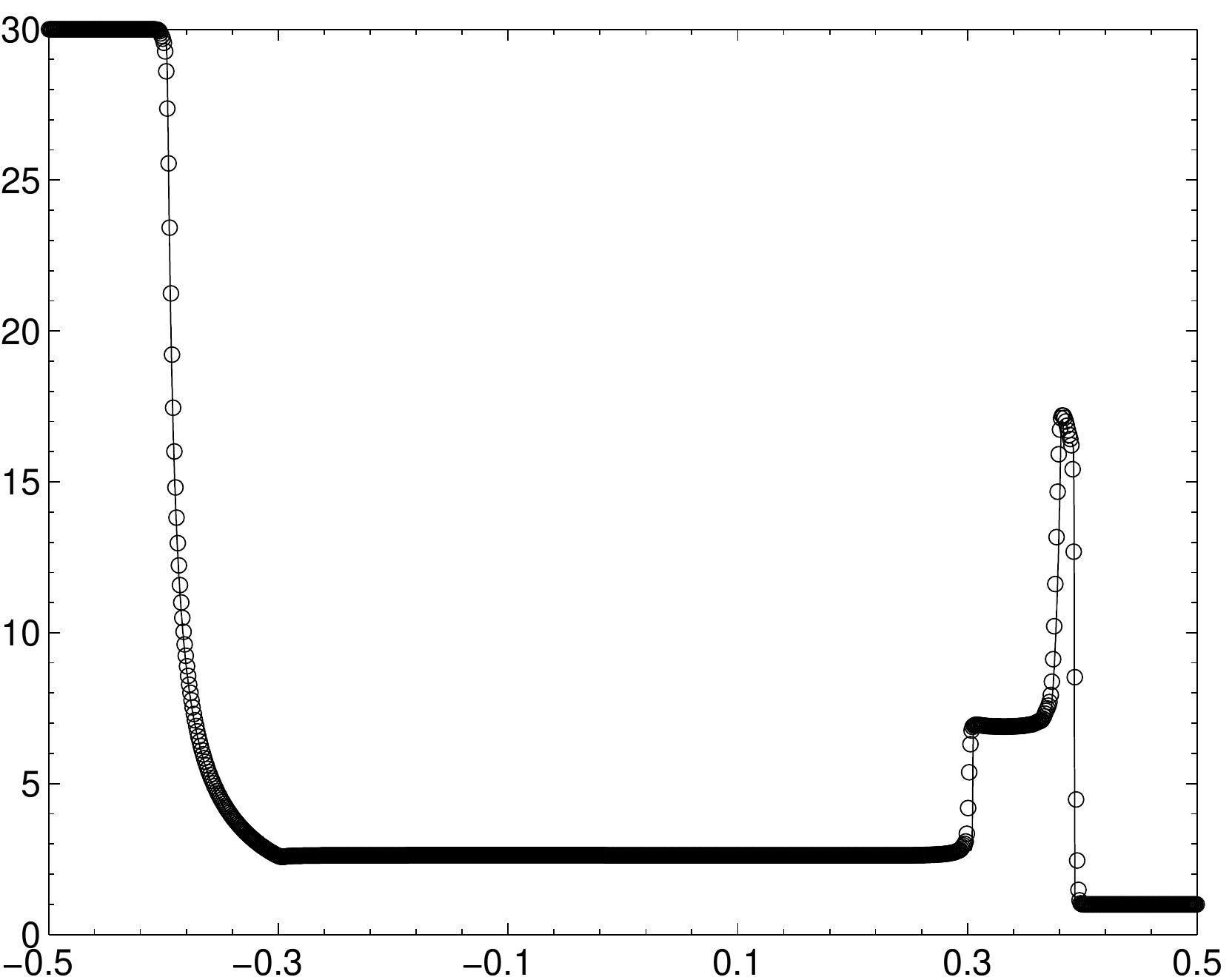}}
  \caption{\small RP I in Example \ref{example1DRiemanns}: The density $\rho$ (top left),
  	velocity $v_1$ (top right), magnetic-field $B_2$ (bottom left), and pressure $p$ (bottom right) at $t=0.4$ obtained by the PCP DG method. The solid lines denote the reference solutions.}
  \label{fig:RPI}
\end{figure}

\begin{figure}[htbp]
  \centering
  {\includegraphics[width=0.48\textwidth]{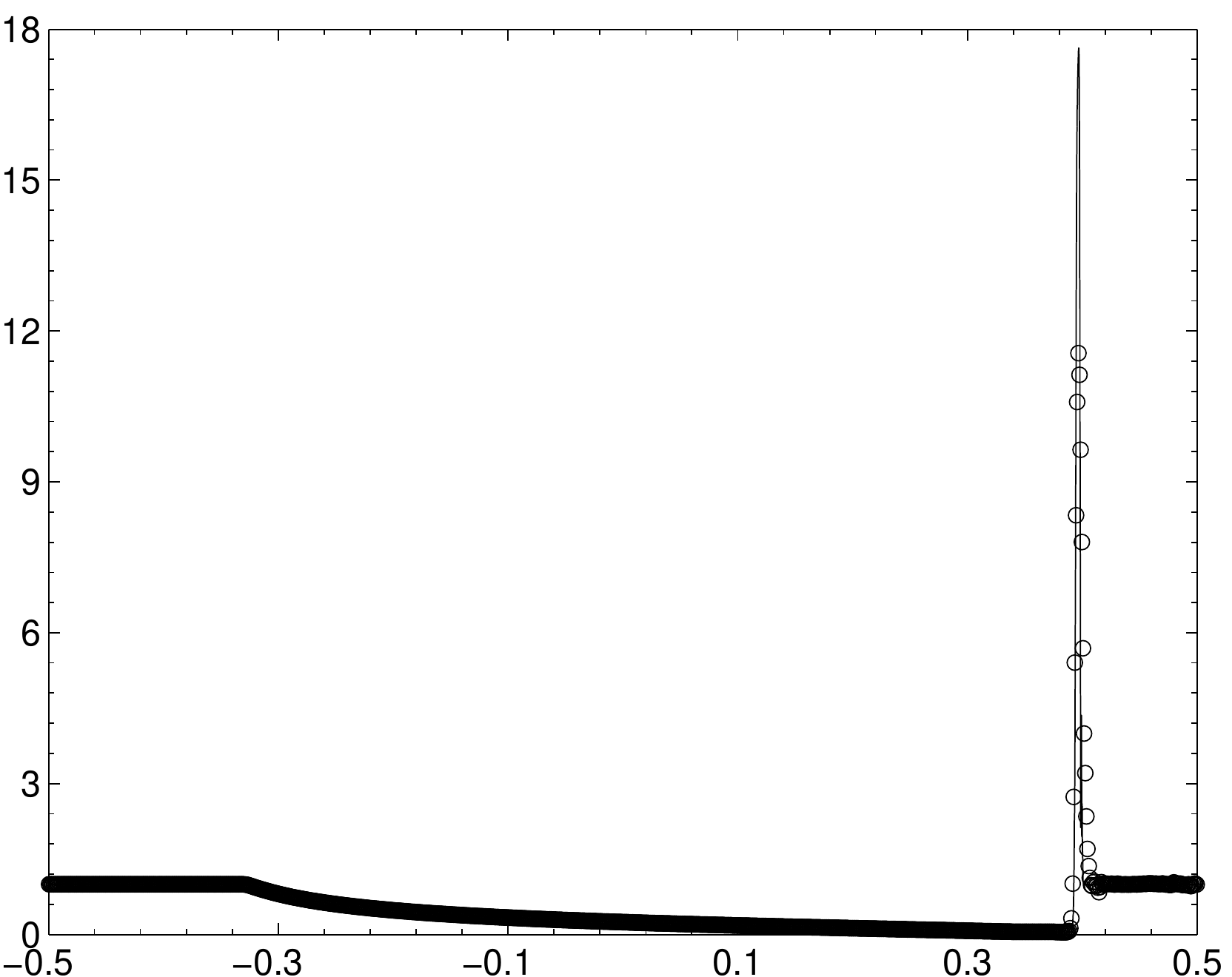}}
  {\includegraphics[width=0.48\textwidth]{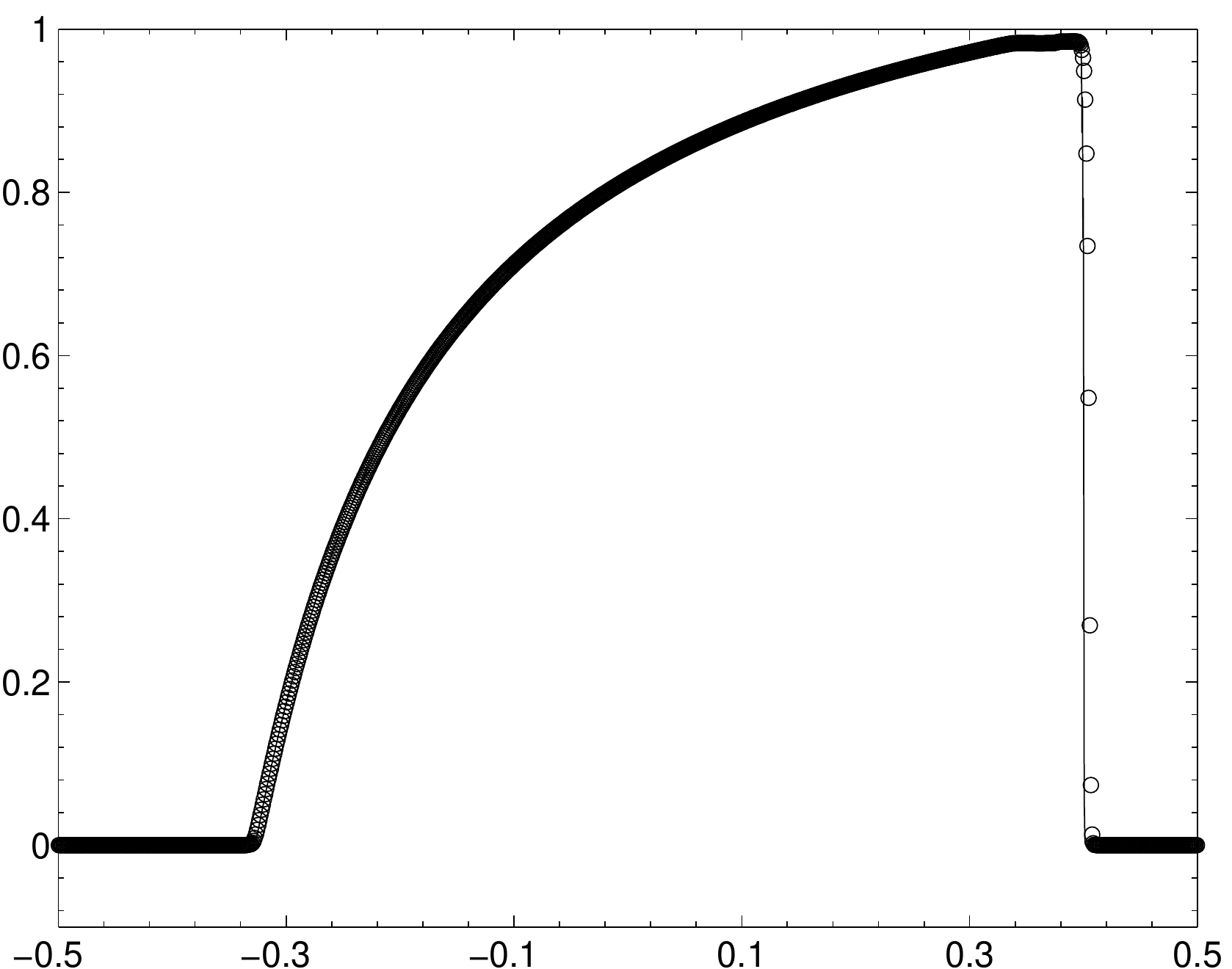}}
  {\includegraphics[width=0.48\textwidth]{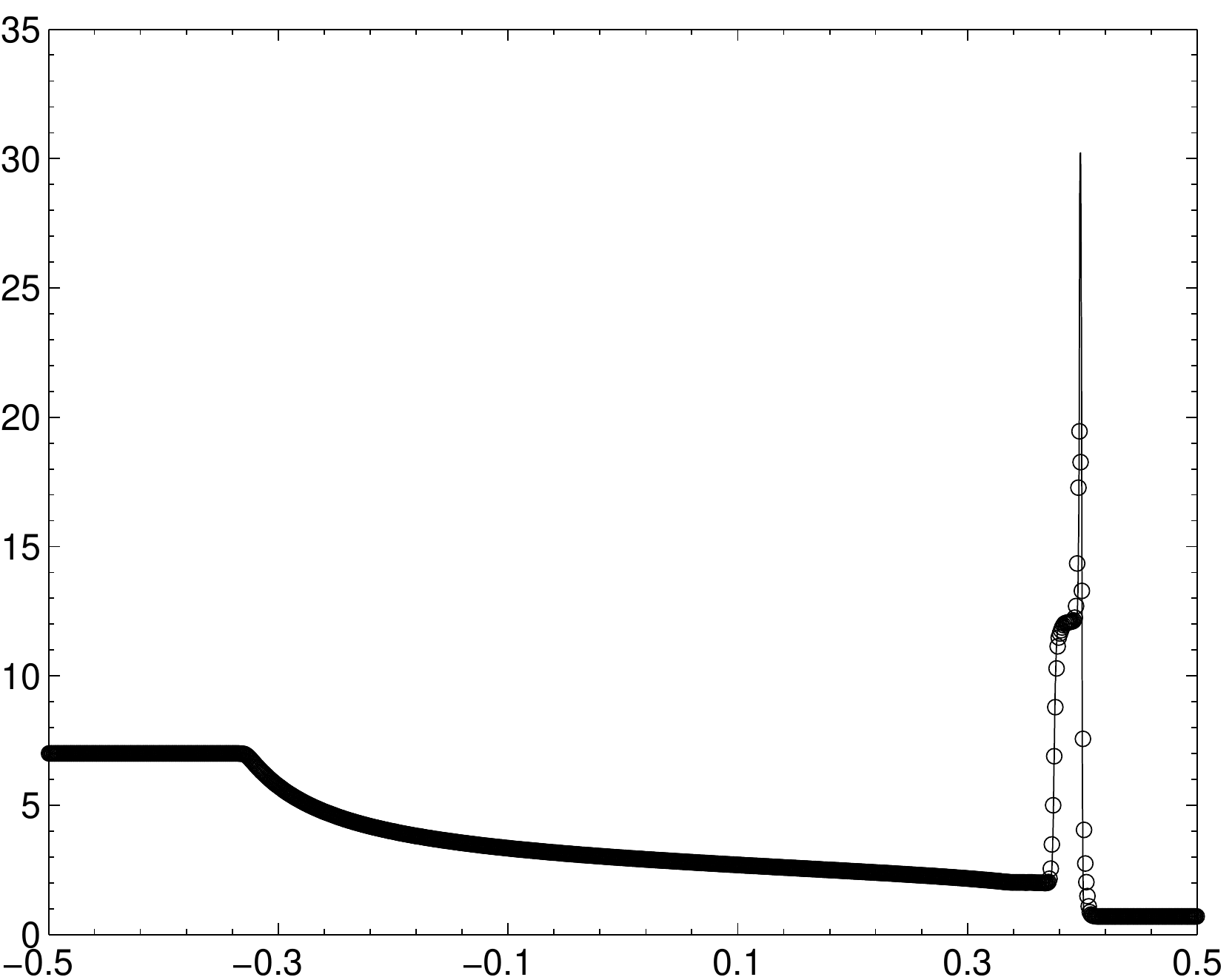}}
  {\includegraphics[width=0.48\textwidth]{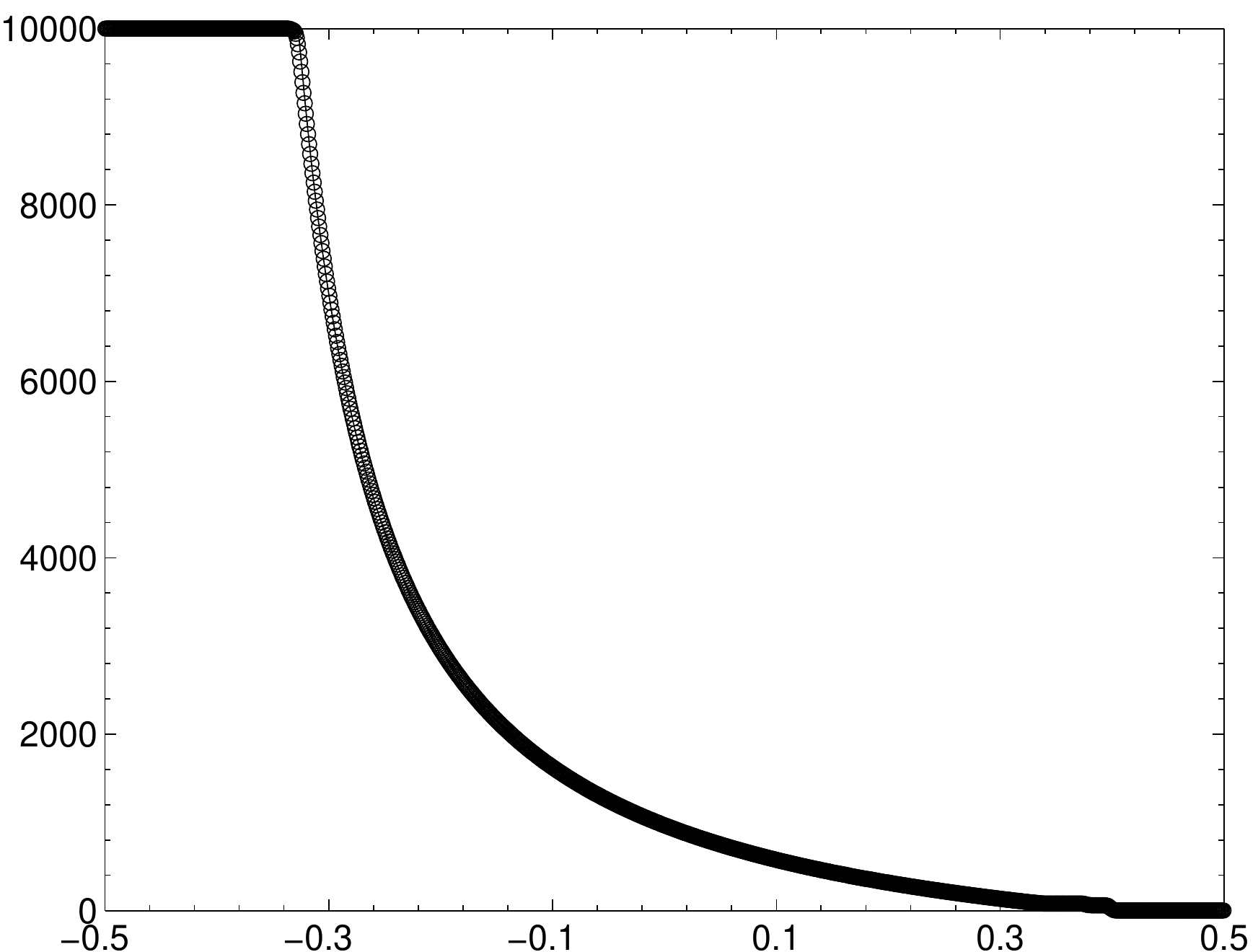}}
  \caption{\small Same as Fig. \ref{fig:RPI} except for RP II.}
  \label{fig:RPII}
\end{figure}

The first two Riemann problems  are  similar to but more ultra than those   1D blast wave problems in   \cite{Balsara2001,Giacomazzo2006}.
Specifically, the stronger magnetic field ($| \vec B | \approx 37.108$) appears in the left state of the first problem, while a very strong initial jump in pressure ($\Delta p :=|p_R-p_L|/p_R \approx 10^{12}$) and extremely low gas pressure (the minimum plasma-beta $\beta:=p/p_m \approx 1.98 \times 10^{-10}$) in the second problem.
The numerical results of those problem at $t=0.4$ obtained by using the  ${\mathbb{P}}^2$-based PCP DG method are displayed by symbols ``{$\circ$}'' in Figs. \ref{fig:RPI} and \ref{fig:RPII} respectively, where and hereafter the solid lines denote the reference solutions obtained by a second-order MUSCL scheme with PCP limiter over the uniform mesh of $20000$ cells.
It is seen that the PCP DG method exhibits good resolution and strong robustness, and the results agree well with the reference ones.
Without employing the PCP limiting procedure, the high-order accurate DG methods will break down quickly
within few time steps due to nonphysical numerical solutions.

\begin{figure}[htbp]
  \centering
  {\includegraphics[width=0.48\textwidth]{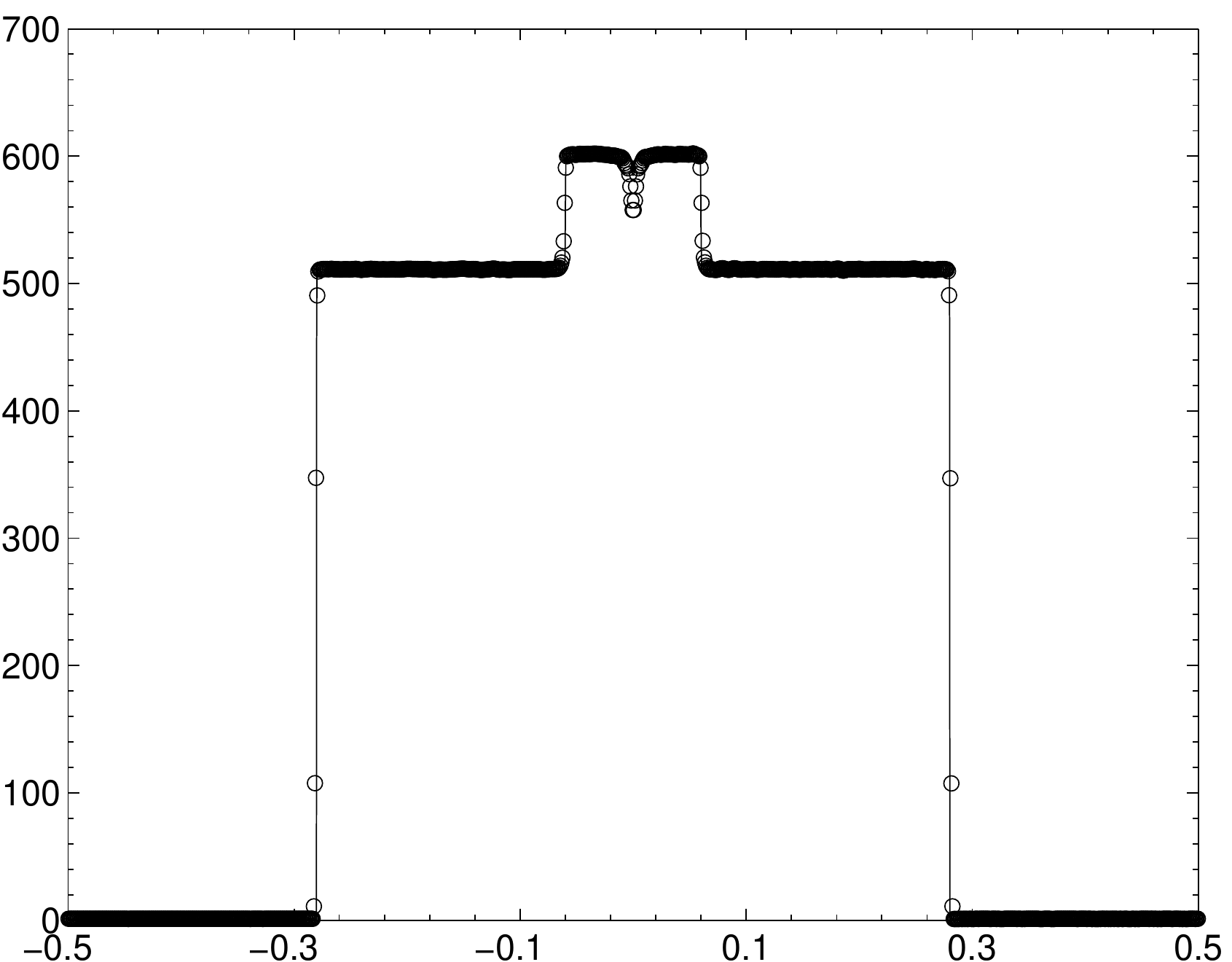}}
  {\includegraphics[width=0.48\textwidth]{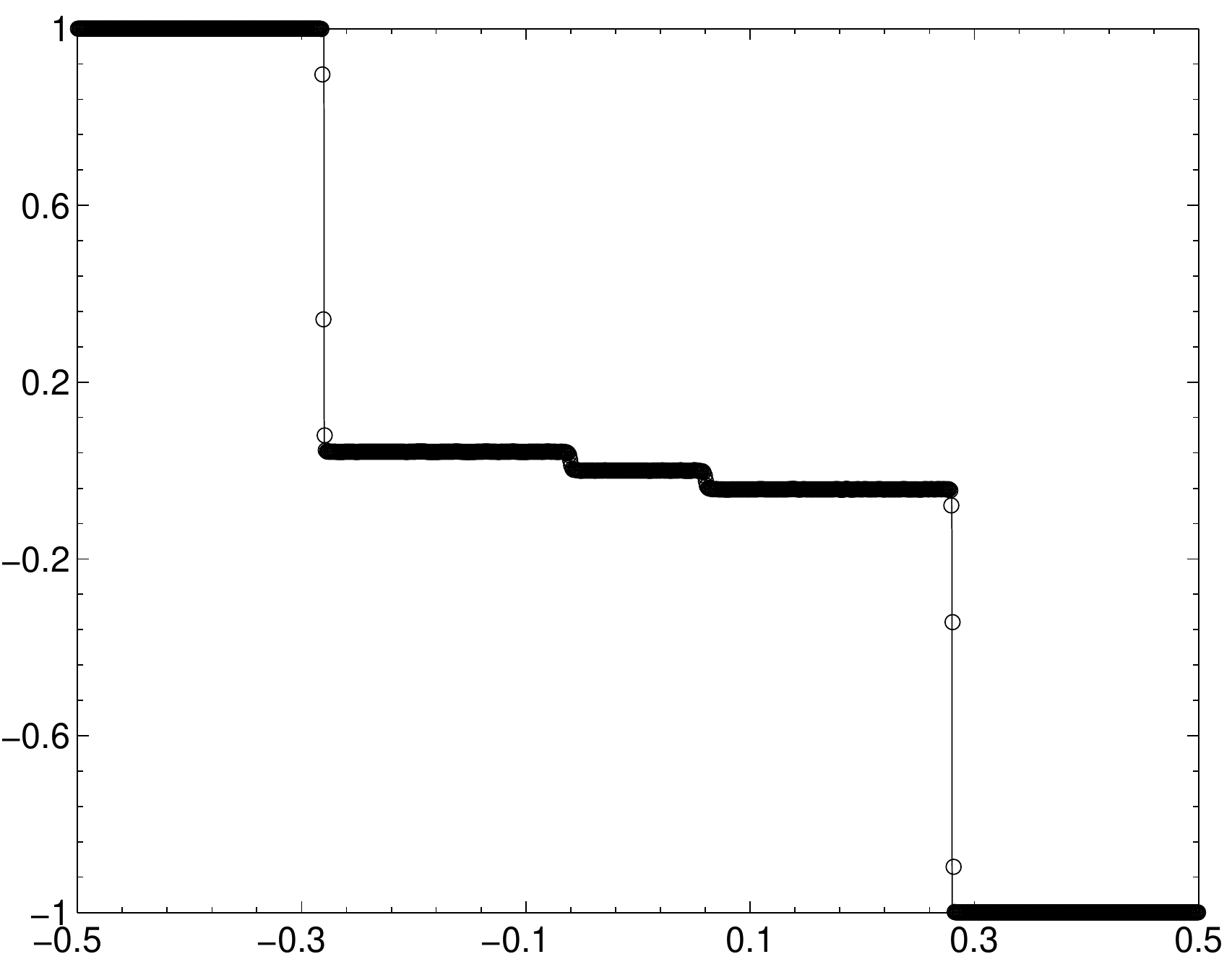}}
  {\includegraphics[width=0.48\textwidth]{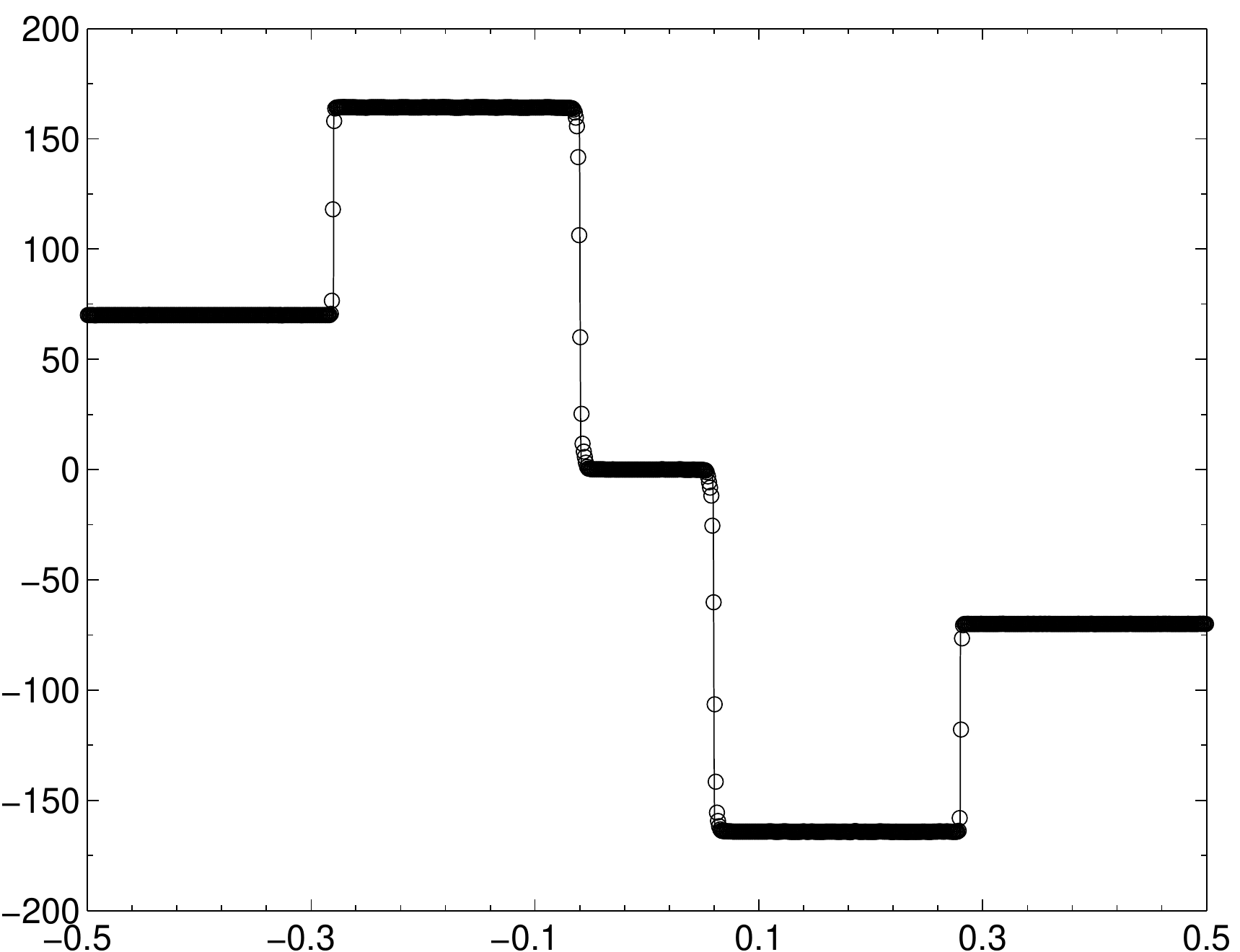}}
  {\includegraphics[width=0.48\textwidth]{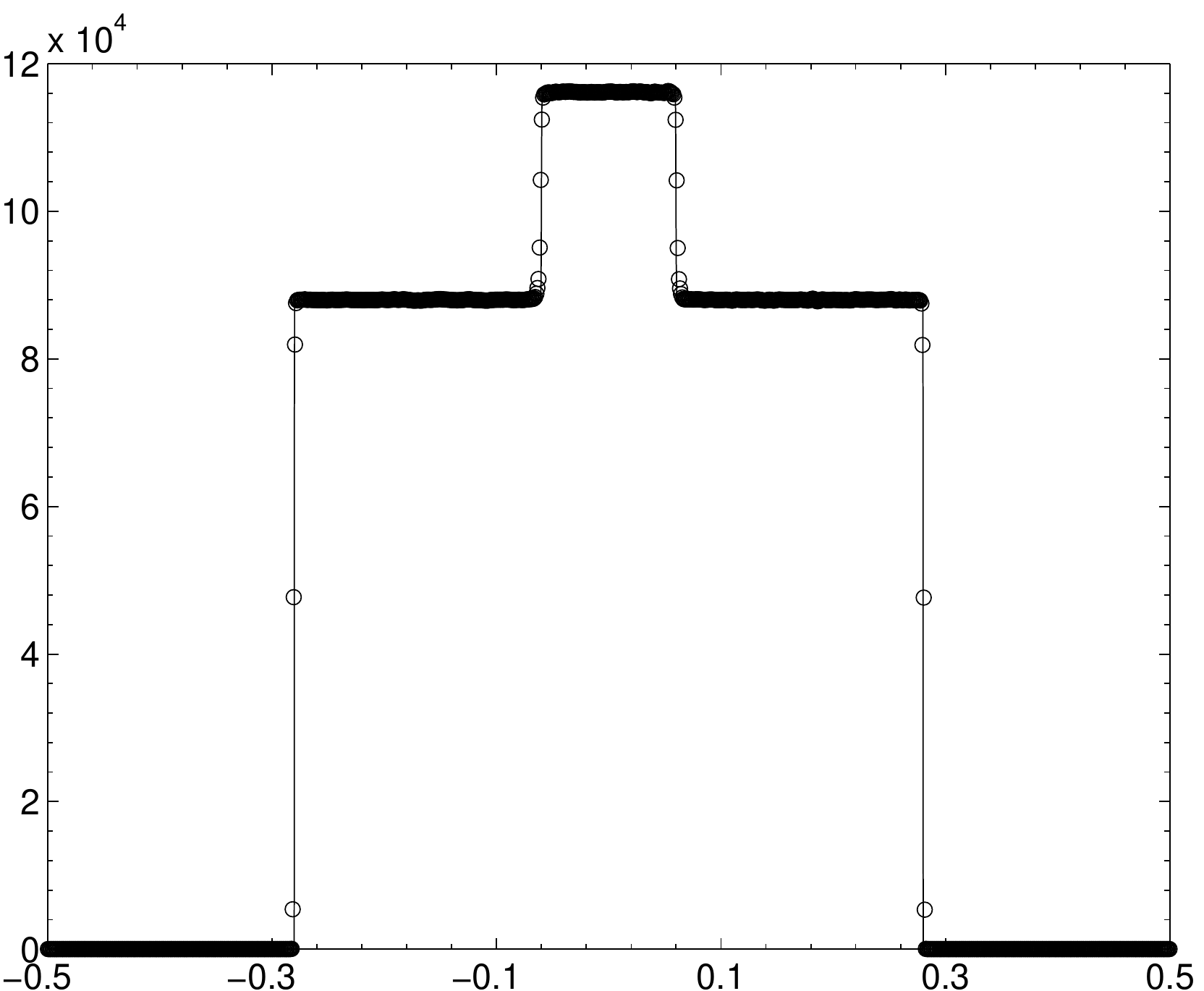}}
  \caption{\small Same as Fig. \ref{fig:RPI} except for RP III.}
  \label{fig:RPIII}
\end{figure}

The third Riemann problem describes the strong collision between two high-speed RMHD flows with a Lorentz factor of about 223.61.
As a result,  it is a very strongly relativistic test problem.
Fig. \ref{fig:RPIII} gives the numerical results at $t=0.4$ obtained by using the ${\mathbb{P}}^2$-based PCP DG method.
As the time increases, we see that two fast and two slow reflected shock waves are produced, and a very high pressure region appears between
the two slow shock waves. Those shock waves are well resolved robustly,
even though there exists the well-known wall-heating type phenomenon around $x=0$, which is often observed
in the literatures e.g. \cite{Balsara2001,HeTang2012RMHD}. It is worth mentioning that the ${\mathbb{P}}^2$-based DG method
fails in the first time step if the PCP limiting procedure is not employed.

\end{example}

\begin{example}[Rotor problem] \label{example2DRo}\rm
This is a benchmark test problem \cite{Zanna:2003} extended from the classical MHD rotor problem and
 has been widely simulated in the literatures.
It is set up on a unit domain $[-0.5,0.5]^2$ with outflow boundary conditions.
Initially, the gas pressure and magnetic field are uniform;
 there is a high-density disk centered at coordinate origin with radius of 0.1,
rotating in anti-clockwise direction at a speed close to $c$; the ambient fluid is homogeneous for $r>0.115$, and the fluid density and velocity
are linearly with respect to $r \in [0.1,0.115]$, where $r=\sqrt{x^2+y^2}$. Specifically, the initial data are
$$
\vec V(x,y,0)=
\begin{cases}
\big(10,-\alpha y,\alpha x,0,1,0,0,1\big)^\top,&   r<0.1,\\
\big(1+9 \delta,-\alpha y \delta/r, \alpha x \delta /r,0,1,0,0,1 \big)^\top,\quad &0.1 \le r \le 0.115, \\
\big(1,0, 0,0,1,0,0,1 \big)^\top,&   r> 0.115,
\end{cases}
$$
with $\delta:=(0.115-r)/0.015$.

\begin{figure}[htbp]
  \centering
  {\includegraphics[width=0.48\textwidth]{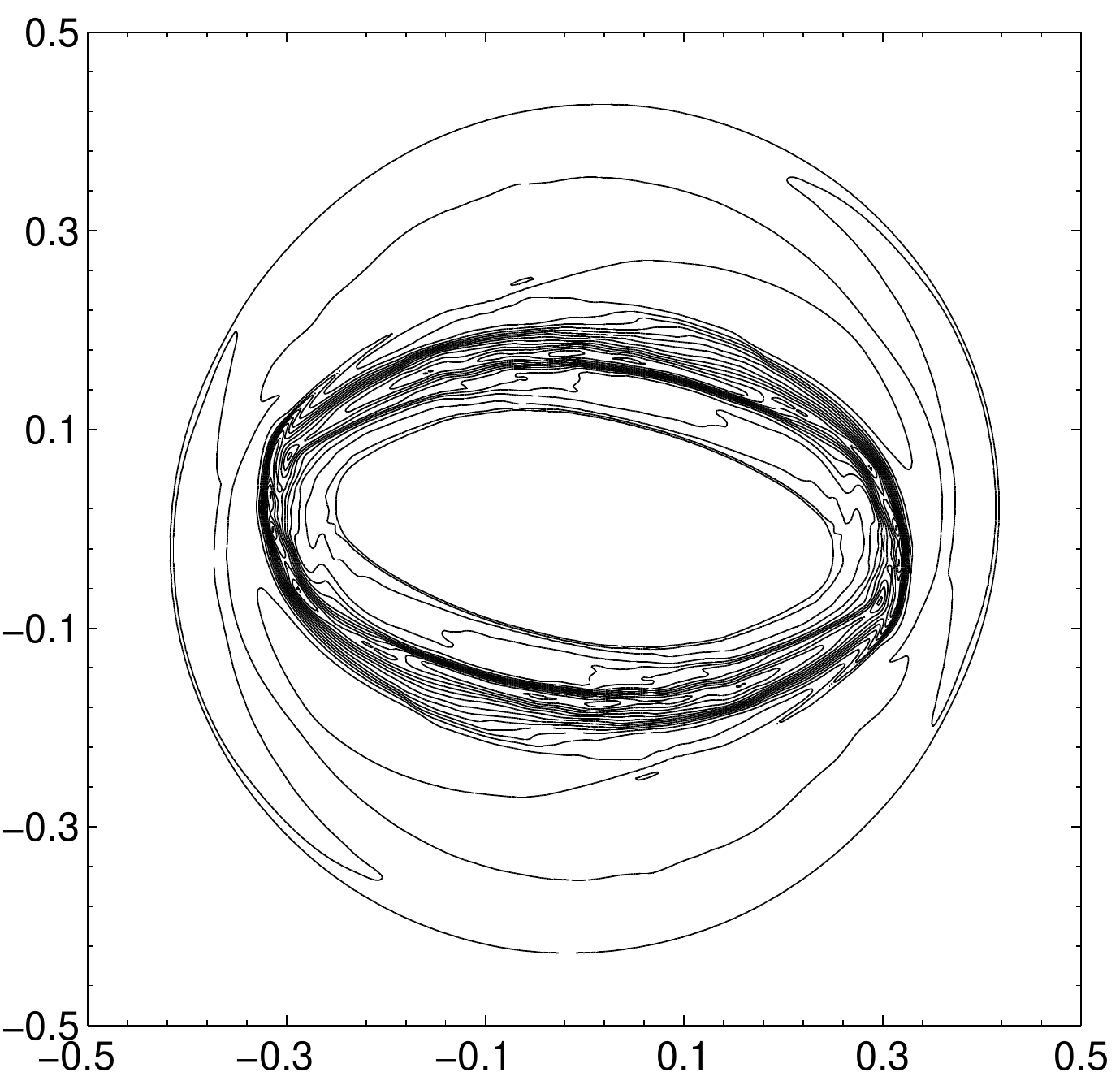}}
  {\includegraphics[width=0.48\textwidth]{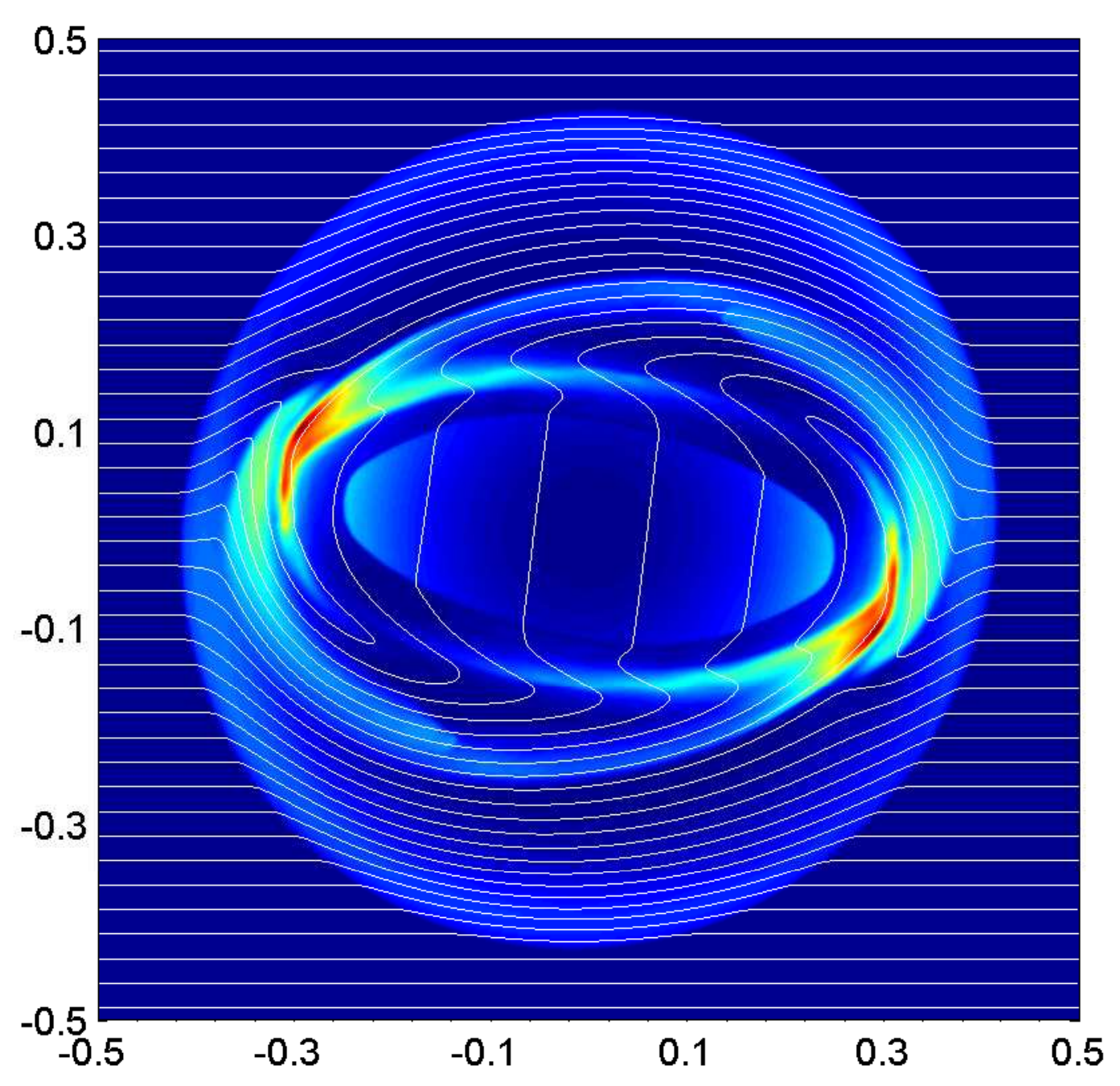}}
  \caption{\small Example \ref{example2DRo}: the contour plot of   rest-mass density (left), and the schlieren image of Lorentz factor with magnetic lines  (right) at $t=0.4$.}
  \label{fig:Ro}
\end{figure}

The parameter $\alpha$ is first taken as 9.95, which corresponds to a maximal initial Lorentz factor of about 10.01.
Fig. \ref{fig:Ro} shows the contour plot of   rest-mass density and the schlieren image of Lorentz factor at $t=0.4$ obtained by using the ${\mathbb{P}}^2$-based DG method with the PCP limiter over the uniform mesh of $400\times 400$ cells. The results agree well with those in the works \cite{HeTang2012RMHD,ZhaoTang2016}.
As one can see that winding magnetic field lines are formed and
decelerates  the disk speed. The central magnetic field
lines are rotated almost $90^\circ$ at $t=0.4$. The initial high density at the center
is swept away completely and a oblong-shaped shell is formed.
Our PCP method can work successfully for more ultra cases with larger $\alpha=9.99$ and $9.999$ (corresponding maximal initial Lorentz factors are of about 22.37 and 70.71). However, if the PCP limiting procedure is not employed, the high-order accurate DG code breaks down.
\end{example}

\begin{example}[Shock and cloud interaction problem] \label{example2DSC}\rm
This problem describes the disruption of a high density cloud by a strong shock wave. The setup is the same as that in \cite{HeTang2012RMHD}.
Different from the setup in \cite{MignoneHLLCRMHD}, the magnetic field is not orthogonal to the slab plane so that the
magnetic divergence-free treatment has to be imposed.
The computational domain is $[-0.2,1.2]\times [0,1]$, with the left boundary specified as
inflow condition and the others as outflow conditions. Initially, a shock wave moves to the right   from $x=0.05$,
with the left and right states  ${\vec V}_L = (3.86859,0.68,0,0,0,0.84981,-0.84981,1.25115)^\top$ and ${\vec V}_R = (1,0,0,0,0,0.16106,0.16106,0.05)^\top$, respectively.
There exists a rest circular cloud centred at the point $(0.25,0.5)$ with radius 0.15.  The cloud has the same states to the surrounding fluid except for a higher density 30.

\begin{figure}[htbp]
  \centering
  %{\includegraphics[width=0.48\textwidth]{2DSC560400_logRHO}}
  %{\includegraphics[width=0.48\textwidth]{2DSC560400_logPm}}
  {\includegraphics[width=0.48\textwidth]{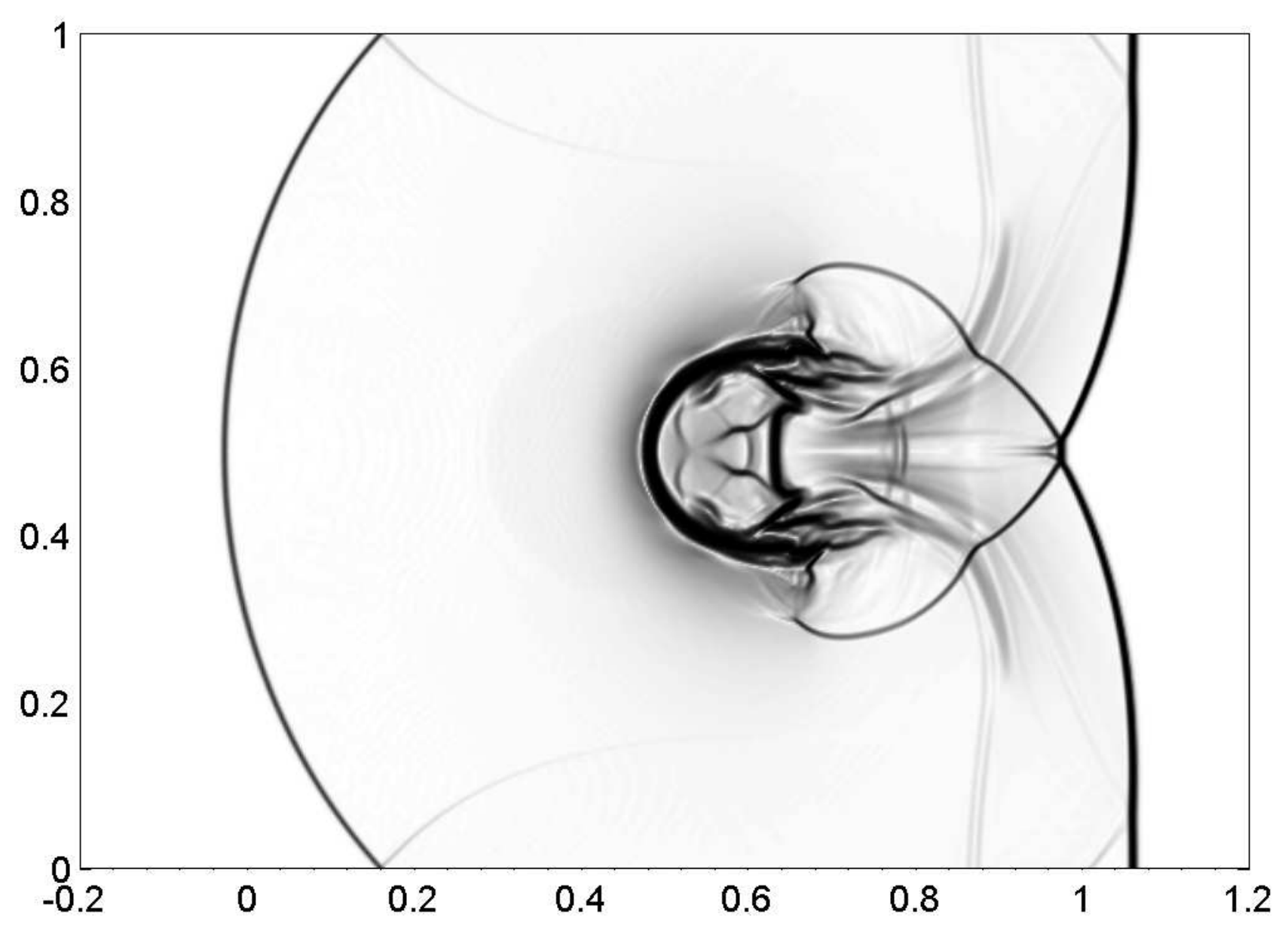}}
  {\includegraphics[width=0.48\textwidth]{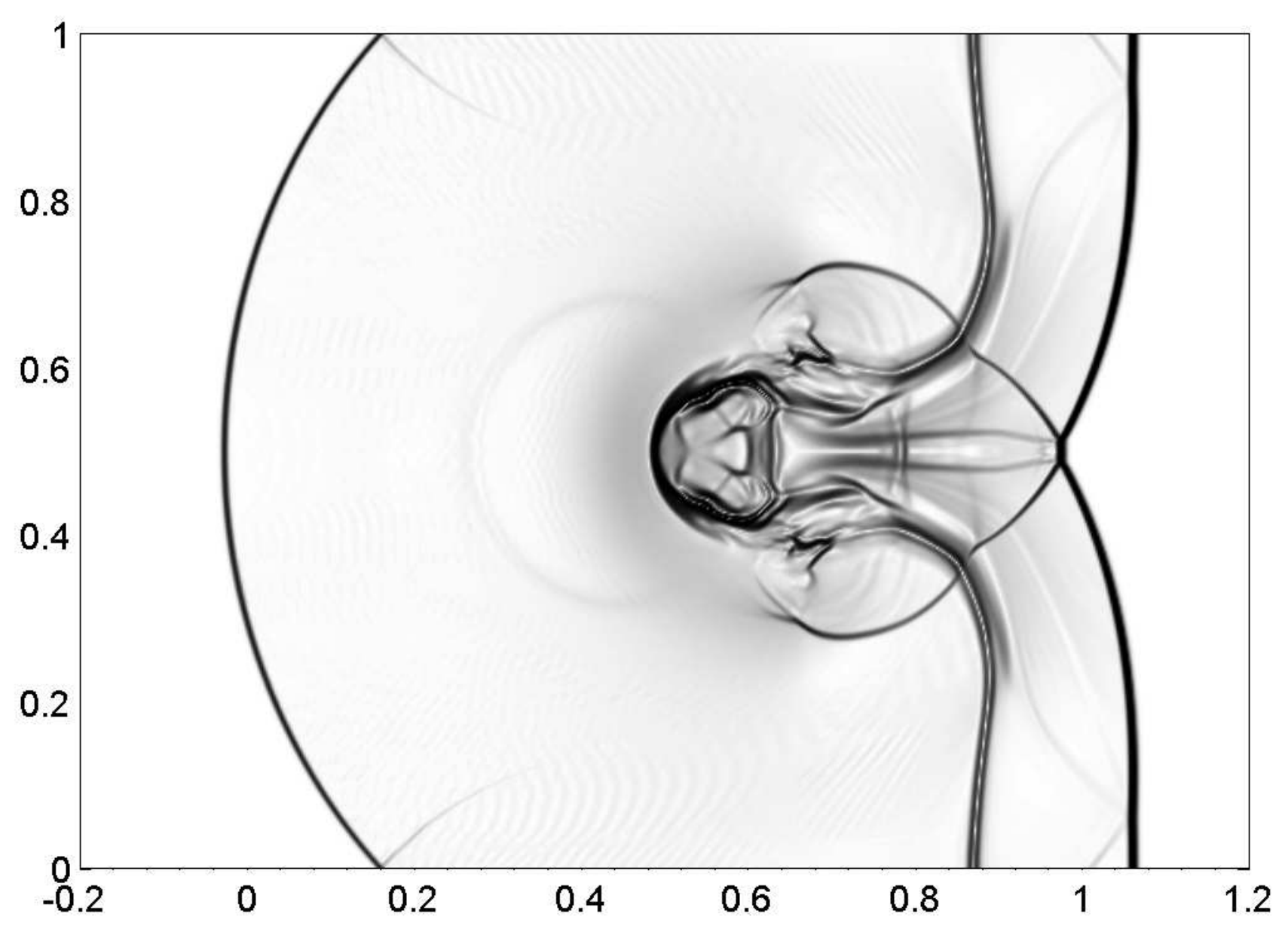}}
  \caption{\small Example \ref{example2DSC}: the schlieren images of rest-mass density logarithm (left) and magnetic pressure logarithm (right) at $t=1.2$.}
  \label{fig:SC}
\end{figure}

Fig. \ref{fig:SC} displays the schlieren images of rest-mass density logarithm $\ln \rho$ and magnetic pressure logarithm $\ln p_m$ at $t=1.2$ obtained by using the ${\mathbb{P}}^2$-based DG method with the PCP limiter over the uniform mesh of $560\times 400$ cells. One can see that the discontinuities are captured with high resolution, and the results agree well with those in \cite{HeTang2012RMHD}. In this test, it is also necessary to use the PCP limiting procedure for the successful performance of
high-order accurate DG methods. The ${\mathbb{P}}^2$-based DG method
without the PCP limiter will fail at $t\approx0.05$ due to inadmissible numerical solutions.

\end{example}

\begin{example}[Blast problems] \label{example2DBL}\rm
%The evolution of cylindrical blast wave into a magnetically dominated medium has become a standard test for 2D numerical schemes in RMHD \cite{Marti2015}.
%was first extended by Komissarov to relativistic case in \cite{GodunovRMHD}. Several variants of this test

Blast problem has become a standard test for 2D RMHD numerical schemes.
If the low gas pressure, strong magnetic field or low plasma-beta $\beta:=p/p_m$ is involved,
then simulating those ultra RMHD blast problems becomes very challenging \cite{Marti2015}.
Several different setups have been used in the literature, see e.g. \cite{GodunovRMHD,MignoneHLLCRMHD,Zanna:2007,Zanotti2015,Marti2015}.
Our setups are similar to that in \cite{MignoneHLLCRMHD,Zanna:2007,BalsaraKim2016,Zanotti2015}. Initially, the computational domain $[-6,6]^2$ is filled with a homogeneous gas at rest with adiabatic index $\Gamma=\frac43$.
The explosion zone ($r<0.8$) has a density of $10^{-2}$ and a pressure of $1$, while the ambient medium ($r>1$) has
a density of $10^{-4}$ and a pressure of $p_a=5\times 10^{-4}$, where $r=\sqrt{x^2+y^2}$.
A linear taper is applied to the density and pressure for $r\in[0.8,1]$. The magnetic field is initialized in the $x$-direction as $B_a$.
As $B_a$ is set larger, the initial ambient magnetization becomes higher ($\beta_a:=p_a/p_m$ becomes lower) and this test becomes more challenging. In the literature \cite{MignoneHLLCRMHD,Zanna:2007,BalsaraKim2016}, $B_a$ is usually specified as 0.1, which corresponds to a moderate magnetized case ($\beta_a=0.1$).
A more strongly magnetized case with $B_a=0.5$ is tested in \cite{Zanotti2015}, corresponding to a lower plasma-beta $\beta_a=4\times 10^{-3}$.
Most existing methods  in literature need some artificial treatments for the strongly magnetized case, see e.g. \cite{GodunovRMHD,MignoneHLLCRMHD}.
It is reported in \cite{Zanna:2007} that the RMHD code {\tt ECHO} is not able
to run this test with $B_a>0.1$ if no ad hoc numerical strategy is introduced.

\begin{figure}[htbp]
  \centering
  {\includegraphics[width=0.48\textwidth]{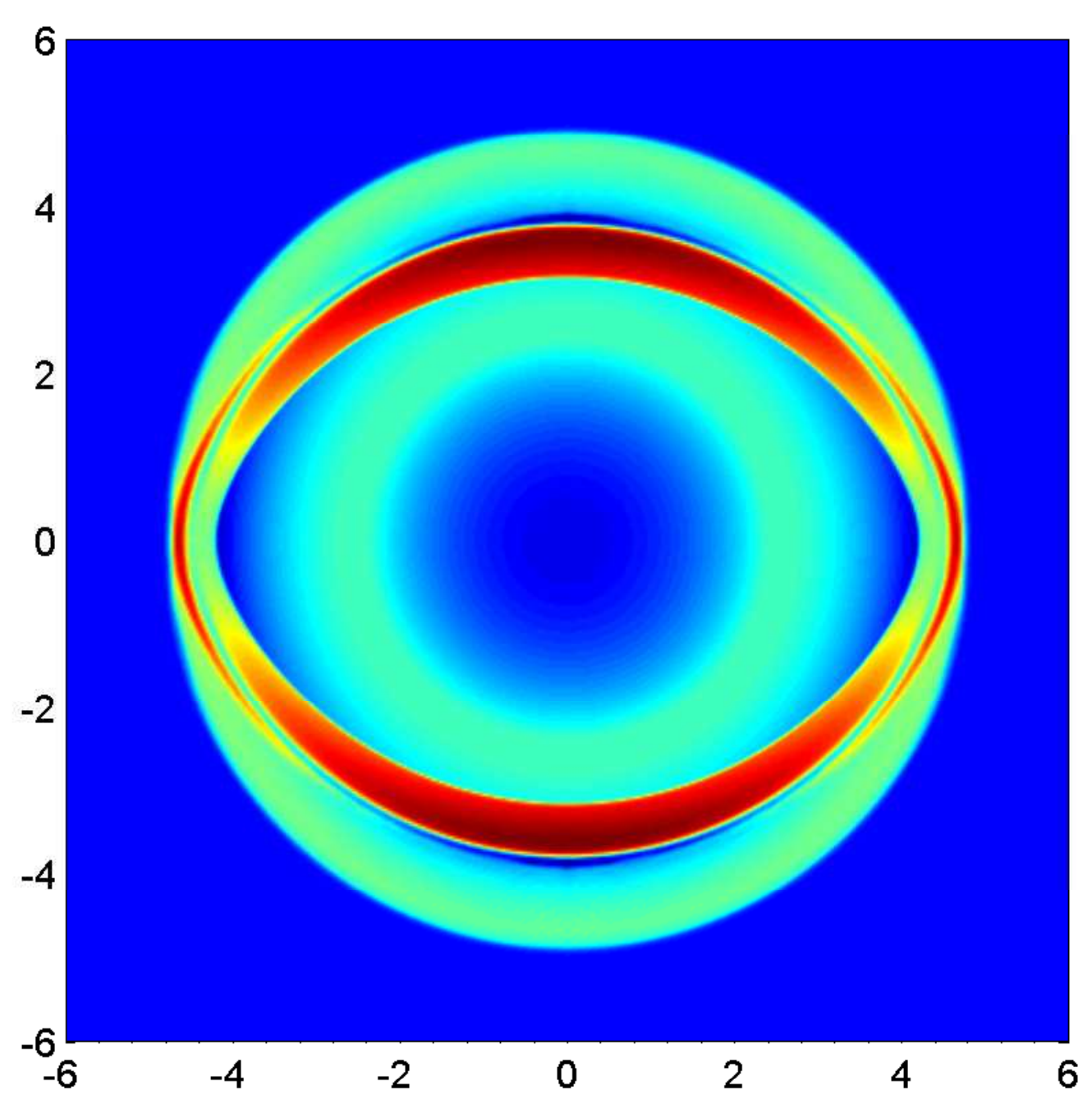}}
  {\includegraphics[width=0.48\textwidth]{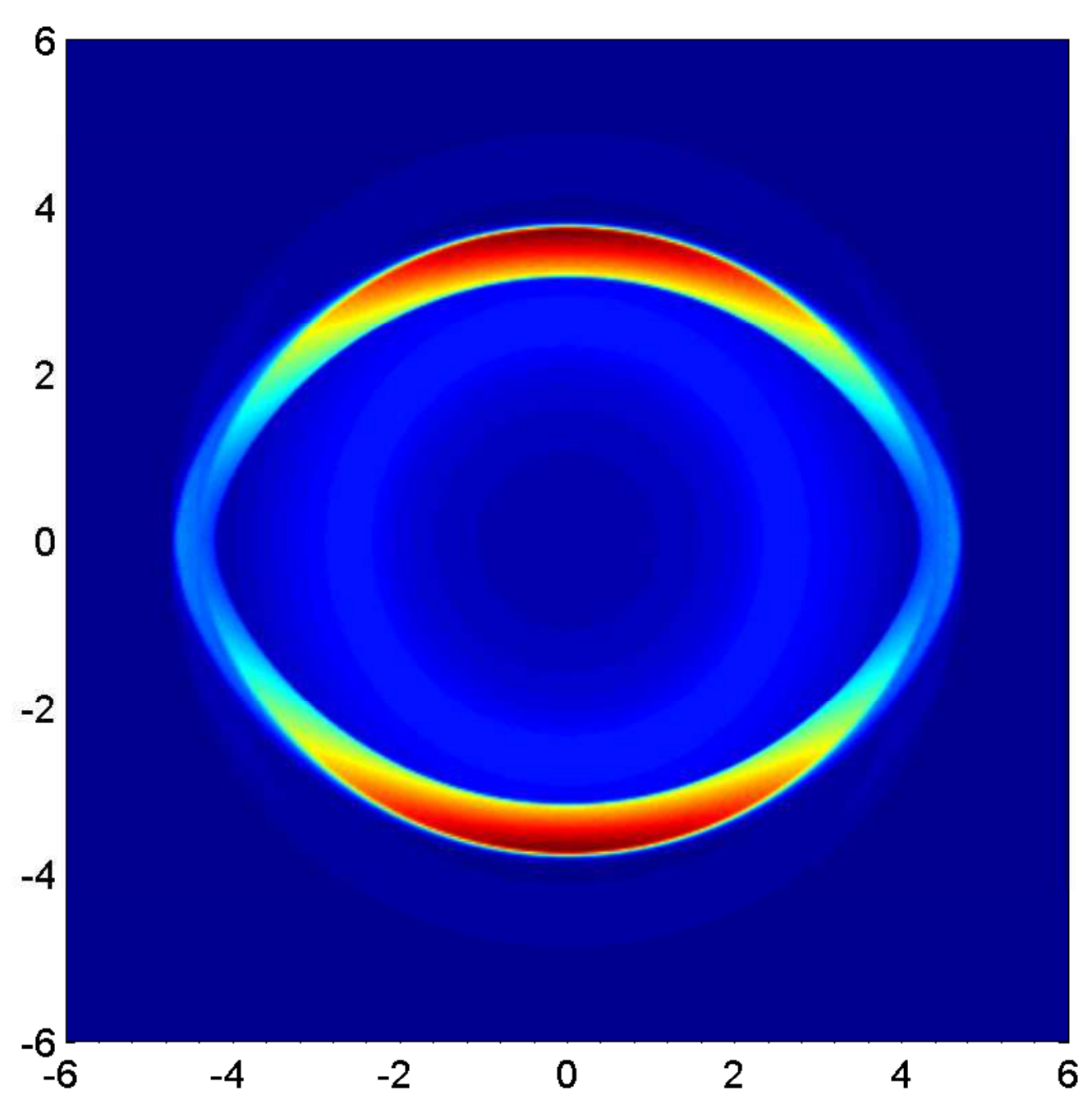}}
  {\includegraphics[width=0.48\textwidth]{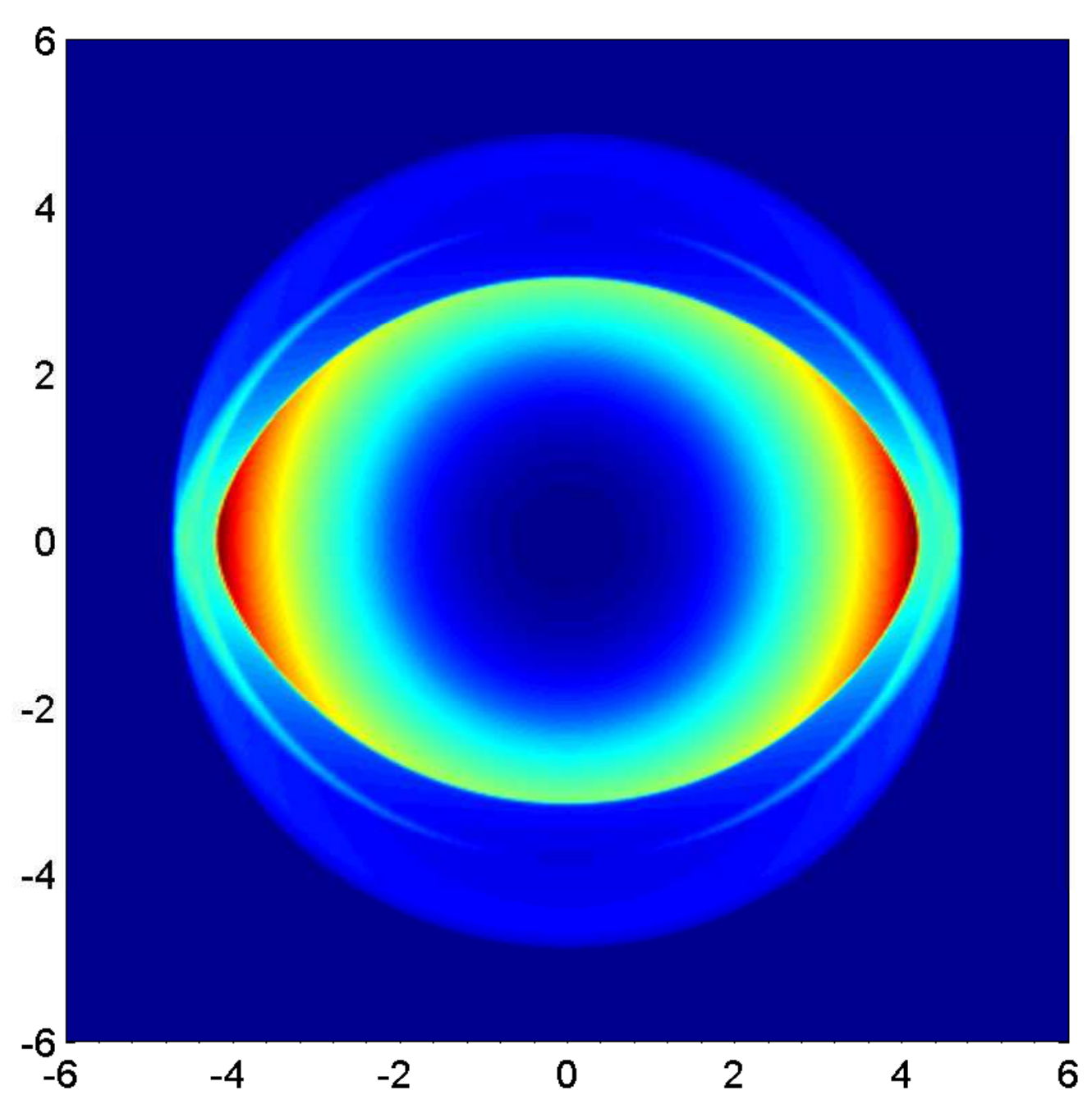}}
  {\includegraphics[width=0.48\textwidth]{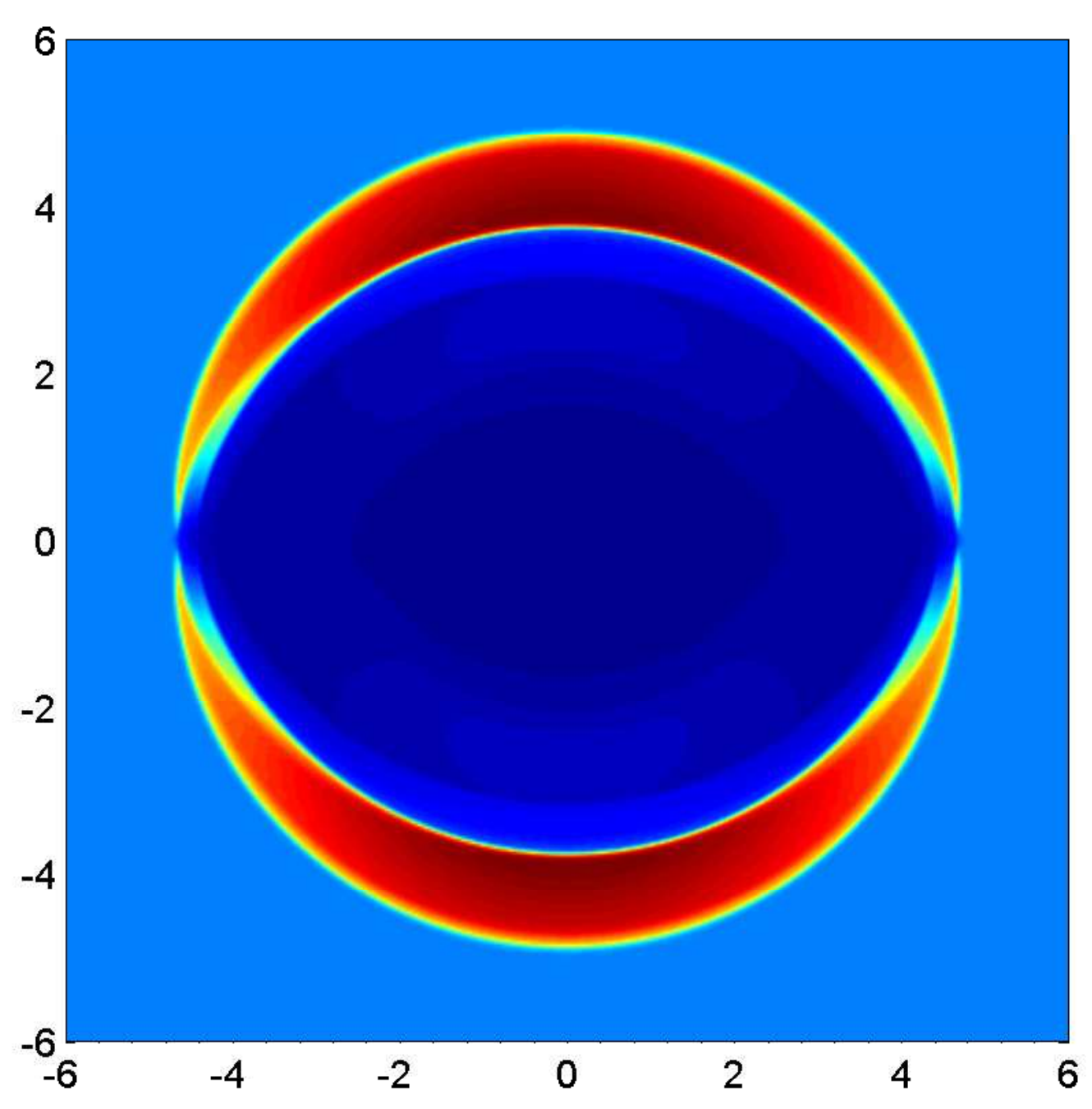}}
  \caption{\small Example \ref{example2DBL} with $B_a=0.1$: the schlieren images of rest-mass density logarithm (top-left), gas pressure (top-right),
  Lorentz factor (bottom-left) and magnetic field strength (bottom-right) at $t=4$.}
  \label{fig:BL}
\end{figure}

\begin{figure}[htbp]
  \centering
  {\includegraphics[width=0.48\textwidth]{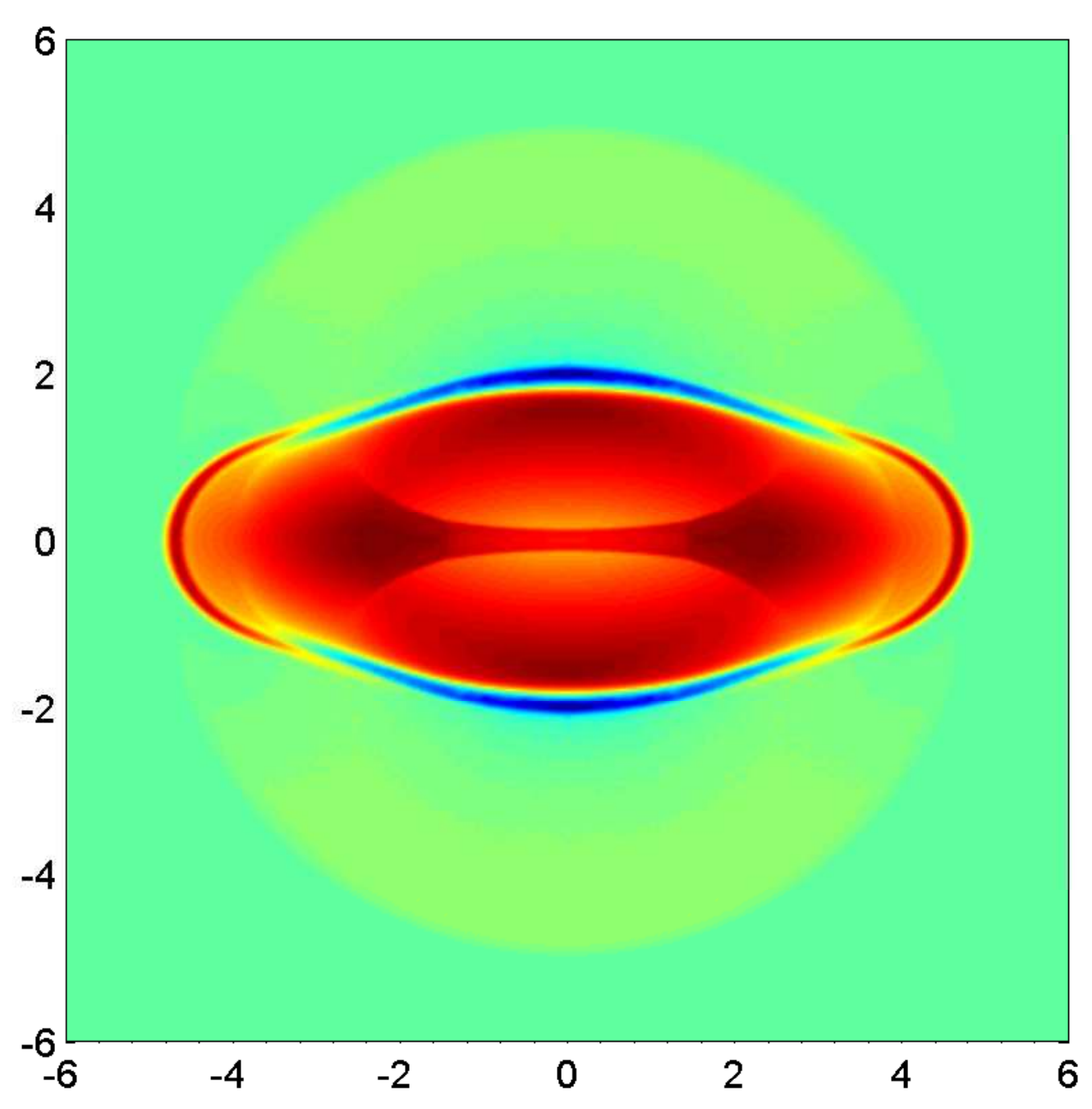}}
  {\includegraphics[width=0.48\textwidth]{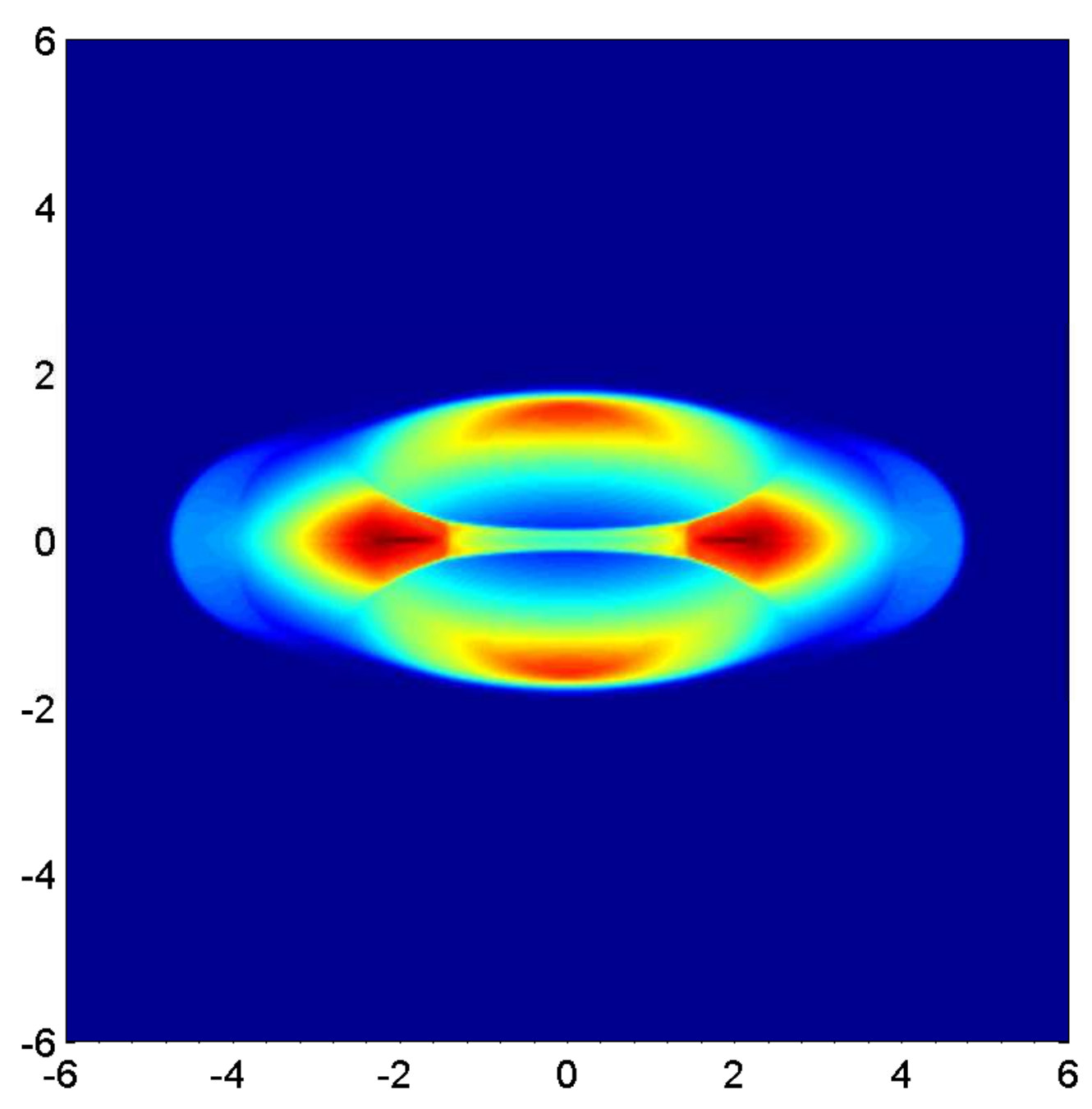}}
  {\includegraphics[width=0.48\textwidth]{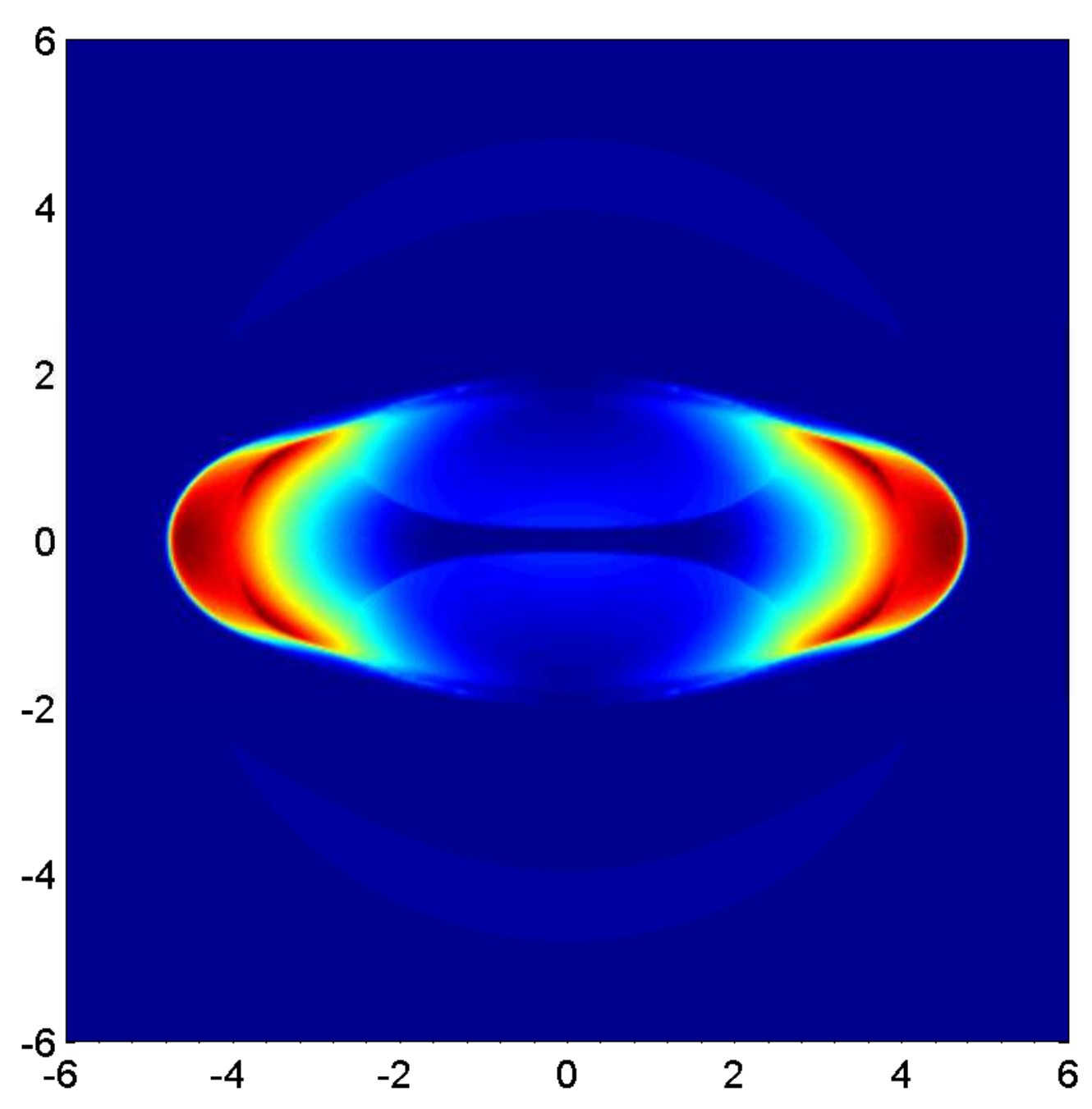}}
  {\includegraphics[width=0.48\textwidth]{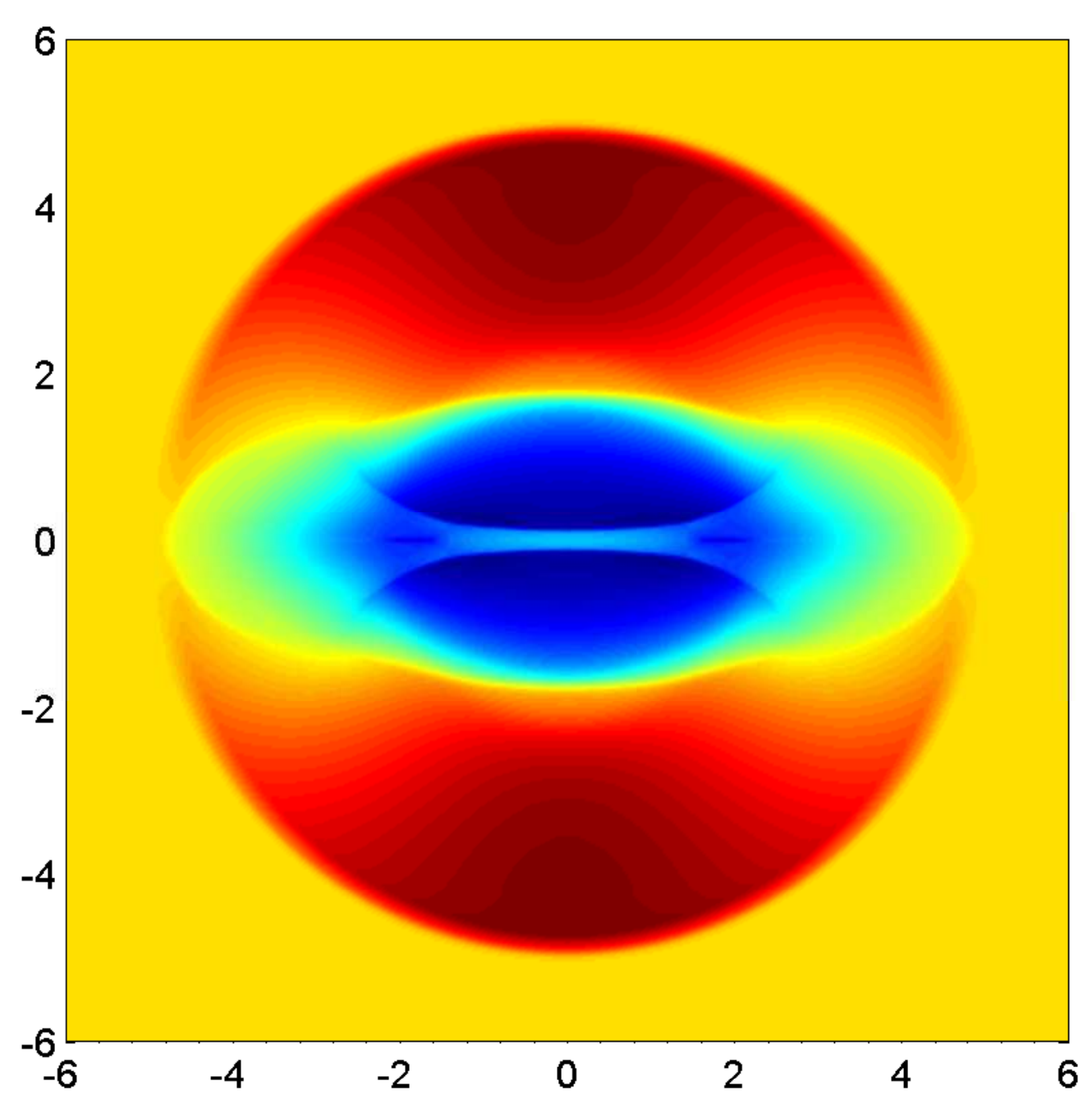}}
  \caption{\small Same as Fig. \ref{fig:BL} except for $B_a=0.5$.}
  \label{fig:BL2}
\end{figure}

\begin{figure}[htbp]
  \centering
  {\includegraphics[width=0.48\textwidth]{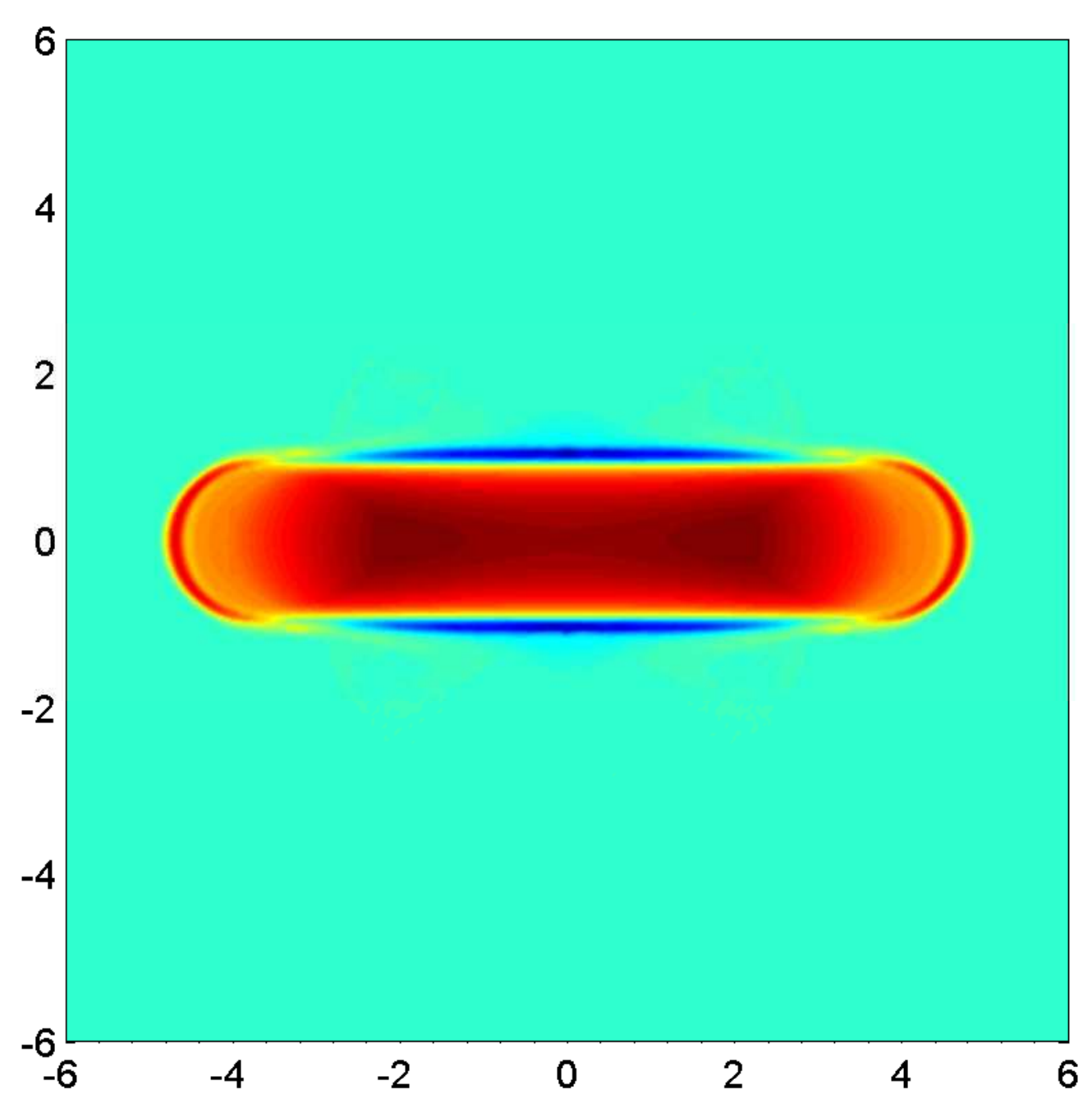}}
  {\includegraphics[width=0.48\textwidth]{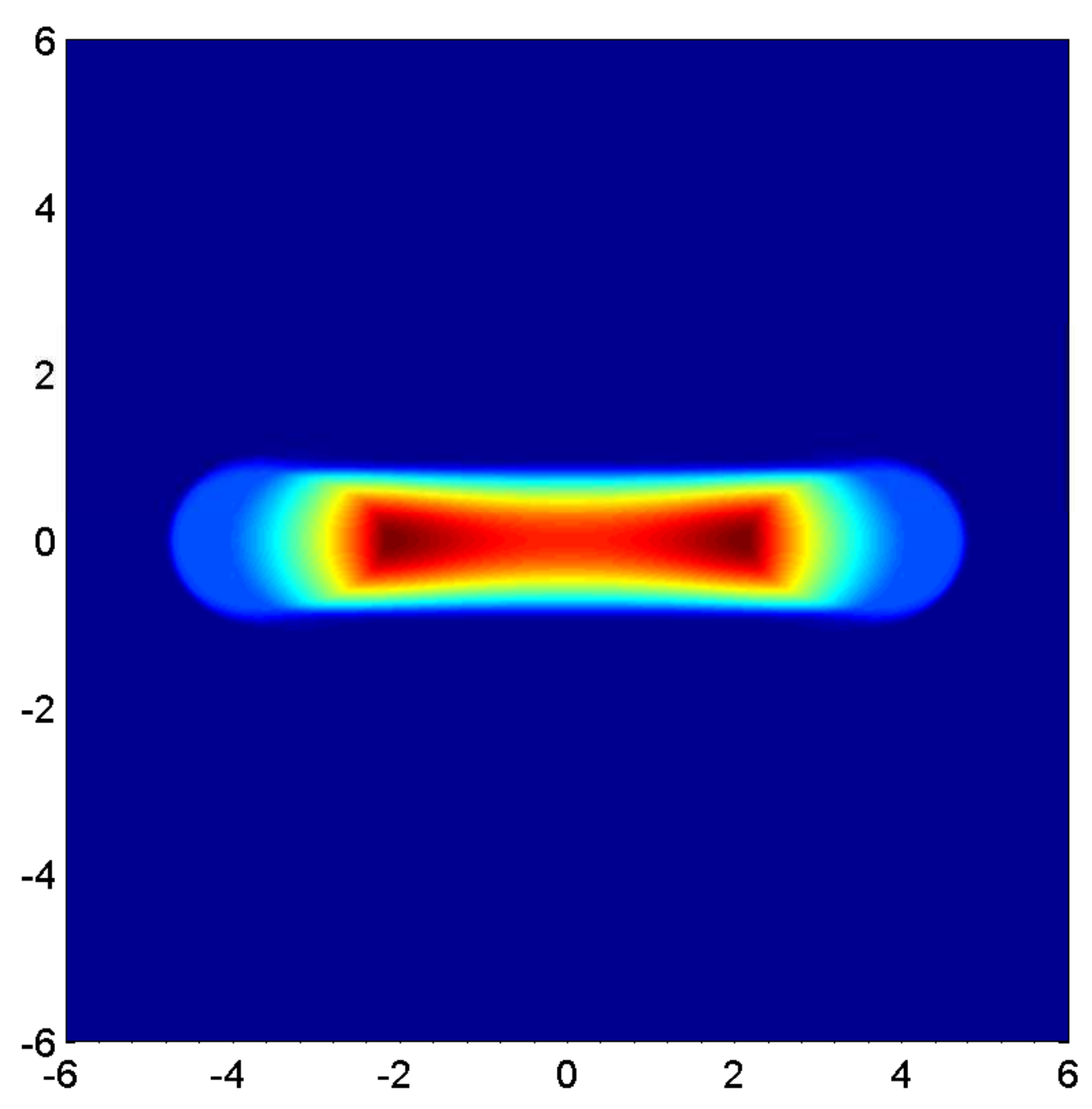}}
  {\includegraphics[width=0.48\textwidth]{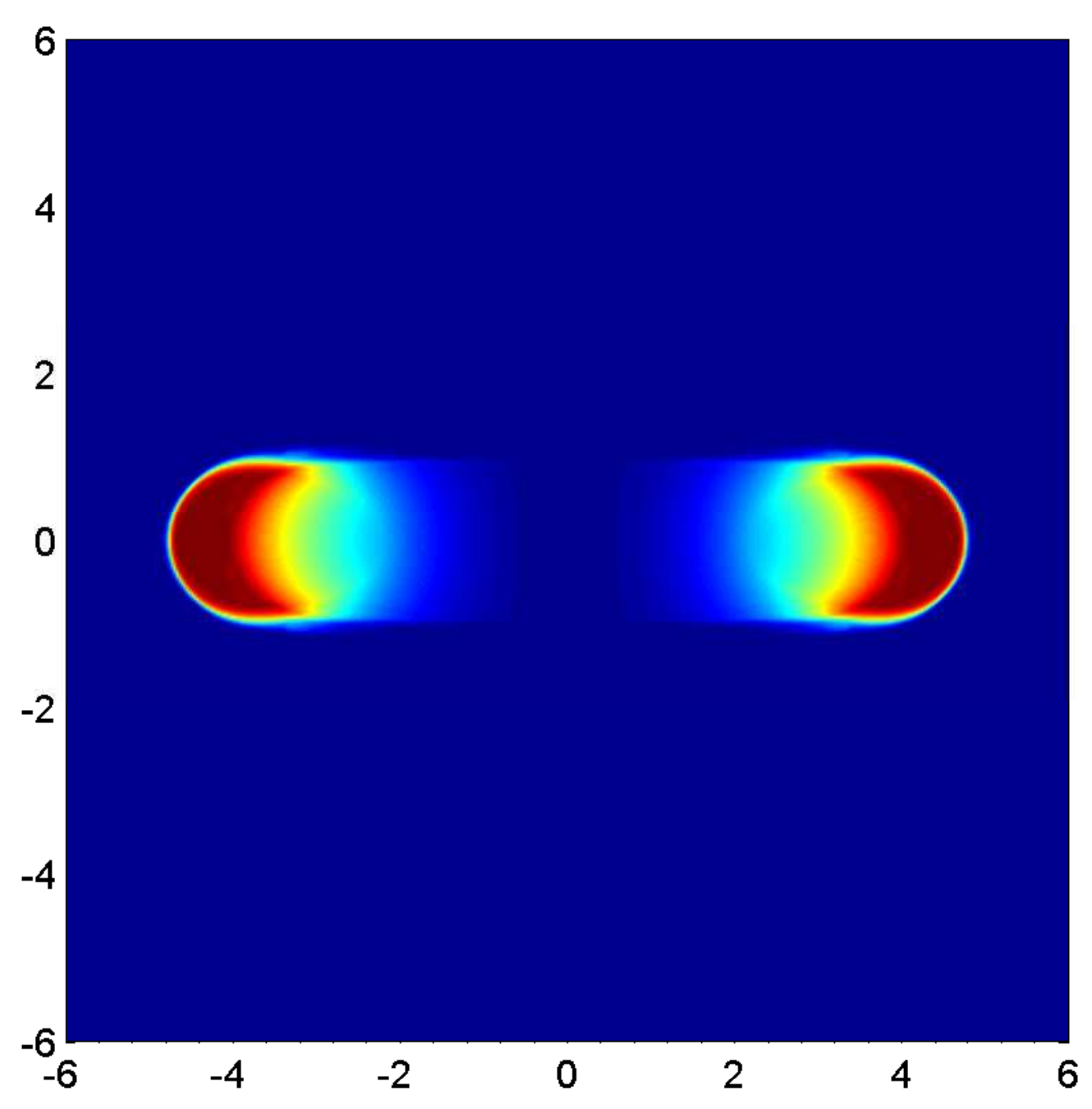}}
  {\includegraphics[width=0.48\textwidth]{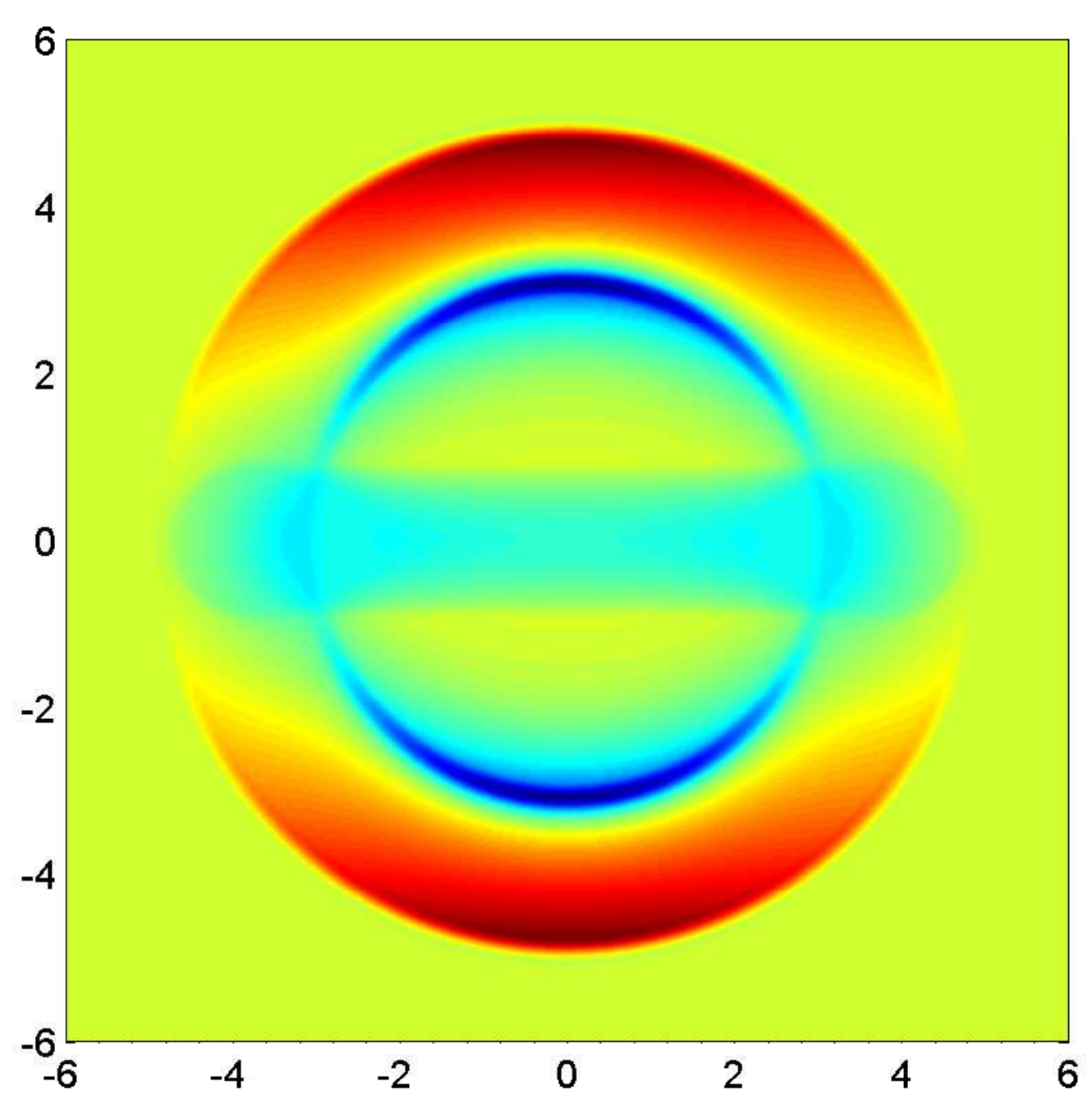}}
  \caption{\small Same as Fig. \ref{fig:BL} except for $B_a=20$.}
  \label{fig:BL3}
\end{figure}

Our numerical results of this test at $t=4$ are shown in Fig. \ref{fig:BL} for the moderately magnetized case with $B_a=0.1$,
in Fig. \ref{fig:BL2} for the relatively strongly magnetized case with $B_a=0.5$, and in Fig. \ref{fig:BL3} for the strongly magnetized case with $B_a=20$ (corresponding $\beta_a=2.5\times10^{-6}$). All of them are obtained by using the ${\mathbb{P}}^2$-based DG method with the PCP limiter over the uniform mesh of $400\times 400$ cells.
During those simulations, the present method exhibits very good robustness without any artificial treatment.
For the first two cases, our results agree quite well with those reported in \cite{Zanotti2015,BalsaraKim2016}.
From Fig. \ref{fig:BL}, it is observed that the wave pattern of the configuration is composed by two main waves,
an external fast   and a reverse shock waves. The former is almost circular, while the latter is somewhat elliptic.
The magnetic field is essentially confined between them, while the inner region is almost devoid of magnetization.
In the case of $B_a=0.5$, the external circular fast shock  is clearly visible in the rest-mass density and in the magnetic field, but very weak.
When $B_a$ is increased to 20, the external circular fast shock becomes much weaker and is only visible in the magnetic field in Fig. \ref{fig:BL3}.
As the magnetization is increased, the blast wave is confined to propagate along the magnetic field lines, creating a structure elongated in the $x$-direction.

\begin{figure}[htbp]
  \centering
  {\includegraphics[width=0.48\textwidth]{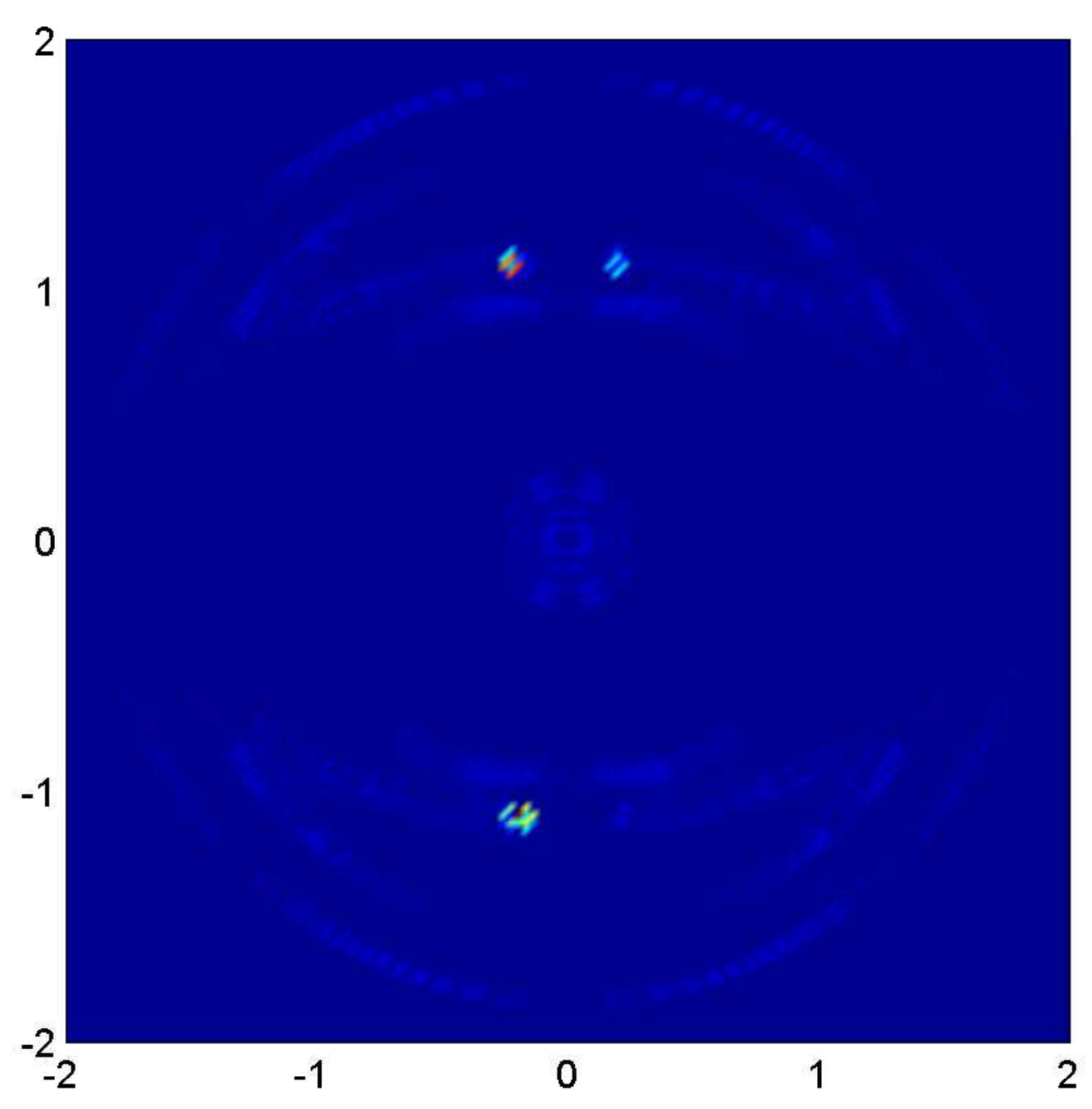}}
  \caption{\small Example \ref{example2DBL} with $B_a=100$: close-up of the schlieren image of $\left|\mbox{\rm div} _{ij}^{\mbox{\tiny \rm out}} {\vec B}\right|$ at $t=0.783$ on the uniform mesh of $400\times400$ cells.}
  \label{fig:BL4diverr}
\end{figure}

To investigate the importance of discrete divergence-free condition \eqref{eq:DivB:cst2}   in Theorem \ref{thm:PCP:2DRMHD}, we now try to test a much lower plasma-beta case $\beta_a=10^{-7}$ (i.e. $B_a=100$) on the   mesh of $400\times 400$ cells. For such extreme case, our method breaks down at $t\approx0.783$.  This failure results from the computed inadmissible cell averages of  conservative variables,   detected in the three cells centered at points $(-0.225,1.095)$, $(-0.165,1.095)$ and $(-0.165,-1.095)$, respectively.  Fig. \ref{fig:BL4diverr} displays the schlieren image of
$\left|\mbox{\rm div} _{ij}^{\mbox{\tiny \rm out}} {\vec B}\right|$ at the moment of failure.
It clearly shows the subregions with large values of $\left|\mbox{\rm div} _{ij}^{\mbox{\tiny \rm out}} {\vec B}\right|$, where the condition \eqref{eq:DivB:cst2} is violated most seriously,
and the three detected cells are exactly located in those  subregions.
This further demonstrates that the condition \eqref{eq:DivB:cst2} is really crucial in achieving
completely PCP schemes in 2D.
As mentioned in Remark \ref{remark:diveonpcp}, for the purpose of numerical simulation,
it is possible to weaken the impact of  violating \eqref{eq:DivB:cst2} by refining the mesh.
By numerical experiments, we find that our method can work successfully on a refined mesh of $600 \times 600$ cells for the case of $B_a=100$ and  a extremely strongly magnetized case with $B_a=1500$ ($\beta_a \approx 4.444 \times 10^{-10}$). The flow structures in those two cases are similar to the case of $B_a=20$ and omitted here. To our best knowledge, the 2D blast test with so low plasma-beta is rarely considered in the literature.

%\begin{figure}[htbp]
%  \centering
%  {\includegraphics[width=0.48\textwidth]{600BL4_logRHO}}
%  {\includegraphics[width=0.48\textwidth]{600BL4_P}}
%  {\includegraphics[width=0.48\textwidth]{600BL4_W}}
%  {\includegraphics[width=0.48\textwidth]{600BL4_B}}
%  \caption{\small Same as Fig. \ref{fig:BL} except for $B_a=100$.}
%  \label{fig:BL4}
%\end{figure}
%
%\begin{figure}[htbp]
%  \centering
%  {\includegraphics[width=0.48\textwidth]{600BL5_logRHO}}
%  {\includegraphics[width=0.48\textwidth]{600BL5_P}}
%  {\includegraphics[width=0.48\textwidth]{600BL5_W}}
%  {\includegraphics[width=0.48\textwidth]{600BL5_B}}
%  \caption{\small Same as Fig. \ref{fig:BL} except for $B_a=1500$.}
%  \label{fig:BL5}
%\end{figure}

\end{example}

\section{Conclusions}\label{sec:con}

The paper studied mathematical properties of the admissible state set $\mathcal G$
defined in \eqref{eq:RMHD:definitionG} of the relativistic magnetohydrodynamical (RMHD) equations \eqref{eq:RMHD1D}.
In comparison with the non-relativistic and relativistic hydrodynamical cases (with the zero magnetic field),  the difficulties
{mainly came} from the extremely strong nonlinearities,  no explicit formulas of the primitive variables and the flux vectors with respect to the conservative variables, and the solenoidal magnetic field.
%the zero divergence of a magnetic field.
To overcome those difficulties,
 the first equivalent form of  $\mathcal G$ with {explicit} constraints on the conservative vector was first
 skillfully discovered with the aid of polynomial root properties,  and
 followed by the scaling invariance.
 The convexity of $\mathcal G$ was proved by utilizing the semi-positive definiteness of the second fundamental form of a hypersurface, and then the second equivalent form of  $\mathcal G$
and  the orthogonal invariance were obtained.
It was verified that the Lax-Friedrichs (LxF) splitting property did not {hold} in general when the magnetic field
was nonzero. However,  by combining
the  convex combination of some  LxF  splitting terms with a ``discrete divergence-free'' condition for the magnetic
field, the  generalized    LxF  splitting properties were subtly discovered with a constructive inequality and some pivotal techniques.
This revealed in theory for the first time  the close connection between the ``discrete divergence-free'' condition and the physical-constraints-preserving (PCP) property of numerical schemes.

The above mathematical properties were footstone of studying
PCP numerical schemes for RMHDs.
Based on the resulting theoretical results,
several 1D and 2D PCP schemes were studied for the first time.
%and the divergence of a magnetic field being zero,
In the 1D case, a first-order accurate LxF   type scheme was
first proved to be PCP under the CFL condition. Then,
the high-order accurate 1D PCP schemes were proposed via a PCP limiter,
which was designed by using the {first} equivalent form of $\mathcal G$.
In the 2D case,
 the  ``discrete divergence-free'' condition and  PCP property were analyzed for {a first-order accurate LxF type} scheme, and
 followed by two sufficient conditions for high-order accurate PCP schemes.
 Several numerical experiments were conducted to demonstrate the theoretical analyses and the performance of  numerical schemes as well as
the importance of discrete divergence-free condition in achieving genuinely PCP scheme in 2D.
The studies on the PCP schemes may be easily extended to the three-dimensional case
by Theorem \ref{theo:RMHD:LLFsplit3D},
 the non-uniform or unstructured meshes by Theorem  \ref{theo:RMHD:LLFsplit2Dus},
and the general EOS case by the similar discussions in \cite{WuTang2016}.

\section*{Acknowledgements}

This work was partially supported by
the National Natural Science Foundation
of China (Nos.  91330205 \& 11421101).
%{\color{red}The authors thank Professor Dinshaw S. Balsara at University of Notre Dame for valuable suggestions on improving the manuscript.}

\end{spacing}


\begin{thebibliography}{00}



\bibitem{Anderson:2006}
M.~Anderson, E.W. Hirschmann, S.L. Liebling and D.~Neilsen,
Relativistic {MHD} with adaptive mesh refinement,
{\it Class. Quantum Grav.} {\bf 23} (2006) 6503--6524.

\bibitem{Anton2010}
L. Ant{\'o}n, J.A. Miralles, J.M. Mart{\'\i}, J.M. Ib{\'a}{\~n}ez, M.A. Aloy and P. Mimica,
Relativistic magnetohydrodynamics: renormalized eigenvectors and full wave decomposition Riemann solver,
{\it Astrophys. J. Suppl. Ser.} {\bf 188} (2010) 1--31.


\bibitem{Balsara2001}
D.S. Balsara,
Total variation diminishing scheme for relativistic magnetohydrodynamics,
{\it Astrophys. J. Suppl. Ser.} {\bf 132} (2001) 83--101.


\bibitem{Balsara2004}
D.S. Balsara,
Second-order-accurate schemes for magnetohydrodynamics with divergence-free reconstruction,
{\it Astrophys. J. Suppl. Ser.} {\bf 151} (2004) 149--184.


\bibitem{Balsara2009}
D.S. Balsara,
Divergence-free reconstruction of magnetic fields and WENO schemes for magnetohydrodynamics,
{\it J. Comput. Phys.} {\bf 228} (2009) 5040--5056.


\bibitem{Balsara2012}
D.S. Balsara,
Self-adjusting, positivity preserving high order schemes for hydrodynamics and magnetohydrodynamics,
{\it J. Comput. Phys.} {\bf 231} (2012) 7504--7517.



\bibitem{BalsaraKim2016}
D.S. Balsara and J. Kim,
A subluminal relativistic magnetohydrodynamics scheme with ADER-WENO predictor and multidimensional Riemann solver-based corrector,
{\it J. Comput. Phys.} {\bf 312} (2016) 357--384.

\bibitem{Borges2008}
R. Borges, M. Carmona, B. Costa and W. S. Don,
An improved weighted essentially nonoscillatory
scheme for hyperbolic conservation laws, {\it J. Comput. Phys.} {\bf 227} (2008) 3101--3211.




\bibitem{Brackbill1980}
J.U. Brackbill and D.C. Barnes, The effect of nonzero $\nabla \cdot \vec B $ on
the numerical solution of the magnetodyndrodynamic equations,
{\it J. Comput. Phys.} {\bf 35} (1980) 426--430.


\bibitem{cheng}
Y. Cheng, F.Y. Li, J.X. Qiu and L.W. Xu,
Positivity-preserving DG and central DG methods for ideal MHD equations,
{\it J. Comput. Phys.} {\bf 238} (2013) 255--280.


\bibitem{Christlieb}
A. J. Christlieb, Y. Liu, Q. Tang and Z.F. Xu,
Positivity-preserving finite difference weighted ENO schemes with constrained transport for ideal magnetohydrodynamic equations,
{\it SIAM J. Sci. Comput.} {\bf 37} (2015) A1825--A1845.




\bibitem{Cockburn0}
B. Cockburn, S.C. Hu and C.-W. Shu,
The Runge-Kutta local projection
discontinuous Galerkin finite element method for conservation laws IV: the
multidimensional case,
{\it Math. Comp.} {\bf 54} (1990) 545--581.



\bibitem{Cissoko1992}
M. Cissoko,
Detonation waves in relativistic hydrotlynamics,
{\it Phys. Rev. D} {\bf 45} (1992) 1045--1052.



\bibitem{Zanna:2003}
L. Del Zanna, N. Bucciantini and P. Londrillo,
An efficient shock-capturing central-type scheme for multidimensional relativistic flows II. Magnetohydrodynamics,
{\it Astron. Astrophys.} {\bf 400} (2003) 397--413.


\bibitem{Zanna:2007}
L. Del Zanna, O. Zanotti, N. Bucciantini and P. Londrillo,
ECHO: a Eulerian conservative high-order scheme for general relativistic magnetohydrodynamics and magnetodynamics,
{\it Astron. Astrophys.} {\bf 473} (2007) 11--30.



\bibitem{Evans1988}
C.R. Evans and J.F. Hawley, Simulation of magnetodydrodynamic flows: A constrained transport method,
{\it Astrophys. J.} {\bf 332} (1988) 659--677.


\bibitem{font2008}
J.A. Font, Numerical hydrodynamics and magnetohydrodynamics in general relativity, {\it Living Rev. Relativity} {\bf 11} (2008) 7.

\bibitem{Friedrichs1974}
K.O. Friedrichs, On the laws of relativistic electro-magneto-fluid dynamics, {\it Comm. Pure Appl. Math.} {\bf 27} (1974) 749--808.



\bibitem{Giacomazzo2006}
B. Giacomazzo and L. Rezzolla,
The exact solution of the Riemann problem in
relativistic magnetohydrodynamics, {\it J. Fluid Mech.} {\bf 562} (2006) 223--259.


\bibitem{Gottlieb2009}
S. Gottlieb, D.J. Ketcheson and C.-W. Shu,
High order strong stability
preserving time discretizations,
{\it J. Sci. Comput.} {\bf 38} (2009) 251--289.



\bibitem{HeTang2012RMHD}
P. He and H.Z. Tang,
An adaptive moving mesh method for two-dimensional relativistic magnetohydrodynamics,
{\it Computers \& Fluids} {\bf 60} (2012) 1--20.






\bibitem{Honkkila:2007}
V.~Honkkila and P.~Janhunen,
{HLLC} solver for ideal relativistic {MHD},
{\it J. Comput. Phys.} {\bf 223} (2007) 643--656.

\bibitem{Host:2008}
B.~van~der Holst, R.~Keppens, and Z.~Meliani,
A multidimensional grid-adaptive relativistic magnetofluid code,
{\it Comput. Phys. Comm.} {\bf 179} (2008) 617--627.



\bibitem{Jonker1975}
L.D. Jonker,
Immersions with semi-definite second fundamental forms,
{\it Canad. J. Math.} {\bf 27} (1975) 610--617.
%http://cms.math.ca/openaccess/cjm/v27/cjm1975v27.0610-0617.pdf


\bibitem{Kim2014}
J. Kim and D.S. Balsara,
A stable HLLC Riemann solver for relativistic magnetohydrodynamics,
{\it J. Comput. Phys.} {\bf 270} (2014) 634--639.

\bibitem{Krivodonova}
L. Krivodonova, J. Xin, J.-F. Remacle, N. Chevaugeon and J.E. Flaherty,
Shock detection and limiting with discontinuous Galerkin methods for hyperbolic conservation laws,
{\it Appl. Numer. Math.} {\bf 48} (2004) 323--338.


\bibitem{GodunovRMHD}
S.S. Komissarov,
A {Godunov-type} scheme for relativistic magnetohydrodynamics,
{\it Mon. Not. R. Astron. Soc.} {\bf 303} (1999) 343--366.




\bibitem{Li2005}
F.Y. Li and C.-W. Shu,
Locally divergence-free discontinuous Galerkin methods for MHD equations,
{\it J. Sci. Comput.} {\bf 22} (2005) 413--442.

%\bibitem{Li2012}
%{\color{red}F. Li and L. Xu,
%Arbitrary order exactly divergence-free central discontinuous Galerkin methods for ideal MHD equations,
%{\it J. Comput. Phys.} {\bf 231} (2012) 2655--2675.}

\bibitem{Li2011}
F.Y. Li, L.W. Xu and S. Yakovlev,
Central discontinuous Galerkin methods for ideal MHD equations with the exactly divergence-free magnetic field,
{\it J. Comput. Phys.} {\bf 230} (2011) 4828--4847.


\bibitem{Marti2015}
J.M. Mart{\'\i} and E.~M{\"u}ller,
Grid-based methods in relativistic hydrodynamics and magnetohydrodynamics,
{\em Living Rev. Comput. Astrophys.} {\bf 1} (2015) 3.



\bibitem{MignoneHLLCRMHD}
{A. Mignone and G. Bodo},
{An {HLLC} {Riemann} solver for relativistic flows--{II}. Magnetohydrodynamics},
{\it Mon. Not. R. Astron. Soc.} {\bf 368} (2006) 1040--1054.


\bibitem{Miyoshi}
T. Miyoshi and K. Kusano,
A multi-state HLL approximate Riemann solver for ideal magnetohydrodynamics, {\it J. Comput. Phys.} {\bf 208} (2005) 315--344.





\bibitem{Newman}
W.I. Newman and N. D. Hamlin,
Primitive variable determination in conservative relativistic magnetohydrodynamic simulations,
{\it SIAM J. Sci. Comput.} {\bf 36} (2014) B661--B683.

\bibitem{Noble}
S.C. Noble, C.F. Gammie, J.C. McKinney and L.D. Zanna,
Primitive variable solvers for conservative general relativistic magnetohydrodynamics,
{\it Astrophys. J., Suppl. Ser.} {\bf 641} (2006) 626--637.


\bibitem{Qamar2005}
S. Qamar and G. Warnecke,
A high-order kinetic flux-splitting method for the relativistic magnetohydrodynamics,
{\it J. Comput. Phys.} {\bf 205} (2005) 182--204.


\bibitem{Qiu2005}
J. Qiu and C.-W. Shu, Runge-Kutta discontinuous Galerkin method using WENO limiters, {\it SIAM J. Sci. Comput.} {\bf 26} (2005) 907--929.


\bibitem{Rossmanith2006}
J.A. Rossmanith, An unstaggered, high-resolution constrained transport method for magnetohydrodynamic flows,
{\it SIAM J. Sci. Comput.} {\bf 28} (2006) 1766--1797.




\bibitem{Toth2000}
G. T\'oth, The $\nabla \cdot \vec B = 0$ constraint in shock-capturing magnetohydrodynamics codes,
{\it J. Comput. Phys.} {\bf 161} (2000) 605--652.




%\bibitem{wukl-thesis}
%K.L. Wu, {\it Mathematical Properties and numerical methods for relativistic hydrodynamics}, Ph.D. thesis (School of Mathematical Sciences, Peking University, 2016).



\bibitem{WuTang2015}
K.L. Wu and H.Z. Tang,
High-order accurate physical-constraints-preserving finite difference WENO schemes for special relativistic hydrodynamics,
{\it J. Comput. Phys.} {\bf 298} (2015) 539--564.

\bibitem{WuTang2016}
K.L. Wu and H.Z. Tang,
Physical-constraints-preserving central discontinuous Galerkin methods for special relativistic hydrodynamics with a general equation of state,
accepted by {\it Astrophys. J. Suppl. Ser.},  2016.
%arXiv:1607.08332.




\bibitem{Xing2010}
Y. Xing, X. Zhang and C.-W. Shu, Positivity-preserving high order well-balanced discontinuous Galerkin methods for the shallow water equations,
{\it Adv. Water Res.} {\bf 33} (2010) 1476--1493.


\bibitem{Yang2016}
H. Yang and F.Y. Li,
Stability analysis and error estimates of an exactly divergence-free method for the magnetic induction equations,
{\em ESAIM: Math. Model. Numer. Anal.} {\bf 50} (2016) 965--993.

\bibitem{Zanotti2015}
O. Zanotti, F. Fambri and M. Dumbser,
Solving the relativistic magnetohydrodynamics equations with ADER discontinuous Galerkin methods, a posteriori subcell limiting and adaptive mesh refinement,
{\it Mon. Not. R. Astron. Soc.} {\bf 452} (2015) 3010--3029.


\bibitem{zhang2010}
X. Zhang and C.-W. Shu, On maximum-principle-satisfying high order schemes for scalar conservation laws,
{\it J. Comput. Phys.} {\bf 229} (2010) 3091--3120.


\bibitem{zhang2010b}
X. Zhang and C.-W. Shu, On positivity-preserving high-order discontinuous {Galerkin} schemes for compressible Euler equations on rectangular meshes,
{\it J. Comput. Phys.} {\bf 229} (2010) 8918--8934.


\bibitem{zhang2011b}
X.X. Zhang and C.-W. Shu, Maximum-principle-satisfying and positivity-preserving high-order schemes for conservation laws: survey and new developments, {\em Proc. R. Soc. A} {\bf 467} (2011) 2752--2776.

%\bibitem{ZhaoThesis}
%J. Zhao, {\it RKDG Methods for Relativistic Hydrodynamics and Magnetohydrodynamics}, Ph.D. thesis (School of Mathematical Sciences, Peking University, 2014).

\bibitem{ZhaoTang2016}
J. Zhao and H.Z. Tang, Runge-Kutta discontinuous Galerkin methods for the special relativistic magnetohydrodynamics, arXiv:1610.03404, 2016.


\end{thebibliography}
\end{document}